\documentclass{amsart}
\usepackage{amsthm,amssymb,amsfonts,amsmath,color,bm}
\usepackage{graphicx,enumerate}
\usepackage{pdfsync,subfigure}
\usepackage{amsfonts,amsmath,amssymb,epsfig,mathdots,tcolorbox,bm}
\usepackage{type1cm}

\usepackage{multicol}        
\usepackage[bottom]{footmisc}
\newcommand{\1}{{\bf 1}}

\newcommand{\clt}{central limit theorem}

\newcommand{\ex}{{\rm e}\,}

\newcommand{\asy}{asymptotic}

\newcommand{\Imath}{I}
\newcommand{\Jmath}{J}
\newcommand{\ts}{time series}
\newcommand{\tsa}{\ts\ analysis}

\definecolor{darkblue}{rgb}{.1, 0.1,.8}
\definecolor{darkgreen}{rgb}{0,0.8,0.2}
\definecolor{darkred}{rgb}{.8, .1,.1}

\newtheorem{lemma}{Lemma}[section]

\newtheorem{theorem}[lemma]{Theorem}

\newtheorem{proposition}[lemma]{Proposition}
\newtheorem{definition}[lemma]{Definition}
\newtheorem{corollary}[lemma]{Corollary}
\newtheorem{example}[lemma]{Example}
\newtheorem{exercise}[lemma]{Exercise}
\newtheorem{remark}[lemma]{Remark}
\newtheorem{fig}[lemma]{Figure}
\newtheorem{tab}[lemma]{Table}

\newcommand{\bfu}{{\bf u}}

\newcommand{\bfU}{{\bf U}}

\newcommand{\bfC}{{\bf C}}

\newcommand{\bth}{\begin{theorem}}
\newcommand{\ethe}{\end{theorem}}

\newcommand{\bre}{\begin{remark}\em }
\newcommand{\ere}{\end{remark}}
\newcommand{\fre}{frequenc}
\newcommand{\ble}{\begin{lemma}}
\newcommand{\ele}{\end{lemma}}

\newcommand{\bde}{\begin{definition}}
\newcommand{\ede}{\end{definition}}
\newcommand{\bco}{\begin{corollary}}
\newcommand{\eco}{\end{corollary}}
\renewcommand{\a}{\alpha}
\newcommand{\bpr}{\begin{proposition}}
\newcommand{\epr}{\end{proposition}}

\newcommand{\bexer}{\begin{exercise}}
\newcommand{\eexer}{\end{exercise}}

\newcommand{\bexam}{\begin{example}}
\newcommand{\eexam}{\end{example}}
\newcommand{\bfl}{{\bf l}}
\newcommand{\bfi}{\begin{fig}}
\newcommand{\efi}{\end{fig}}

\newcommand{\btab}{\begin{tab}}
\newcommand{\etab}{\end{tab}}
\newcommand{\bfm}{{\bf m}}

\newcommand{\per}{periodogram}

\newcommand{\fidi}{finite-dimensional distribution}
\newcommand{\rv}{random variable}

\newcommand{\var}{{\rm var}}

\newcommand{\cov}{{\rm cov}}
\newcommand{\corr}{{\rm corr}}

\renewcommand{\P}{{\mathbb P}}
\newcommand{\E}{{\mathbb E}}

\newcommand{\bfTh}{\mbox{\boldmath$\Theta$}}
\newcommand{\bfomega}{\mbox{\boldmath$\w$}}
\newcommand{\bfla}{\mbox{\boldmath$\lambda$}}

\newcommand{\rhs}{right-hand side}
\newcommand{\df}{distribution function}

\newcommand{\beao}{\begin{eqnarray*}}
\newcommand{\eeao}{\end{eqnarray*}\noindent}

\newcommand{\beam}{\begin{eqnarray}}
\newcommand{\eeam}{\end{eqnarray}\noindent}

\newcommand{\beqq}{\begin{equation}}
\newcommand{\eeqq}{\end{equation}\noindent}

\newcommand{\bce}{\begin{center}}
\newcommand{\ece}{\end{center}}
\newcommand*{\dif}{\mathop{}\!\mathrm{d}}

\newcommand{\barr}{\begin{array}}
\newcommand{\earr}{\end{array}}

\newcommand{\stp}{\stackrel{\P}{\rightarrow}}
\newcommand{\std}{\stackrel{d}{\rightarrow}}

\newcommand{\stw}{\stackrel{w}{\rightarrow}}

\newcommand{\w}{\omega}

\newcommand{\nto}{n\to\infty}
\newcommand{\kto}{k\to\infty}
\newcommand{\xto}{x\to\infty}

\newcommand{\ov}{\overline}
\newcommand{\wt}{\widetilde}
\newcommand{\wh}{\widehat}
\newcommand{\vep}{\varepsilon}

\newcommand{\regvary}{regularly varying}

\newcommand{\regvar}{regular variation}

\newcommand{\bbr}{{\mathbb R}}

\newcommand{\bbz}{{\mathbb Z}}
\newcommand{\bbn}{{\mathbb N}}

\newcommand{\con}{convergence}

\newcommand{\st}{such that}
\newcommand{\fif}{if and only if}
\newcommand{\wrt}{with respect to}
\newcommand{\chf}{characteristic function}
\newcommand{\fct}{function}

\newcommand{\ds}{distribution}

\newcommand{\rep}{representation}
\newcommand{\cmt}{continuous mapping theorem}
\newcommand{\seq}{sequence}

\newcommand{\pro}{probabilit}

\newcommand{\ms}{measure}

\newcommand{\bfx}{{\bf x}}
\newcommand{\bfX}{{\bf X}}

\newcommand{\bfY}{{\bf Y}}
\newcommand{\bfj}{{\bf j}}
\newcommand{\bfy}{{\bf y}}

\newcommand{\bfG}{{\bf G}}

\newcommand{\bft}{{\bf t}}

\newcommand{\bfs}{{\bf s}}
\newcommand{\bfh}{{\bf h}}

\newcommand{\bfii}{{\bf i}}
\newcommand{\bfk}{{\bf k}}
\begin{document}
\today

\title[Whittle estimation based on the extremal spectral density]{Whittle estimation based on the extremal spectral density of a heavy-tailed random field}

\author[Damek, Mikosch, Zhao, Zienkiewicz]{Ewa Damek, Thomas Mikosch, 
Yuwei Zhao, Jacek  Zienkiewicz}
\address{E. Damek and J. Zienkiewicz\\ University of Wroclaw\\Department 
of Mathematics\\
50-384 Wroclaw\\
pl. Grunwaldzki 2/4\\ Poland}
\email{edamek@math.uni.wroc.pl\\
zenek@math.uni.wroc.pl}
\address{T. Mikosch\\ University of Copenhagen\\
Department of Mathematics\\
Universitetsparken 5\\
2100 Copenhagen\\
Denmark}
\email{mikosch@math.ku.dk}
\address{Y. Zhao\\
Fudan University\\
Shanghai Center for Mathematical Sciences\\
2005 Songhu Road\\
Shanghai\\
China}
\email{yuwei$\underline{~}$zhao@fudan.edu.cn}
\thanks{
Ewa Damek's research is partly supported by the NCN 
grant 2019/33/B/ST1/00207. 
Thomas Mikosch's research is partly supported
by  Danmarks Frie Forskningsfond Grant No 9040-00086B. Yuwei Zhao's
research is partly supported by the NSFC grants No.~11971115 and No.~1181086.} 
\begin{abstract}
We consider a strictly stationary random field on the two-dimensional 
integer lattice 
with \regvary\ marginal
and \fidi s. Exploiting the \regvar , we define the spatial extremogram
which takes into account only the largest values in  the random field.
This extremogram is a spatial autocovariance \fct . We define the 
corresponding extremal spectral density and its estimator, the extremal
\per . Based on the extremal \per , we consider the Whittle estimator
for suitable classes of parametric random fields including 
the Brown-Resnick random field and \regvary\ max-moving averages.  
\end{abstract}
\maketitle

\section{Introduction and motivation}
\subsection{A \regvary\ random field}\label{subsec:1}
We consider a $d$-dimensional  strictly stationary  random field
$(X_{\bfs})_{\bfs \in \bbz^2} $ with generic element $X$; the restriction to a two-dimensional lattice
is for notational convenience only. Our focus will be on heavy-tailed fields.
Following the recent developments by Basrak and Planini\'c 
\cite{basrak:planinic:2019} and Wu and Samorodnitsky
\cite{wu:samorodnitsky:2019}, we will deal with a {\em \regvary\
  field} $(X_{\bfs})_{\bfs \in \bbz^2}$. This means that there exists a
random field $(\Xi_\bfs)_{\bfs\in\bbz^2}$
and some $\a>0$ such that for any finite set $ A\subset \bbz^2$ and $ t>0$,
\beam\label{eq:1aa} 
\P \big( (X_\bfs/|X_{\bf0}|)_{\bfs \in A} \in \cdot\, \mid\,|X_{\bf0}| >x\big) &\stw
&  \P\big((\Xi_\bfs)_{\bfs\in A}\in \cdot \big)\,,\\
\dfrac{\P(|X|>t\,x)}{\P(|X|>x)}& \to &t^{-\a}\,,\qquad t>0\,,\qquad \xto\,.\label{eq:1bb}
\eeam
In what follows, we will often switch between the definition 
\eqref{eq:1aa}--\eqref{eq:1bb} and the equivalent sequential version
when $\xto$ is replaced by a \seq\ $a_n\to\infty$ \st\ 
$n\,\P(|X|>a_n)\to 1$ as $\nto$.
\par
Concrete examples studied
throughout the paper are {\em max-stable random fields} with Fr\'echet marginals and 
{\em max-moving average random fields} with iid \regvary\ Fr\'echet noise; see Section~\ref{sec:examples}.

\subsection{The spatial extremogram}\label{sec:extre} For a \regvary\ random
field we can introduce the {\em spatial extremogram}:
\beam\label{eq:spex00}
\gamma(\bfh) = \lim_{x\to \infty}\P\big( |X_{\bfh}|>x\,\mid\,|X_{\bf0}|>x \big)\,, \qquad \bfh\in\bbz^2\,.
\eeam
Calculation yields $\gamma(\bfh)= \E[1 \wedge |\Xi_{\bfh}|^\a]$.
It is not difficult to see that 
\beao
\gamma(\bfh)= \lim_{\xto}\corr\big(\1( |X_{\bfh}|>x)\,,\1(|X_{\bf0}|>x)\big)\,,\qquad \bfh\in\bbz^2\,.
\eeao
Hence $\gamma$ is a proper autocorrelation \fct\ on $\bbz^2$ with the special 
property that it does not assume negative value. In Section~\ref{sec:examples}
the extremogram will be calculated for some \regvary\ fields.
\par
In this paper we consider a special case of the 
general extremogram based on the events $\{x^{-1}X_{\bf0}\in A\}$ and 
$\{x^{-1}X_{\bfh}\in B\}$ with $A=B=\{x:|x|>1\}$. In the literature
this case runs under the names {\em extremal coefficient} or {\em tail dependence  \fct }. The main reason for choosing the special sets $A,B$ is
that we are interested in an estimation problem for particular classes
of random fields; we exploit the particular form of the \fct\ $\gamma(\bfh)$ to construct
a suitable estimator.
\par
The extremogram for \ts\  and general Borel sets $A,B$ was introduced by 
Davis and Mikosch~\cite{davis:mikosch:2009}.
Davis et al. \cite{davis:kluppelberg:steinkohl:2013}, 
Cho et. al.~\cite{cho:davis:ghosh:2016}, Buhl et al. \cite{buhl:davis:klueppelberg:steinkohl:2019}, Huser and Davison \cite{huser:davison:2013,huser:davison:2014} extended this notion to random fields and used it for parameter estimation based on the idea
of pairwise composite likelihood. For \ts , related work is due to 
Linton and Whang \cite{linton:whang:2007}, Han et al. \cite{han:linton:oka:whang:2016} who introduced the quantilogram and cross-quantilogram for measuring
the dependence in the non-extreme parts of the \ts .

\subsection{The empirical spatial extremogram}\label{sec:empextre}
We assume that we observe the random field $(X_\bfs)_{\bfs\in\bbz^2}$
on the index set $\Lambda_n^2=\{1,\ldots,n\}^2$ for increasing
$n$.  Consider an integer-valued  \seq\ $m=m_n\to \infty$ \st\
$\lim_{n\to    \infty} m_n/n =0$. In what follows, we often
suppress the dependence of $m$ on $n$. 
 Following Cho et al. \cite{cho:davis:ghosh:2016}, we denote the {\em
   empirical spatial extremogram} by
 \beam\label{eq:jan29a}
 \widetilde{\gamma} (\bfh) &=& \frac{m_n \sum_{\bfs, \bfs +\bfh \in
   \Lambda^2_n} \1 \big(|X_{\bfs}| >a_{m_n}\,, |X_{\bfs +\bfh }| >a_{m_n}
   \big)}{\#\Lambda_n^2}\nonumber\\&=& \frac{m_n}{n^2} \sum_{\bfs, \bfs +\bfh \in
   \Lambda^2_n} \1 \big(|X_{\bfs}| >a_{m_n}\,, |X_{\bfs +\bfh }| >a_{m_n}
   \big)\,,
 \eeam
where $\bfh$ are the observed lags in $\Lambda_n^2$. In
Section~\ref{subsec:empspatextr} we discuss growth conditions on $(m_n)$
and \asy\ properties of $\wt \gamma$.

\par

\subsection{The extremal spectral density}\label{subsec:extrspd}
For \ts , exploiting the idea that $\gamma$ is an autocorrelation \fct\ of a stationary process, 
Mikosch and Zhao~\cite{mikosch:zhao:2012, mikosch:zhao:2015} considered the 
Fourier series based on $\gamma$ (spectral density), introduced the notion of
an extremal \per\ and proved various basic properties of it.
Among them are \asy\ exponential limit \ds s  
and independence at finitely many distinct \fre ies. This is similar to the classical \per ;
see Brockwell and Davis \cite{brockwell:davis:1991}, Chapter 10, 
for the linear process case, Rosenblatt \cite{rosenblatt:1985}, Theorem 3 on p. 131, for strongly mixing stationary processes,  
and Peligrad and Wu \cite{peligrad:wu:2010} 
for general stationary ergodic \ts . Moreover,  \cite{mikosch:zhao:2012, mikosch:zhao:2015} used these properties
to provide limit theory for the integrated extremal \per\ with different 
weight \fct s. In particular,
they were able to prove limit results for the 
Grenander-Rosenblatt and Cram\'er-von Mises goodness-of-fit test statistics
based on a \fct al \clt\ for the integrated extremal \per . 
\par
The basic idea of this approach
goes back to Rosenblatt \cite{rosenblatt:1956}. Early on, he found
that the empirical \ds\ of the \per\ ordinates of a strictly stationary real-valued \seq\ at the Fourier \fre ies has many properties in common 
with the empirical process of independent exponential \rv s. 
By virtue of a \fct al empirical \clt ,  a continuous \fct al based on
the \per\ at the Fourier \fre ies converges in \ds\ to the corresponding
\fct al based on independent exponential \rv s. A modern (Vapnik-\v Cervonenkis) approach to the empirical spectral process aspects of the \per\ was worked out by Dahlhaus \cite{dahlhaus:1988}. Among others, he introduced the empirical
spectral process indexed by suitable \fct\ classes and explained the
relation with Whittle estimation, spectral goodness-of-fit tests, and various
other applications. 
\par
Similar to the \ts\ case, we can introduce the {\em extremal spectral
density} of the random field $(X_\bfs)_{\bfs\in\bbz^2}$:
\beao
f({\bfomega})=\sum_{\bfh\in\bbz^2} \ex^{i\,{\mathbf \w}^\top \bfh}\,\gamma(\bfh)\,,\qquad {\bfomega}\in [0,2\pi]^2=:\Pi^2\,,
\eeao
where we assume throughout that $\gamma$ is absolutely summable on $\bbz^2$.
Based on the empirical spatial extremogram $\wt \gamma$, we can define 
the  {\em extremal \per } for the
\regvary\ random field $(X_\bfs)_{\bfs\in\bbz^2}$ observed 
on $\Lambda_n^2$ as the empirical version of $f$:
\beao
\wt f(\bfomega) &=& \sum_{\|\bfh\|<n} 
\ex^{i\,\bfomega^\top \bfh}\,\wt \gamma(\bfh)
=\frac{m_n}{n^2}\Big| \sum_{\bft\in\Lambda_n^2} \1(|X_\bft|>a_{m_n})
                         \,\ex^{i \,{\bfomega}^\top\bft}
                         \Big|^2 \,, \qquad {\bfomega}\in \Pi^2\,. 
\eeao 
The goal of this paper is to consider classes of parametric extremal spectral densities
$f_\Theta$, $\Theta\in \bfTh$,
and to estimate the parameter $\Theta_0$ underlying the random field $(X_\bfs)$
through {\em 
Whittle-type estimators.}
This amounts to minimizing the score \fct\ (also called {\em Whittle likelihood;} see Whittle 
\cite{whittle:1953}) 
\beao
\int_{\bfomega \in\Pi^2}
\dfrac{\wt f(\bfomega)}{f_{\Theta}(\bfomega)} \,d\,\bfomega
\eeao
on the parameter set $\bfTh$. In this paper we will not work with this
integral likelihood but a Riemann sum approximation at the Fourier
\fre ies
which is more appropriate (see \eqref{eq:jan31a}) but we will keep the
name of Whittle likelihood.
\par
We will apply Whittle estimation with the extremal spectral density
to  max-stable random fields with Fr\'echet marginals and 
max-moving average fields with Fr\'echet noise. For these classes of random fields
maximum likelihood estimation is difficult since the joint
density of the observations is not tractable and pairwise composite
likelihood methods were employed instead; see Davis et al. 
\cite{davis:kluppelberg:steinkohl:2013}, Cho et al.
\cite{cho:davis:ghosh:2016}, Davison et al. \cite{davison:padoan:ribatet:2012},
Buhl et al. \cite{buhl:davis:klueppelberg:steinkohl:2019}, Huser and Davison \cite{huser:davison:2013,huser:davison:2014}.
The Whittle likelihood involves the extremal \per , i.e., the Fourier transform 
of the entire empirical spatial extremogram $\wt \gamma$. In other words,
this method exploits the whole information contained in $\wt \gamma$.
We show that Whittle estimation based on the extremal \per\ is a serious competitor to the aforementioned estimation techniques.
\par 
The paper is organized as follows. In Section~\ref{sec:preliminaries}
we introduce the necessary mixing conditions, provide a \clt\ for the 
empirical spatial extremogram and \asy\ theory for the extremal \per .
In particular, Theorem~\ref{thm:cltintper} is crucial for proving the \asy\
results on Whittle estimation. In Section~\ref{sec:examples} we introduce 
two major examples of stationary  \regvary\ random fields: the Brown-Resnick 
random field
and max-moving averages. These examples will be used throughout the paper
to illustrate the theory. In Section~\ref{sec:whittle} we present 
the main result of this paper: Theorem~\ref{thm:cltwhittle}
yields a \clt\ for the Whittle estimator based on the extremal \per .
While it is less complicated to verify the conditions of this theorem for
max-moving averages, it takes some effort to check these assumption 
for the Brown-Resnick process. This is achieved in Section~\ref{sec:br}.
We continue with a short simulation study in 
Section~\ref{sec:simu} where we 
focus on parameter estimation in the Brown-Resnick random field and max-moving averages. The remaining sections contain proofs.

\section{Preliminaries}\label{sec:preliminaries}\setcounter{equation}{0}
\subsection{Mixing conditions}
We will work under {\em $\a$-mixing} for a strictly stationary field 
$(X_\bfs)_{\bfs\in\bbz^2}$. We will use the max-norm $\|\bfx\|$ 
in $\bbr^2$ on its subset $\bbz^2$, write $\|T-S\|$ for the distance of two subsets
$T,S\subset\bbz^2$ \wrt\ this norm. 
Following Rosenblatt's \cite{rosenblatt:1956} classical definition, 
the $\alpha$-mixing coefficient between two $\sigma$-fields 
$\mathcal A,\mathcal B$ on $\Omega$
is given by 
\beao
\alpha(\mathcal{A},\mathcal{B})=\sup_{A\in\mathcal A,B\in\mathcal B}
\big| \P(A\cap B)-\P(A) \,\P(B)\big|\,.
\eeao
A related $\alpha$-mixing coefficient for the random field $(X_\bfs)$ 
can be found in Rosenblatt \cite{rosenblatt:1985}, p.~73: 
\beao
  \alpha_{j,k}(h)=\sup_{ S,T\subset \mathbb{Z}^2,
  \#S\le j, \#T\le k,
  \|S-T\|\ge h} \a\big(
  \sigma(X_{\bm{s}},\bm{s} \in
  S),\sigma(X_{\bm{t}},\bm{t}\in T)\big)\,,
\end{eqnarray*}
where $\sigma(X_{\bm{s}}, \bm{s} \in Q)$ denotes the $\sigma$-field
generated by the family of random variables in parentheses. For discussions
of mixing coefficients for random fields we refer to Doukhan \cite{doukhan:1994}, Section 1.3., Bradley \cite{bradley:1993,bradley:2005}, Rosenblatt 
\cite{rosenblatt:1985}, Chapter III.6. 
\par
In what follows, we modify condition {\bf (M1)} in 
Cho et al.~\cite{cho:davis:ghosh:2016} for the purposes of this paper.
These conditions are motivated by small-large block 
techniques which are standard in \asy\ theory for 
strictly stationary fields. We consider
integer \seq s $m_n,r_n\to\infty$ \st\ $m_n=o(n)$ and $r_n=o(m_n)$. 
Recall the definition of $(a_n)$ from Section~\ref{subsec:1}. 
{\em We write $a_m=a_{m_n}$.}
\subsection*{Condition (M1)}
Assume that there exist integer \seq s $(m_n)$, $(r_n)$ as above 
\st\\ $n\,r_n/m_n^{3/2}\to 0$, $r_n^4/m_n \to 0 $, and  
\begin{enumerate}
\item[\rm (1)]
For all $\delta>0$, 
\beam\label{eq:6aa} 
\lim_{h\to\infty} \limsup_{\nto} m_n \sum_{\bfh: h<\|\bfh\|\le r_n}
\P(|X_{\bf0}|>\delta a_m\,,|X_{\bfh}|>\delta a_m\big)&=&0\,.
\eeam
\item[\rm (2)] There exist $K,\tau>0$, $\rho\in (0,1)$ and a non-increasing \fct\ $\alpha(h)$ 
\st\ $\sup_{j,k}\alpha_{j,k}(h)\le \a (h)\le K\,\rho^{h^\tau}$ and 
\beam \label{eq:6bb}
&& \lim_{\nto}m_n\,\alpha(r_n)=0\,.
\eeam
\end{enumerate}
\bre\label{rem:1a} Our
condition {\bf (M1)} is stronger in various aspects than {\bf (M1)} in
\cite{cho:davis:ghosh:2016}. In particular, we require 
uniform bounds for the mixing coefficients $\a_{j,k}$ in (2) and we assume
a geometric-type bound for $(\alpha(h))$. This rate is satisfied by the
major example of this paper, the Brown-Resnick random field in Section~\ref{subsec:br}, whose estimation is the main motivation for writing this paper. In \cite{cho:davis:ghosh:2016} a condition similar
to \eqref{eq:6aa} appears, but these conditions are not directly comparable.
Conditions of the type of  \eqref{eq:6aa} are often referred to as {\em anti-clustering
conditions}; see Davis and Hsing \cite{davis:hsing:1995}, Basrak and Segers \cite{basrak:segers:2009}, Davis and Mikosch
\cite{davis:mikosch:2009}. They ensure that
simultaneous exceedances of high thresholds 
 for $X_\bfs$ with ``small indices'' $\bfs$ are rather unlikely given that
$|X_{\bf0}|$ is large. Indeed, \eqref{eq:6aa} is equivalent to
\beam\label{eq:2}
\lim_{h\to \infty}\limsup_{n\to \infty} \sum_{\bfh: h<\|{\bfh}\|\le r_n} \P\big(|X_{\bfh}|>\delta a_m \,\mid \,|X_{\bf0}|>\delta a_m \big)=0\,, \qquad \delta>0\,.
\eeam 
This fact is easily checked by \regvar\ of $|X|$ and the definition of
$(a_m)$. Conditions \eqref{eq:6bb} and $r_n^4/m_n\to0$  are satisfied for 
$r_n=[(C\,\log m_n)^{1/\tau}]$
for $C>1/(-\log \rho)$. 
\ere
\bre\label{rem:2a}
The main difference between our condition {\bf (M1)} and {\bf (M1)} 
in \cite{cho:davis:ghosh:2016}
is the rate
$n\,r_n/m_n^{3/2}\to 0$. 
Condition {\bf (M1)} in \cite{cho:davis:ghosh:2016} 
requires the alternative rate $m_n^3/n\to 0$. For example, 
if $r_n=[C\,\log n]$ and $m_n=n^{\xi}$ for some $\xi\in (0,1)$ and $C>0$,
then $n\, r_n/m_n^{3/2}\to 0$ holds for $\xi>2/3$ while $m_n^3/n\to 0$ is only
possible for $\xi<1/3$. Throughout this paper it will turn out that
rather large values of $m_n$ compared to $n$ are crucial for the \asy\ 
theory developed in this paper; see the discussion about condition {\bf (M2)}
in the subsequent Section~\ref{subsec:empspatextr}.
\ere
In Section~\ref{sec:examples} we will verify {\bf (M1)} for examples of 
\regvary\ fields.
\subsection{Asymptotic theory for the empirical spatial extremogram}\label{subsec:empspatextr}
 By  \regvar\ 
and stationarity of $(X_\bfh)$ we observe that
\beam\label{eq:3aa}
 \E[\wt \gamma(\bfh)]&=&\dfrac{m_n}{n^2}\E\Big[\sum_{\bfs,\bfs+\bfh\in\Lambda_n^2} \1(|X_\bfs|>a_m\,,|X_{\bfs+\bfh}|>a_m)\Big]\nonumber\\&\sim  & m_n\,\P(|X_{\bf0}|>a_m\,,|X_{\bfh}|>a_m)\nonumber\\
&\sim& p_m(\bfh):=\P(|X_\bfh|>a_m\mid |X_{\bf0}|>a_m)\to \gamma(\bfh)\,,
\eeam
and 
\beao
\dfrac{m_n}{n^2}\,\E\big[\#\{\bfs\in \Lambda_n^2: |X_\bfs|>a_m\}\big]&=&
m_n\,\P(|X|>a_m)\to 1\,,\qquad \nto\,.\nonumber
\eeao
Moreover, under {\bf (M1)} the variances of 
\beao
\dfrac{m_n}{n^2}\sum_{\bfs,\bfs+\bfh\in\Lambda_n^2} \1(|X_\bfs|>a_m\,,|X_{\bfs+\bfh}|>a_m)\,,\qquad \dfrac{m_n}{n^2}\,\#\{\bfs\in \Lambda_n^2: |X_\bfs|>a_m\}\,,
\eeao
converge to zero at rate $O(m_n/n^2)$; see (S7) in the
  supplementary material of \cite{cho:davis:ghosh:2016} using the arguments from \cite{davis:mikosch:2009}.
 Therefore for $\bfh\in\bbz^2$,
$\widetilde{\gamma}(\bm{h}) \stp\gamma(\bfh)$, $\nto$.
Cho et al. \cite{cho:davis:ghosh:2016} proved the corresponding central limit theory under their condition {\bf (M1)}. We modify this result under our 
condition {\bf (M1).}
\begin{theorem}\label{thm:clt}
Consider a strictly stationary regularly varying random field $(X_\bfs)_
  {\bfs\in \mathbb{Z}^2}$ with tail index $\alpha>0$ which also satisfies 
{\rm \bf (M1)}. Then for any finite set 
$A\subset \mathbb{Z}^2$, 
\beao
\dfrac{n}{\sqrt{m_n}} \Big[ \widetilde{\gamma}(\bfh)
  - p_m(\bfh) \Big]_{\bfh \in A} \std \big(Z_\bfh)_{\bfh\in A}\sim N({\bf0}, {\Sigma_A})\,,
\eeao
where the covariance matrix $\Sigma_A$ is given in Theorem 1 of
\cite{cho:davis:ghosh:2016} and $p_m(\bfh)$ is defined in \eqref{eq:3aa}. 
\end{theorem}
Before we provide a sketch of the proof of this theorem some comments are in
place.
\par
The centering constant $p_m(\bfh)$ corresponds to 
$\E[\wt \gamma(\bfh)]$; see \eqref{eq:3aa}.
Davis and Mikosch \cite{davis:mikosch:2009} refer to $p_m(\bfh)$ as 
{\em pre-\asy\ centering,} and they also give examples where $p_m(\bfh)$
cannot be replaced by its limit $\gamma(\bfh)$.
This observation is not untypical in extreme value statistics.
Indeed, parameter estimation involving tail characteristics typically requires
{\em second-order tail \asy s}. 
\par
Pre-\asy\ centering in Theorem~\ref{thm:clt}
can be avoided only if the following second-order tail condition is satisfied
for every fixed $\bfh\in\bbz^2$:
$|p_m(\bfh) -\gamma(\bfh)\big|=o(m_n^{1/2}/n)$ as $\nto$. For the 
purposes of this paper we will need a related condition which requires
uniformity of the \con\ for an increasing number of indices:
\subsection*{Condition (M2)} We have $m_n\,\P(|X_0|>a_m)=1+O(m_n^{-1})$ as $\nto$ and 
\beam\label{eq:oktober}
\lim_{n\to \infty} \dfrac{n}{\sqrt{m_n}} \sup_{\bfh:1\le\|\bfh\|\le r_n}\big|p_m(\bfh) -\gamma(\bfh)\big| =0\,.
\eeam
\bre\label{rem:3a}
Now we return to Remark~\ref{rem:2a} concerning {\bf (M2)}. 
For the examples in this paper we typically have 
$p_m(\bfh)-\gamma(\bfh)=O(1/m_n)$ uniformly for $\bfh$. 
This means that \eqref{eq:oktober}
is satisfied if $n/m_n^{3/2}\to 0$. 
On the other hand, condition {\bf (M1)} in 
\cite{cho:davis:ghosh:2016} requires $m_n^3/n\to 0$. This means that 
{\bf (M2)} cannot hold under their {\bf (M1)}. Therefore we need
to modify the \clt\ in Cho et al. \cite{cho:davis:ghosh:2016} under
the assumption $n\,r_n/m_n^{3/2}\to 0$ required in {\bf (M1)}. 
\ere 
\bce
{\em In the proofs the symbol $c$ 
is a positive constant
whose value may vary from line to line.}
\ece
\noindent {\em Proof of Theorem~\ref{thm:clt}.}
 We indicate where the proof in Appendix A of the supplementary
 material of
Cho et al. \cite{cho:davis:ghosh:2016} has to be altered. 
The proof of the joint \asy\ normality of the empirical spatial extremogram 
at finitely many lags boils down to re-proving Proposition~A2 of the supplementary material in \cite{cho:davis:ghosh:2016} and then 
using the Cram\'er-Wold device.
We exploit the notation of  \cite{cho:davis:ghosh:2016}. In particular,
$\bfY_\bfs= (X_\bft)_{\bft\in \bfs+ A}$.
The field $(\bfY_\bfs)$ inherits
strong mixing from $(X_\bft)$ but the mixing coefficients of $(\bfY_\bfs)$
now depend on the finite set $A$. By the 
geometric-type rate of $(\alpha(h))$ the rate \fct\ for  $(\bfY_\bfs)$ is bounded by
$c\,(\alpha(h))$ for some constant $c>0$.
In view of \regvar\ of $(X_\bfs)$,
$\bfY_\bfs$ is \regvary\ as well with the same tail index, hence  
there exist non-null  Radon \ms s $\mu_{\bfY_{\bf0}}$ and  
$\mu_{\bfY_{\bf0},\bfY_{\bfh}}$ \st\ 
\beao
m_n\,\P(a_m^{-1} \bfY_{\bf0}\in C)&\to & \mu_{\bfY_{\bf0}}(C)\,,\\
m_n\,\P(a_m^{-1} \bfY_{\bf0}\in C\,,a_m^{-1} \bfY_{\bfh}\in C)
&\to& \mu_{\bfY_{\bf0},\bfY_{\bfh}}(C\times C)\,,
\eeao
provided $C$ is bounded away from $0$, both  $C$ and $C\times C$ are continuity sets of the corresponding 
limit \ms s; see \cite{davis:hsing:1995,basrak:segers:2009}.
\par
The following lemma is key to the proof of Theorem~\ref{thm:clt}. 
In contrast to Cho et al. \cite{cho:davis:ghosh:2016} (who appeal to
Stein's lemma in Bolthausen \cite{bolthausen:1982}) we apply a classical
small-large block argument in combination with \chf s.
 \ble Assume {\rm \bf (M1)} and that $C$, $C\times C$ 
are continuity sets with respect to $\mu_{\bfY_{\bf0}}$ and 
$\mu_{\bfY_{\bf0},\bfY_{\bfh}}$, respectively. Then 
\beam\label{eq:5aa}
\dfrac{\sqrt{m_n}}{n} S_n:= \dfrac{\sqrt{m_n}}{n} \sum_{\bfs\in\Lambda_n^2}
\Big[\underbrace {\1\big(a_m^{-1}\bfY_\bfs\in
  C\big)-\P\big(a_m^{-1}\bfY_\bfs\in C\big)}_{ \mbox{$=:\wt I_\bfs$}}\Big]
\std N({\bf0},\sigma_\bfY^2(C))\,,\qquad \nto\,,\nonumber\\
\eeam
where the limiting variance is given by
\beao
\sigma_\bfY^2(C)= \mu_{\bfY_{\bf0}}(C)+ 
2\sum_{ \bfh \in\bbz^2\setminus \{\bf0\}} \mu_{\bfY_{\bf0},\bfY_{\bfh}}(C\times C)\,. 
\eeao
\ele
\begin{proof}[Proof of the lemma]
We divide $\Lambda_n^2$ into $k_n^2=(n/m_n)^2$ disjoint blocks
where we assume without loss of generality that $k_n$ is an integer:
\beao
A_{j,k}=\{\bfl=(l_1,l_2): (j-1)m_n+1\le l_1\le j\,m_n\,,(k-1)m_n+1\le l_2\le k\,m_n\}\,,\qquad j,k=1,\ldots,k_n\,,    
\eeao
and the smaller blocks $\wt A_{j,k}\subset A_{j,k}$:
\beao
\wt A_{j,k}=\{\bfl: (j-1)m_n+1\le l_1\le j\,m_n-r_n\,,(k-1)m_n+1\le
l_2\le k\,m_n-r_n\}\,, \qquad j,k=1,\ldots,k_n\,. 
\eeao
If $k_n$ is not an integer one can apply similar bounds as
in \eqref{eq:october14} below to show that the expected value of the 
absolute value of the sum of the remaining terms $(m_n^{1/2}/n)\wt I_\bfl$ which are not covered by
the index sets $A_{j,k}$ and $\wt A_{j,k}$ is of the order 
$O((m_n^{1/2}/n)(k_n m_n)\P(|X|>a_m))=O(m_n^{-1/2})=o(1)$.
Write
\beao
S_{j,k}=\sum_{\bfs\in A_{j,k}
\backslash \wt A_{j,k}} \wt I_\bfs\,,\qquad
\wt S_{j,k}=\sum_{\bfs\in \wt A_{j,k}} \wt I_\bfs\,.
\eeao
First we consider 
\beam\label{eq:october14}
\dfrac{\sqrt{m_n}}{n} \sum_{j,k=1}^{k_n}\E[|S_{j,k}|]
&\le & c\,\dfrac{\sqrt{m_n}}{n} k_n^2 (m_nr_n)\,\P(|X|>a_m)\le c\,\dfrac{n}{m_n^{3/2}}\,r_n\,.
\eeam
The \rhs\ converges to zero by assumption {\bf (M1)}.
Hence it suffices to prove the \clt\ for  $\wt T_n =
(\sqrt{m_n}/n) \sum_{i=1}^{k_n^2} \wt S_{i}$  with limit \ds\ $N({\bf0},\sigma_\bfY^2(C))$ 
where the $k_n^2$ partial sums $S_{j,k}$ are denoted by
$(\wt S_i)_{i=1,\ldots,k_n^2}$. 
We introduce iid copies $(\wt S_i')_{i=1,\ldots}$ of $\wt S_1$ and write
$\wt T_n' =(\sqrt{m_n}/n) \sum_{i=1}^{k_n^2} \wt S_i'\,.$
 Then for $x\in \bbr^d$, using \eqref{eq:6bb},
\beao
\Big|\E\big[\ex^{i\,x^\top \wt T_n}\big]- \E\big[\ex^{i\,x^\top\wt T_n'}\big]\Big|
&=&\Big|\sum_{k=1}^{k_n^2} \E\Big[\prod_{j=1}^{k-1}\ex^{i\,x^\top\wt S_j} \Big(\ex^{i\,x^\top \wt S_k}
-\E\big[\ex^{i\,x^\top\wt S_k} \big]
\Big)\prod_{l=k+1}^{k_n^2} \E\big[\ex^{i\,x^\top\wt S_1}\big]\Big]\Big|\\
&\le & c\,k_n^2 \alpha(r_n)\le  c\,m_n\alpha(r_n)\to 0\,.
\eeao
In the last step we also used  {\bf (M1)}: since $n/m_n^{3/2}\to 0$
we have $k_n^2\le (n/m_n)^2\le m_n$ for sufficiently large $n$.

Hence the \clt\ for $\wt T_n$ will follow if it holds for $\wt T_n'$.
We will then apply a classical central limit theorem for
  triangular arrays of independent random variables; see for example,
  Theorem 4.1 in Petrov~\cite{petrov:1995}. Thus it remains to
calculate the \asy\ variance $\sigma_\bfY^2(C)$ of $\wt T_n'$. Write 
\beao
c_{\bfs,\bft}=\cov\big(\1(a_m^{-1}\bfY_\bfs\in C), \1(a_m^{-1}\bfY_\bft\in C)\big)\,,
\qquad \bfs,\bft\in \bbz^2\,.
\eeao
We have for fixed $h\ge 1$, by \regvar\ and the definition of $(a_m)$,
\beam
\var(\wt T_n')&=& k_n^2 (m_n/n^2)\var(\wt S_{1})
\le m_n^{-1}\var(\wt S_{1})\nonumber\\
&=& m_n^{-1} (m_n-r_n)^2\,\var(\1(a_m^{-1}\bfY_{\bf0}\in C))
+ m_n^{-1} \sum_{\bft,\bfs\in \Lambda_{m_n-r_n}^2\,,\bft\ne \bfs}c_{\bfs,\bft}\nonumber\\
&= & \mu_{\bfY_{\bf0}}(C)+o(1)\nonumber\\&& +m_n^{-1}\sum_{\bfs,\bft\in\Lambda_{m_n-r_n}^2}
\big(\1_{(0,h]}(\|\bft-\bfs\|)+\1_{(h,r_n]}(\|\bft-\bfs\|)+
\1_{(r_n,\infty)}(\|\bft-\bfs\|)\big)\,c_{\bfs,\bft}\nonumber\\
&=: &\mu_{\bfY_{\bf0}}(C) +o(1)+I_1+I_2+I_3\,, \qquad \nto\,.\label{eq:7aa}
\eeam
By stationarity we observe that there are only finitely many 
distinct covariances in $I_1$ and the number of their appearances
is proportional to $m_n^2$. Therefore, by stationarity and \regvar ,
\beao
I_1\to 2\sum_{
\bfh\in\bbz^2 : 0<\|\bfh\|\le h} 
\mu_{\bfY_{\bf0},\bfY_{\bfh}}(C\times C)\,.
\eeao 
Since $C$ is bounded away from zero there is $\delta>0$ \st
\beao
I_2&\le &
 m_n \sum_{\bfh\in \Lambda_{m_n-r_n}^2\,,h<\|\bfh\|\le r_n}
\P(|\bfY_{\bf0}|>\delta a_m\,,|\bfY_{\bfh}|>\delta a_m\big)\,.
\eeao
In view of the anti-clustering condition \eqref{eq:6aa} the \rhs\ converges to zero by
first letting $\nto$ and then $h\to\infty$.
Finally, by assumption \eqref{eq:6bb} and 
since $\alpha(\|\bfh\|)\le K\,\rho^{\|\bfh\|^{\tau}}$ (assuming without loss of generality that $\alpha(r_n)>0$), we get
\beao
I_3&\le& c\,m_n\,\alpha(r_n) \sum_{\bfh: r_n<\|\bfh\|}\dfrac{\a(\|\bfh\|)}{\a(r_n)}\to 0\,.
\eeao 
In view of the previous calculations for $I_1,I_2,I_3$
the \rhs\ in \eqref{eq:7aa} converges to the desired \asy\ variance
\beao
\var(T_n')\to \mu_{\bfY_{\bf0}}(C)+ 2
\sum_{\bfh \in \bbz^2\setminus \{\bf0\}} \mu_{\bfY_{\bf0},\bfY_{\bfh}}(C\times C)\,.
\eeao 
\end{proof}

\subsection{Asymptotic theory for the extremal
  periodogram}\label{eq:exspden}
Mikosch and 
Zhao~\cite{mikosch:zhao:2012, mikosch:zhao:2015} showed 
that the {\em (integrated) extremal periodogram} 
for \ts\ shares several key
asymptotic properties with the (integrated) periodogram of a linear process 
(such as \asy ally independent exponential \ds\  at distinct \fre ies); cf. 
Brockwell and Davis~\cite{brockwell:davis:1991}, Section 10.3.
\par
For technical convenience,
in the remainder of this paper we will work with a modification 
$\wh\gamma$ of the empirical extremogram $\wt\gamma$ defined in 
\eqref{eq:jan29a}: it has the same structure as 
$\wt \gamma$ except that the indicator \fct s $\1(|X_\bfs|>a_m)$ 
are replaced by the centered versions
$\wh I_\bfs=\1(|X_\bfs|>a_m)-\P(|X|>a_m)$. The resulting empirical  extremogram
and extremal \per\ (we keep the same names) are then given by
\beam
\wh \gamma(\bfh)&=& \dfrac {m_n}{n^2} \sum_{\bfs,\bfs+\bfh\in\Lambda_n^2}
\wh I_\bfs\wh I_{\bfs+\bfh}\,,\qquad \bfh\in\bbz^2\,,\nonumber\\
 \wh f(\bfomega) &= &\frac{m_n}{n^2}\Big| \sum_{\bft\in\Lambda_n^2} \wh I_\bft
\,\ex^{i \,{\bfomega}^\top\bft} \Big|^2  = \sum_{\|\bfh\|<n} 
\wh \gamma(\bfh) \cos (\bfomega^T \bfh)
\,, \qquad {\bfomega}\in \Pi^2\,.\label{eq:2021a}
\eeam  
As a matter of fact, the aforementioned \asy\ results for 
$\wt \gamma$ and $\wh\gamma$ are the same. However, working with the 
centered quantities $\wh I_\bft$ will be beneficial for Fourier analysis 
when mixing conditions are required. 
We also observe that $\wt f(\bfla_\bfj)=\wh f(\bfla_\bfj)$ 
for {\em Fourier \fre ies} $\bfla_\bfj=2\pi\bfj/n\in \Pi^2$.

\par
Mikosch and
Zhao~\cite{mikosch:zhao:2015} proved that
the {\em integrated extremal periodogram} of a stationary 
regularly varying time series satisfies a 
central limit theorem indexed by suitable \fct s. Here we show a related
result for the integrated extremal
periodogram of the random field $(X_{\bfs})_{\bfs \in \Lambda_n^2}$. For practical
purposes it will be convenient to use the Riemann sum approximation
$
  \wh{F}(g) = \frac{(2 \pi)^2}{n^2} \sum_{\bfj\in\Lambda_n^2}
  \wh f(\bfla_{\bfj}) g(\bfla_{\bfj} )
$
to the extremal spectral \ds\ $F(g) = 
\int_{\Pi^2} f(\bfomega)\,
g (\bfomega) \dif \bfomega$ 
indexed by suitable \fct s $g \in L^1
(\Pi^2)$.
The following result will be crucial
for the \asy\ normality of the Whittle estimator; the proof 
is given in Section~\ref{sec:cltintper}.
\bth \label{thm:cltintper}
Assume that $(X_{\bfs })_{\bfs \in \mathbb{Z}^2}$ is 
stationary regularly varying with index $\a>0$ and satisfies 
{\bf (M1)} and {\bf (M2)}. We further assume that 
\beam 
\label{eq:newlabel} 
&& \lim_{n\to \infty}\frac{n}{\sqrt{m_n}} \sum_{\bfh:r_n<\|\bfh\|\le n}\gamma (\bfh)= 0 \,,\quad 
\lim_{\nto} n^7\, \a(r_n)=0\,, \quad (\log n)^4 \dfrac{r_n^2}{m_n}+\dfrac{r_n^4}{m_n}=o(1)\,.   \eeam
Let $g$ be a periodic function satisfying 
$g(x_1, x+2 \pi) = g(\bfx) = g(x_1 +2\pi, x_2)$, $\bfx \in
 \Pi^2$, and
\beam  
 &&\sup_{\bfx \in \Pi^2} \Big| \frac{\partial^2 g(\bfx)}{\partial x_1
   \partial x_2}\Big| <\infty\,.\label{eq:gderiv}
\eeam
Moreover assume that the Fourier coefficients
\beao
\psi_\bfh=\int_{\Pi^2} \cos
  (\bfh^{\top } \bfomega) g(\bfomega) \dif \bfomega\,,\qquad \bfh\in\bbz^2\,,
\eeao
are absolutely summable.
Then the following central limit theorem holds:
\begin{align}
\label{eq:cltintper}
\frac{n}{\sqrt{m_n}}  \dfrac{(2\pi)^2}{n^2} \sum_{\bfj  \in\Lambda_n^2} (\wh{f} (\bfla_{\bfj })
  - f(\bfla_{\bfj} ) )\,g(\bfla_{\bfj}) \std G:=\sum_{\bfh \in \mathbb{Z}^2} \psi_\bfh\,Z_{\bfh} \,,
\end{align}
where $(Z_{\bfh})$ is a mean-zero Gaussian random field whose covariance 
structure is indicated in Theorem~\ref{thm:clt}, and the infinite series
constituting $G$ converges in \ds .
\end{theorem}
\bre
The covariance structure of $(Z_\bfh)$ is described in Cho et al. 
\cite{cho:davis:ghosh:2016}. We refrain from giving formul\ae\ here
because they are complicated and do not contribute to a better 
understanding of this theorem.
\ere
\begin{remark}\label{rem:29}\rm
 Conditions {\bf (M1)} requires that the mixing rate $\alpha(h)$ decays
 faster to zero than any power function as $h\to \infty$. This implies
 that, if $\lim_{n\to \infty} n^7 \alpha(r_n) =0$,
we have\\
$ \lim_{n\to \infty}\, n^7 \sum_{\bfh:
  \|\bfh\|>3r_n}\alpha(\|\bfh \|) =0$.  To get this
one can use a similar argument as in Remark~\ref{rem:1a} with 
$r_n = [(7 C \log n)^{1/\tau}]$ for $C>1/(-\log
\rho)$.
\end{remark}

\section{Examples of \regvary\ random fields}\label{sec:examples}\setcounter{equation}{0}
\subsection{The Brown-Resnick process}\label{subsec:br}
In the context of this paper we consider the special case of
a strictly stationary {\em Brown-Resnick random field} with unit Fr\'echet
marginals  given by
\begin{align}\label{eq:brproc}
X_{\bm{s}} = \sup_{j\ge 1} \Gamma_j^{-1} \ex^{W_{\bm{s}
 }^{(j)} -\delta(\bfs)}\,, \qquad \bm{s} \in \mathbb{R}^2\,,
\end{align}
where
$\Gamma_j=E_1+\cdots +E_j$, $j\ge 1$, $(E_i)$ is an iid \seq\ of standard
exponential \rv s which are independent of the \seq\ of 
iid mean-zero Gaussian random fields $(W_{\bm{s}}^{(j)})_{\bm{s} \in \mathbb{R}^2}$,
$j\ge 1$, with stationary increments, $\delta(\bfs)=\var(W_\bfs^{(j)})/2$.
By construction, $\P(X\le x)= \Phi_1(x)=\ex^{-x^{-1}}$, $x>0$.
A generic element $W$ has 
covariance function 
\beao
\cov \big( W_{\bm{s}}, W_{\bm{t}} \big) = \dfrac{c}{2}\big(
  \|\bm{s}\|^{2H} +\|\bm{t}\|^{2H} - \|\bm{s} -\bm{t}\|^{2H}\big)\,,
\eeao
for some $H\in (0,1]$, $c>0$. Here and in the remainder of this 
subsection $\|\cdot\|$ denotes the Euclidean norm. 
\par
The Brown-Resnick field is a special max-stable process. The latter class
was introduced by de Haan \cite{dehaan:1984}. Kabluchko et al.
\cite{kabluchko:schlather:dehaan:2009}
extended the original work of Brown and Resnick \cite{brown:resnick:1977} 
(who focused on Brownian motion $W$) to general Gaussian processes $W$
with stationary increments. 
\par
We will consider the restriction of the Brown-Resnick field to $\bbz^2$.
Cho et al.  \cite{cho:davis:ghosh:2016} derived the spatial extremogram 
\beam\label{eq:brextrem}
\gamma (\bm{h}) = 2\,\ov \Phi( \sqrt{\delta(
  \bm{h}}))= 2\,\big(1- \Phi \big(\sqrt{\delta (\bfh)}\big)\big) \,, 
  \quad \bm{h}\in \mathbb{Z}^2\,,
\eeam
where $\Phi$ and $\ov \Phi$ 
stand for the standard normal \df\ and its
right tail, respectively. They also calculated uniform bounds for the $\alpha$-mixing 
coefficients (see (34) in  \cite{cho:davis:ghosh:2016})
\beao
\alpha_{j,k} (h)\le c_0\,\sup_{l\ge h}
\dfrac {\ex^{- \delta(l)/2}}{\sqrt{\delta(l)}}=:\alpha(h) \,,
\eeao
for some constant $c_0$ independent of $j,k$. 
They proved that $(X_\bfs)_{\bfs\in\bbz^2}$ is \regvary\ with index $\a=1$. 
Now we choose $m_n=n^{\zeta}$ for $\zeta\in (2/3,1)$ and 
  $r_n=[C\,(\log n + \log m_n)]^{1/(2H)}$ for $C$ sufficiently
large. Then the conditions $nr_n/m_n^{3/2}+r_n^4/m_n\to 0$ and (2) in
{\bf (M1)} are satisfied. 

Cho et al. \cite{cho:davis:ghosh:2016} also derived the formula
\begin{eqnarray*}
 \P(X_{\bf0} \le y\,, X_{\bfh} \le y) 
 & =& \exp \Big\{ - 2 y^{-1} \Phi\big(
     \sqrt{\delta(\bfh)} \big)
      \Big\}\,, \qquad y>0\,. 
\end{eqnarray*}
Hence, choosing $a_n$ \st\ $\P(X>a_n)=1/n$, we have $a_n=n+ O(n^{-1})$
which implies that $\lim_{n\to \infty}m_na_m^{-1} =1$.
Moreover, by a Taylor expansion, uniformly for $\bfh\in\bbz^2$,
\begin{eqnarray*}
 & &m_n\P(X_{\bf0} >a_m\,, X_{\bfh} >a_m)\\ 
 & =& m_n\Big(1- \P (X_{\bf0}\le a_m) - \P (X_{\bfh }\le a_m) + \P
      (X_{\bf0} \le a_m\,, X_{\bfh} \le a_m)\Big)\\ 
 & =& m_n \Big( \frac 2{m_n}-1+ \exp\Big\{-2a_m^{-1}
      \Phi\big(\sqrt{\delta(\bfh)} \big) \Big\} \Big)\\
 & =& 2 + m_n \Big(- 2 a_m^{-1} \Phi ( \sqrt{\delta (\bfh)} ) + 
             \big(2 a_m^{-1} \Phi ( \sqrt{\delta (\bfh)} )  \big)^2 + O(a_m^{-3}) \Big)\\
 & =& \gamma(\bfh)+ O(1/m_n)\,.
\end{eqnarray*}
Using this relation and the growth rates of $(r_n)$ and $(m_n)$, (1) of {\bf (M1)} is satisfied.
Similarly, we can show that 
$p_m(\bfh)-\gamma(\bfh)=O(1/m_n)$ uniformly for $\bfh\in\bbz^2$
and therefore {\bf (M2)} is also
satisfied for $(m_n)$ and $(r_n)$ chosen as above; see
Remark~\ref{rem:3a}.
Choosing $C$ sufficiently large, we also observe that
\eqref{eq:newlabel} holds; see Remark~\ref{rem:29}.

\subsection{Max-moving averages}\label{subsec:mma}
Here we follow Cho et al.  \cite{cho:davis:ghosh:2016}, Section~3.1.
We start with an iid unit Fr\'echet random field $(Z_\bfs)_{\bfs\in\bbz^2}$
and a non-negative weight \fct\ $(w(\bfs))_{\bfs\in\bbz^2}$. The process
\beao
X_\bft= \max_{\bfs\in\bbz^2} w(\bfs)\,Z_{\bft-\bfs}\,,\qquad \bft\in\bbz^2\,,
\eeao
is called {\em max-moving average} (MMA). It is a max-stable process
with unit Fr\'echet marginals.
Obviously, $(X_\bft)$ is strictly stationary if it is finite a.s. Since $\P(Z_{\bf0}\le x)=\Phi_1(x)$, $x>0$,
it is easily seen that 
\beao
\P(X_{\bf0}\le x)= \ex^{-x^{-1} \sum_{\bfs\in \bbz^2} w(\bfs)}=:\Phi_1^{ w_0}(x)\,,
\eeao
and therefore $w_0<\infty$ is necessary and sufficient
for the existence of $(X_\bft)$.
\par
Next we calculate the extremogram. We have
$
m_n\,\P(X>a_m)= 
 m_n(1- \ex^{-a_m^{-1} w_0})\to 1\,.
$
Therefore we may choose $a_m=m_n\,w_0$. Then we have for $x>0$, uniformly for
$\bfh\in\bbz^2$,
\beao
\P( X_\bfh>a_m\mid X_{\bf0}>a_m)
&=& \dfrac{\P\Big(\max_{\bfs\in\bbz^2} w(\bfs+\bfh) Z_{-\bfs}\wedge 
\max_{\bfs \in\bbz^2} w(\bfs) Z_{-\bfs }>a_m\Big)}{\P (X > a_m ) }\\
&=& \dfrac{\sum_ {\bfs\in\bbz^2} w(\bfs)\wedge 
w(\bfs+\bfh)}{\sum_{\bfs\in \bbz^2} w(\bfs)}+O(1/m_n)\\
&=&\gamma(\bfh) + O(1/m_n)\,.
\eeao
For a finite MMA we have $w(\bfs)=0$ for $\|\bfs\|\ge k_0$ for some $k_0>1$.
Then {\bf (M1), (M2)} are easily verified for 
suitable choices of $(r_n)$, $(m_n)$.
\par
If the dependence
ranges over infinitely many lags the anti-clustering condition
\eqref{eq:2} can still be verified. Indeed, the Taylor expansion
argument used above holds uniformly for $\bfh$ and therefore
\beao
\sum_{\bfh\in\bbz^2:h<\|\bfh\|\le r_n} \P\big(\bfX_\bfh>\vep a_m\mid \bfX_{\bf0}>\vep a_m\big)&\le & \sum_{\bfh\in\bbz^2:h<\|\bfh\|\le r_n} \gamma(\bfh)+c\,m_n^{-1}
\#\big\{\bfh\in\bbz^2: \|\bfh\|\le r_n\big\}\\
&\le  &  \sum_{\bfh\in\bbz^2:h<\|\bfh\|} \gamma(\bfh)+ O(r_n^2/m_n)\,.
\eeao
The \rhs\ converges to zero if $(\gamma(\bfh))$ is 
summable and $r_n^2/m_n\to 0$. 
The strong mixing condition follows from a result by Dombry and Eyi-Minko
\cite{dombry:eyiminko:2012} (see Proposition 1 in \cite{cho:davis:ghosh:2016}),
and one obtains
\beao
\alpha_{j,k}(h)\le c\sum_{\bfs\in S,\bft\in T: \|\bft-\bfs\|=\|\bfh\|\ge h} \gamma(\bfh)
\,. 
\eeao
The expression on he \rhs\ can be taken as a definition of $\alpha(h)$.
For example, if the weights $w(\bfs)$ are chosen \st\ $\alpha(h)$ decays exponentially fast then we can find $(r_n)$ and $(m_n)$ satisfying {\bf (M1), (M2)}
and \eqref{eq:newlabel}.
\section{The Whittle estimator}\label{sec:whittle}\setcounter{equation}{0}
Throughout this section we consider stationary \regvary\ 
random fields $(X_{\bft}(\Theta))_{\bft \in \mathbb{Z}^2}$ with 
the same tail index $\a>0$ which are parametrized
by $\Theta=(\theta_1,\ldots,\theta_s)\in \bfTh$ for some $s\ge 1$.
The parameter set $\bfTh$ is a compact subset of $\bbr^s$.  
Our observations stem from $(X_\bft)=(X_{\bft}(\Theta_0))$ 
for some parameter $\Theta_0\in\bfTh$ which is also assumed to be an 
inner point of $\bfTh$.  
Our goal is to estimate $\Theta_0$ from the observations 
$(X_\bft)_{\bft\in\Lambda_n^2}$. We write $f_\Theta$ and 
$\gamma_\Theta$ for the extremal spectral density and extremogram of
 $(X_{\bft}(\bfTh))$, respectively. 
The {\em Whittle estimator of $\Theta_0$} is the minimizer $\Theta_n$ on $\bfTh$ of the {\em discrete Whittle likelihood  \fct } 
\beam\label{eq:jan31a}
  \sigma^2_n(\Theta ) := \dfrac{\overline{\sigma}^2_n(\Theta )}{n^2} 
\sum_{\bfj \in
    \Lambda_n^2} \dfrac{\wh {f}(\bfla_{\bfj})}{ f_{\Theta} (\bfla_{\bfj})}\,,\quad 
\mbox{ where }\quad \overline{\sigma}_n^2 ( \Theta) = \exp\Big(n^{-2} \sum_{\bfj
 \in \Lambda_n^2 } \log f_{\Theta}(\bfla_{\bfj}) \Big)\,. 
\eeam
Throughout we will work under the following assumptions.

\subsection*{Condition (W)} 
\begin{enumerate}
\item[\rm (1)]
$(X_{\bft})_{\bft
\in \mathbb{Z}^2}$ satisfies {\bf (M1)} and {\bf
(M2)} and in addition \eqref{eq:newlabel}.
\item[\rm (2)] 
for all $\Theta\in\bfTh$,
\beam\label{eq:jan30a} 
  0 < \inf_{(\bfomega, \Theta) \in \Pi^2 \times \bfTh}
  f_{\Theta}(\bfomega )  \,.
\eeam
\item[\rm (3)] For each $\Theta\in \bfTh$ with $\Theta\ne \Theta_0$,
$f_{\Theta_0}/f_{\Theta}$ is not constant on $\Pi^2$.
\item[\rm (4)]
  The following relations hold 
\begin{eqnarray}
\label{eq:new41}
  \sup_{\Theta\in\bfTh}\sum_{\bfh \in \mathbb{Z}^2} |h_1 h_2 \gamma_\Theta (\bfh)| & < & \infty \,,\\ 
\label{eq:new42}
\max_{i,j\in \{1,\ldots, s\}}\sup_{\Theta \in \bfTh} \sum_{\bfh
  \in \mathbb{Z}^2} \Big[\Big| \frac{\partial
  \gamma_{\Theta}(\bfh)}{\partial\theta_{i}} \Big|+ \Big| \frac{\partial^2 \gamma_{\Theta}
  (\bfh)}{\partial \theta_{i} \partial \theta_{j}} \Big|\Big]
& < & \infty\,. 
\end{eqnarray}
\item[\rm (5)]The matrix $\var \big( \partial \log f_{\Theta_0}
  (\bfU)/\partial \Theta \big)$ with $\bfU$ uniform on 
  $\Pi^2$ is non-singular. 
\end{enumerate}
\bre
Condition \eqref{eq:new41} implies 
the summability of $(\gamma_{\Theta}(\bfh))$ uniformly on $\Theta\in\bfTh$, ensuring in particular that
\beam\label{eq:march1a}
\sup_{(\bfomega,\Theta)\in \Pi^2\times \bfTh} f_\Theta(\bfomega)<\infty\,.
\eeam
We also have
\begin{eqnarray*}
  \sup_{\Theta \in \bfTh} \sup_{\bfomega \in \Pi^2}\Big| \frac{\partial^2
  f_{\Theta}(\bfomega)}{\partial \omega_1 \partial \omega_2} \Big| <\infty
\quad\mbox{and}\quad  \sup_{\Theta \in \bfTh} \sup_{\bfomega \in \Pi^2}\Big| \frac{\partial^2
  f^{-1}_{\Theta}(\bfomega)}{\partial \omega_1 \partial \omega_2} \Big| <\infty\,.
\end{eqnarray*}
The first relation follows from \eqref{eq:new41}  and the second one 
by a combination of \eqref{eq:jan30a} and \eqref{eq:new41}; for details
see \eqref{eq:rev1}.
\ere
Now we present our main result about the \asy\ normality of the 
Whittle estimator $\Theta_n$ of $\Theta_0$.
\bth\label{thm:cltwhittle}
Assume that condition {\bf (W)} holds. Then a unique minimizer 
$\Theta_n$ of the discrete Whittle likelihood \fct\ 
\eqref{eq:jan31a} exists and   
\begin{eqnarray}
\label{eq:cltwhittle}
  \frac{n}{\sqrt{m_n}} \big( \Theta_n - \Theta_0 \big) \overset{d}{\to
  } \bfG\sim N({\bf0},W)\,,
\end{eqnarray}
where the Gaussian vector $\bfG$ is defined in \eqref{eq:march8a}. 
\ethe

\section{Example: The Brown-Resnick random field}\label{sec:br}\setcounter{equation}{0}
From Section~\ref{subsec:br} we recall the definition of a 
Brown-Resnick field with parameter $H\in (0,1)$ and extremogram
$\gamma_H (\bm{h}) = 2\,\ov \Phi( \sqrt{c}\|\bfh\|^H)$; see \eqref{eq:brextrem}.
In this section $\|\cdot\|$ is Euclidean distance.
We will verify the 
conditions (1)-(5) of {\bf (W)} in Theorem~\eqref{thm:cltwhittle} for this case.
\par
As a parameter space for $H$ we choose $\bfTh=[\delta,1-\delta]$ for a fixed 
but arbitrarily small $\delta>0$ and assume $H_0\in (\delta,1-\delta)$.
It will be convenient to refer to $f_H'$, $\gamma_H'$, etc. as the derivatives of $f_H$, $\gamma_H$, etc. \wrt\ $H$.\\
{\bf (1)} In view of the discussion in Section~\ref{subsec:br} the mixing conditions  of {\bf (W)} are satisfied.\\
{ \bf (3)} $f_{H_0}/f_H$ is non-constant.\\
{ \bf (4)} Due to the structure of $\gamma_H$, the infinite series
$\sum_\bfh|h_1h_2\gamma_H(\bfh)|$, $\sum_\bfh|\gamma_H'(\bfh)|$
and $\sum_\bfh |\gamma_H''(\bfh)|$ are finite and 
continuous \fct s of $H$. Hence their suprema over $H\in\bfTh$ are finite.\\
{ \bf (5)} This condition reads as $\var(f_{H_0}'(\bfU)/f_{H_0}(\bfU))>0$ 
which is satisfied because $f_H'(\bfU)/f_H(\bfU)$ is not constant.
\par
In the remainder of this section we will verify {\bf (2)}: the positivity of
the spectral density $f_H$ for all $H\in \bfTh$. Since $f_H(\bfomega)$ 
is continuous on $\Pi^2\times \bfTh$ the infimum of $f_H(\bfomega)$ 
over $\Pi^2\times \bfTh$ is positive.

\subsection*{Positivity of the spectral density}
The \fct\ $(\gamma_H(\bfh))_{\bfh\in\bbr^2}$ is non-negative definite, isotropic
(i.e., depends only on $\|\bfh\|$) 
 and for any multi-index $\alpha $, $\bfh ^{\alpha}\gamma _H(\bfh )$ is integrable and vanishes at infinity.  
Therefore $\gamma_H$ is the Fourier transform of a finite 
positive measure on $\bbr ^2$ which has a smooth Lebesgue density $g_H$, i.e., 
$\gamma_H=\wh g_H$ on $\bbr^2$. Indeed, direct calculation shows 
$$
D^{\alpha }\int _{\mathbb{R} ^2}\ex^{i\,\bfx ^\top \bfh }\,\gamma _H(\bfh)\, d\bfh
=i^{|\alpha |}\int _{\mathbb{R} ^2}\ex^{i\,\bfx ^\top\bfh } \bfh^{\alpha }\,\gamma _H(\bfh)\ d\bfh.$$
Note that  $g_H$ inherits isotropy from $\gamma_H$. 
\par
The following result is key to the proof of the positivity of $f_H$.
The proof is provided at the end of this section.
\bth\label{thm:aux} For each $H\in (0,1)$, there is $c>0$ such that 
\beam\label{eq:main}
g_H(\bfx)=c\,\| \bfx\| ^{-(2+H)}(1+o(1))\,,\qquad\|\bfx\|\to\infty\,.
\eeam
\ethe
By \eqref{eq:main} the following quantity is well defined:
\beam\label{eq:main1}
\wt g_H (\bfx)=\sum _{\bfh\in \bbz ^2}g_H(\bfx+2\pi \bfh)\,, \qquad\bfx\in \Pi^2\, ;
\eeam
see the calculations below.
For $\bfm\in \bbz ^2$ we have
\beao
\int _{\Pi^2}\ex^{-i\,\bfm^\top \bfx}\,\wt g_H(\bfx)\, d\bfx&=& 
 \sum _{\bfh\in \bbz ^2} \int _{\Pi^2}\ex^{-i\,\bfm^\top\bfx} g_H(\bfx+2\pi \bfh)\, d\bfx\\
&=& \sum _{\bfh\in \bbz ^2} \int _{\Pi^2+2\pi \bfh}\ex^{-i\,\bfm^\top (\bfx-2\pi \bfh)} \,g_H(\bfx)\, d\bfx\\
&=& \sum _{\bfh\in \bbz ^2} \int _{\Pi^2+2\pi \bfh}\ex^{-i\,\bfm^\top \bfx} \,g_H(\bfx)\, d\bfx\\
&=& \int _{\bbr^2}\ex^{-i\,\bfm^\top \bfx} g_H(\bfx)\, d\bfx= \wh g_H(\bfm)=\gamma_H (\bfm)\,.
\eeao
Therefore
\beao
\wt  g_H(\bfx)=\frac{1}{(2\pi) ^2}\,\sum _{\bfh\in \bbz ^2} \ex^{i\,\bfh^\top \bfx}\,
\gamma_H (\bfh)=\frac{1}{(2\pi) ^2}f_H(\bfx)\,.
\eeao
We choose  $r>\sqrt{2}$ and take $\bfh$ \st\ $\| \bfh\| >r$. 
An application of the triangle inequality for $\bfx\in\Pi^2$ 
ensures  that $\|\bfx+2\pi \bfh\| \geq (r-\sqrt{2})2\pi$. If $r$ 
is sufficiently large then in view of \eqref{eq:main} 
	we have for $\bfx\in\Pi^2$, 
$$
g_H(\bfx+2\pi \bfh)\geq \frac{c}{2\| \bfx+2\pi \bfh\|^{2+H}}\,.$$
Therefore 
\beao
\wt g_H (\bfx)
&\geq& 
\sum _{\bfh\in \bbz ^2, \| \bfh\| 
\geq r}\,g_H(\bfx+2\pi \bfh)  
\geq \sum _{\bfh\in \bbz ^2, \| \bfh\| \geq r}\frac{c}{2\|\bfx+ 2\pi \bfh\| ^{2+H}}\\
 &\geq& \frac{1}{(3\pi )^{2+H} }\,\sum _{\bfh\in \bbz ^2, \| \bfh\| >r}\frac{c}{2\|\bfh\| ^{2+H}}>0. 
\eeao
In the last step we used
\beao
\| \bfx+2\pi \bfh\| \leq \| \bfx\|+2\pi \| \bfh\| \leq 2\sqrt{2}\pi +2\pi \| \bfh\| \leq 3\pi \| \bfh\|.
\eeao
Finally, we prove that the series in \eqref{eq:main1} converges. By
\eqref{eq:main} there is $C>0$ 
\st  
\beao 
\sum _{\bfh\in \bbz ^2, \| \bfh\| > 4\sqrt{2}\pi }g_H(\bfx+2\pi \bfh)
&\leq& \sum _{\bfh\in \bbz ^2, \| \bfh\| >4\sqrt{2}\pi}\frac{C}{\|\bfx+ \bfh\| ^{2+H}}
\leq \sum _{\bfh\in \bbz ^2, \| \bfh\| >4\sqrt{2}\pi}\frac{2^{2+H}C}{\|\bfh\| ^{2+H}}\\. 
&\leq &4\sum _{\bfh\in \bbz ^2, h_1\geq 2\ \mbox{or}\  h_2\geq 2}\frac{2^{2+H}C}{\|\bfh\| ^{2+H}}\\
&\leq &4\sum _{\bfh\in \bbz ^2, h_1\geq 2, h_2\geq 2}\int _{[h_1-1,h_1]\times [h_2-1,h_2]}\frac{2^{2+H}C}{\|\bfy\| ^{2+H}}\, d\bfy+ 16\sum _{m\geq 2}\frac{C2^{2+H}}{m^{2+H}}.
\eeao
Moreover, 
\beao
\sum _{\bfh\in \bbz ^2, h_1\geq 2, h_2\geq 2}\int _{[h_1-1,h_1]\times [h_2-1,h_2]}\frac{2^{2+H}C}{\|\bfy\| ^{2+H}}\, d\bfy
&\leq& \int _{\| \bfy\| >1 }\frac{2^{2+H}C}{\|\bfy\| ^{2+H}}\, d\bfy\\
&=&C\, 2^{2+H}\,2\pi \int _{1}^{\infty} r^{ -1-H}\ dr <\infty. 
\eeao 
\subsection*{Proof of Theorem~\ref{thm:aux}}
In what follows, we need properties of the Fourier transform of 
Schwartz functions and distributions; they will be  quoted from 
Rudin \cite{rudin:1973}. The
	Fourier transform and its inverse 
of a \fct\ $\eta$ on $\bbr^2$ will be denoted by
\beao
	\mathcal{F}\eta (\bfh)=\int _{\bbr ^2}\eta (\bfx)\,\ex^{-i\bfx^\top \bfh}\ {\rm d}\bfx\quad\mbox{ and }\quad
	\mathcal{F}^{-1}\eta (\bfx)=\frac{1}{(2\pi )^2}
\int _{\bbr ^2}\eta(\bfh)\,\ex^{i\bfx^\top \bfh}\ {\rm d}\bfh\,,  
\eeao
respectively. We observe that $\gamma_H$ is a smooth function 
off the origin. We split it into two parts
	\beao
\gamma_H =\gamma _1+\gamma _2\,,\qquad \mbox{where}\qquad \gamma _1=\gamma \psi\,,
\eeao
and  $\psi \in C_c^{\infty}(\bbr ^2)$ is an isotropic bump \fct\ \st\
$\psi (\bfh)=1$ for $\| \bfh\|\leq 1$, ${\rm supp}\, \psi \subset \{\bfh: \| \bfh\| \leq 2\}$.
	Then 
\beao
g_H=g_1 +g_2\qquad\mbox{where} \qquad g_i=\mathcal{F}^{-1}\gamma _i\,,\qquad i=1,2\,.
\eeao
Since $\gamma_2$ is a Schwartz function so is $g_2$ (Theorem 7.4 in \cite{rudin:1973}) and therefore it suffices for \eqref{eq:main} to describe the \asy\ behavior of $g_1(\bfx)$ as $\|\bfx\|\to \infty$.
Since the density $\ex^{-x^2/2}/\sqrt{2\pi}$ of the standard 
normal distribution is real analytic and $\xi=2(1-\Phi )$ inherits this property, it can be written as a series $\xi(t)=\sum _{k=0}^{\infty}c_k\,t^k$
whose \con\ is uniform on compact sets. Therefore
\beao
\gamma _1(\bfh)=\Big(\sum _{k=0}^{m-1}+\sum _{k=m}^{\infty}\Big)c_k\,c^{k\slash 2}\,\| \bfh\| ^{H\,k}\,\psi (\bfh)=:I_m(\bfh)+R_m(\bfh), 
\eeao
Notice that $\| \bfh\| ^{kH}\psi (\bfh)$ has at least $\lfloor kH\slash 2\rfloor $ integrable derivatives. Hence
	 $R_m$ has $\lfloor mH\slash 2\rfloor $ derivatives and we choose $m$ such that $\lfloor mH\slash 2\rfloor> H+2$.
Then  $\mathcal{F}^{-1}R_m(\bfx)$ 
decays faster than $\| \bfx \| ^{-H+2}$. Indeed, by Theorem 7.4 
in \cite{rudin:1973}, for $p=\lfloor mH\slash 2\rfloor $, 
$j=1,2$ and $D_j^p=\frac{\partial ^p}{\partial x_j^p}$, we have
\beao
\mathcal{F} ^{-1} (R_m)(\bfx)=(-ix_j)^{-p}\mathcal{F} ^{-1} (D_j^pR_m)(\bfx)
\eeao
and
\beao
|\mathcal{F} ^{-1} (D_j^pR_m)(\bfx)|\leq \int _{\mathbb{R}^2}|D_j^pR_m)(\bfh)|d\bfh <\infty .
\eeao
Therefore
it remains to study $\mathcal F^{-1}I_m$ for each single term
	 $\| \bfh\| ^{Hk}\psi (\bfh)$ for $k=1,\ldots,m-1$; for $k=0$ we again have the Schwartz function $\psi$.
\par
For $\beta,A>0$ we introduce the quantities
\beao
F_{\beta}(\bfx)=\frac{1}{(2\pi )^2}\int _{\bbr ^2}\| \bfh\| ^{\beta }\psi (\bfh)
\,\ex^{i\bfx^\top \bfh}\ {\rm d}\bfh  \quad\mbox{and}\quad F_{\beta ,A}(\bfx)=
\frac{1}{(2\pi )^2}\int _{\bbr ^2}\| \bfh\| ^{\beta }\psi (\bfh/A) \,\ex^{i\bfx^\top \bfh}\, {\rm d}\bfh\,. 
\eeao
\ble\label{l1}
The following statements hold:
\begin{enumerate}
\item[\rm 1.]
For all $\bfx\neq \bf0$ the limit
 $\wt F _{\beta}(\bfx)=\lim _{A\to \infty}F_{\beta , A}(\bfx)$ exists.
\item[\rm 2.]
The limit
$\lim _{\|\bfx \|\to \infty}F_{\beta}(\bfx)\,\|\bfx\| ^{2+\beta}=c_{\beta}$
exists and $c_{\beta}\neq 0$ for $0<\beta \leq 1$. In particular,  
\beao
\lim _{\| \bfx\| \to \infty}\| \bfx\| ^{2+H }\frac{1}{(2\pi )^2}\int _{\bbr^2}
\ex^{i \bfx^\top \bfh} I_m(\bfh)\,{\rm d}\bfh=
\lim _{\| \bfx\| \to \infty} F_H(\bfx)\,\| \bfx\| ^{2+H }=c_{H}\ne 0.
\eeao
\end{enumerate}
\ele
\begin{proof} {\bf 2.}
Note that $F_{\beta}$ and $F_{\beta ,A}$ are isotropic.
We will show that {\bf 1.} implies {\bf 2.}.
First we show that $\wt F _{\beta }$ is homogeneous. Indeed, for $t>0$,
 \begin{align*}
\wt F _{\beta }(t\,\bfx)=&\lim _{A\to \infty}\frac{1}{(2\pi )^2}\int _{\bbr ^2}\| \bfh\| ^{\beta }\psi(\bfh/A)\,\ex^{i\,t\,\bfx^\top \bfh}\,{\rm  d}\bfh\\
=&\lim_{A\to\infty}\frac{1}{(2\pi )^2}\int _{\bbr ^2}\| \bfh/t\| ^{\beta }\psi(\bfh/(tA)) \,\ex^{i\bfx^\top \bfh}\,t^{-2}\ {\rm d}\bfh=t^{-\beta -2}\,\wt F _{\beta }(\bfx)
 \end{align*}
Writing $\bfx=\| \bfx\| \bfx_0$, we have
 $$
 \wt F _{\beta }(\bfx)=\wt F _{\beta }(\| \bfx\| \bfx_0)=\| \bfx\| ^{-\beta -2}\wt F_\beta(\bfx_0)$$
and since $\wt F_\beta(\bfx)$ is isotropic, $\wt F_\beta(\bfx_0)=:c_\beta$ does not not depend on $\bfx_0$. Fix a unit vector $\wt \bfx _0$. Then 
$F_{\beta }(\bfx)=	F_{\beta }(\| \bfx\| \bfx _0)=F_{\beta }(\| \bfx \| \wt \bfx _0)$ and 
\beao
 \lim _{\| \bfx\|\to \infty} F_{\beta}(\bfx)\| \bfx\| ^{2+\beta }
&=&\lim _{\| \bfx\| \to \infty}\| \bfx\| ^{2+\beta }\frac{1}{(2\pi )^2}\int _{\bbr ^2}\| \bfh\| ^{\beta }\psi (\bfh)\,\ex^{i\|\bfx\|\wt \bfx_0^\top \bfh}\ {\rm d}\bfh \\
 &=&\lim _{\| \bfx\| \to \infty}\| \bfx\| ^{2+\beta }\frac{1}{(2\pi )^2}\int _{\bbr ^2}\| \bfu\| ^{\beta }\psi(\bfu/\| \bfx\|) \ex^{i\wt \bfx_0^\top \bfu}\| \bfx\| ^{-2-\beta }\ {\rm d}\bfu\\
& =&\lim _{\| \bfx\| \to \infty}\frac{1}{(2\pi )^2}\int _{\bbr ^2}\| \bfu\| ^{\beta }\psi(\bfu/\| \bfx\|)\,\ex^{i\wt \bfx_0^\top \bfu}{\rm  d}\bfu =c_{\beta }.
\eeao 
It remains to prove that $c_\beta\neq 0$ for $\beta \leq 1$.  
The functions $\Psi _A(\bfh)=\| \bfh\| ^\beta\,\psi (\bfh/A)$ 
converge to $\| \bfh\| ^\beta$ in the sense of distributions because for a test function $\varphi \in C_c^{\infty}(\bbr ^2)$
\beao 
\lim _{A\to \infty}\int _{\bbr ^2}\| \bfh\| ^\beta\,\psi (\bfh/A) \,
\varphi (\bfh)\, {\rm d}\bfh =\int _{\bbr ^2}\| \bfh\| ^\beta \,\varphi 
(\bfh)\,{\rm  d}\bfh\, .
\eeao
Here we have in mind tempered distributions and the Fourier transform defined on them.
Therefore $\mathcal{F}^{-1}\Psi _A(\bfh)$ tends to a non-zero 
$\mathcal{F}^{-1}(\| \bfh\|^\beta)$ in the sense of distributions; 
see Theorem 7.15 in \cite{rudin:1973}. Now suppose that $c_\beta=0$. 
Then for a test function $\varphi $ such that ${\bf0}\notin {\rm supp} \,\varphi $,
\beao
\langle \mathcal{F}^{-1}\Psi _A, \varphi \rangle= 
\int _{\bbr ^2}\frac{1}{(2\pi )^2}\int _{\bbr ^2}\ex^{i\bfx^\top \bfh}\Psi _A (\bfh) \varphi (\bfx)\ {\rm d}\bfh \,{\rm d}\bfx\to \int _{\bbr ^2} \wt F_\beta(\bfx)\,\varphi (\bfx)\,{\rm  d}\bfx =0.
\eeao
Therefore, the support of $\mathcal{F}^{-1}(\| \bfh\|^\beta)$ 
must be contained in $\{ \bf0\}$ and so  $\mathcal{F}^{-1}(\| \bfh\|^\beta)$ 
is a differential operator; see Theorem 6.25 in \cite{rudin:1973}. 
However the latter 
is not possible because $\| \bfh\| ^\beta$ is not a polynomial. Indeed,
the Fourier transform of a differential operator 
is a polynomial. This follows from the definition 
of the Fourier transform and Theorem 7.15 in \cite{rudin:1973}.\\[2mm]
{\bf 1.} For $A>1$
choose an integer $m$ such that $2^m\leq A< 2^{m+1}$. Then
\beao
\psi \left (\frac{\bfh}{A}\right )&=&\psi \left (\frac{\bfh}{A}\right )-\psi \left (\frac{\bfh}{2^m}\right )
+\sum _{s=1}^m\left (\psi \left (\frac{\bfh}{2^s}\right )-\psi \left (\frac{\bfh}{2^{s-1}}\right )\right )+\psi (\bfh)\,.
\eeao
Thus
\beao\lefteqn{
\int _{\bbr ^2}\ex^{i\bfx^\top \bfh}\,\| \bfh\| ^{\beta}\,\psi \left (\frac{\bfh}{A}\right )\ {\rm d}\bfh}\\&=&\underbrace{\int _{\bbr ^2}\ex^{i\bfx^\top \bfh}\| \bfh\| ^{\beta}\left (\psi \left (\frac{\bfh}{A}\right)-\psi \left (\frac{\bfh}{2^m}\right )\right )\ {\rm d}\bfh}_{=:I_0(\bfx)}\\
&&+\sum _{s=1}^m\underbrace{\int _{\bbr ^2}\ex^{i\bfx^\top \bfh}\| \bfh\| ^{\beta}\left (\psi \left (\frac{\bfh}{2^s}\right)-\psi \left (\frac{\bfh}{2^{s-1}}\right )\right )\ {\rm d}\bfh}_{=: I_s(\bfx) \,(2\pi)^2} +\int _{\bbr ^2}\ex^{i\bfx^\top \bfh }\| \bfh\| ^{\beta}\psi (\bfh)\, {\rm d}\bfh.
\eeao
First we prove that the sum converges as $m\to \infty$. 
Changing variables, $\bfh=2^{s-1}\bfu$, 
we have
\beao
I_s(\bfx)&=& \frac{1}{(2\pi )^2}\int _{\bbr ^2}\ex^{i\, 2^{s-1}\,
\bfx^\top\bfu}\,2^{(s-1)\beta } \| \bfu\| ^{\beta}\left (\psi \left (\frac{\bfu}{2}\right )-\psi (\bfu)\right )\,2^{2(s-1)}\ {\rm d}\bfu\\
&=&\mathcal{F}^{-1}\rho (2^{s-1}\bfx)\,2^{(s-1)(\beta +2)}, 
\eeao
where
$\rho (\bfu)=\| \bfu\| ^{\beta}\,(\psi(\bfu/2)-\psi (\bfu))$
is a smooth function supported on $\{\bfu\in \bbr ^2: \|\bfu\| \leq 4\}$. 
Indeed, $\psi (\bfu/2)-\psi (\bfu)=0$ if $\| \bfu\| \leq 1$ 
which cuts out the singularity of $\| \bfu\| ^{\beta}$. 
Hence $\mathcal{F}^{-1}\rho$ is a Schwartz function, and so for all $M$ and $\bfx\neq \bf0$,
$|\mathcal{F}^{-1}\rho (\bfx) |\leq C_M \| \bfx\|^{-M}.$
Therefore, for $\bfx\neq \bf0$
\beao
|I_s(\bfx)|&\leq 2^{(s-1)(\beta +2)}|\mathcal{F}^{-1}\rho (2^{s-1}\bfx) |
\leq 2^{(s-1)(\beta +2)}\,C_M\,\| 2^{s-1}\bfx\|^{-M}
\leq 2^{(s-1)(\beta +2)} 2^{-(s-1)M}\| \bfx\|^{-M}.
\eeao
This proves that $\sum _sI_s(\bfx)$ converges for $\bfx\neq 0$ 
(we choose $M>\beta +2$).
\par
Now we consider $I_0(\bfx)$. 
For $A>0$ define 
$$
\eta _A(\bfh)=\| \bfh\| ^{\beta}\,\left(\psi (\bfh)-\psi \left (\frac{A\bfu}{2^m}\right )\right).$$
Changing variables $\bfh=A\,\bfu$ we have
$$I_0(\bfx)=
A^{\beta +2}\,\int _{\bbr ^2}\ex^{iA\,\bfx^\top \bfu}\,\eta _A(\bfu)\, {\rm d}\bfu.$$
We will prove that the \rhs\ vanishes as $A\to \infty$. Notice that 
$\eta _A(\bfu)=0$ if $\| \bfu\| \leq \frac{1}{2} \leq 
\frac{2^m}{A}$, or 
$\| \bfu\| >2\geq \frac{2^{m+1}}{A}$ since $2^m\leq A<2^{m+1}$. Hence
$\eta _A\in C_c^{\infty}(\bbr ^2)$ and for all $k\in \bbn$, 
$j=1,2$, 
$\sup _{A}|D_j^{k}\eta _A(\bfu)|<\infty,$ where
$D_j^{k}:=\partial ^k/\partial x_j^k$.
Indeed, 
$$
|D_j^{k}\eta _A(\bfu)|=\left | \sum _{p=0}^k \binom {k}{p} D_j^{k-p } (\| \bfu\| ^{\beta})\,D^{p}\eta _A(\bfu)\right |\leq C(k)\left (\frac{A}{2^m}\right )^k.$$
Then there is a constant $\wt C(k)$ such that for $\bfx\neq \bf0$,
\begin{equation}\label{decay1}
\left |\int _{\bbr ^2}\ex^{i\bfx^\top \bfu }\,\eta _A(\bfu)\,{\rm  d}\bfu\right |\leq \wt C(k)\|\bfx\|^{-k}\left (\frac{A}{2^m}\right )^k.
\end{equation}
This inequality follows directly from the following property of the Fourier transform 
$$
\mathcal{F} ^{-1} (\eta _A)(x)=(-ix_j)^{-k}\mathcal{F} ^{-1} (D_j^k\eta _A)(x),$$
see Theorems 7.4 and 7.5 in \cite{rudin:1973}.
Hence for $\bfx\neq \bf0$,
\beao
|I_0(\bfx)|
&\leq& \wt C(k)\,A^{-k}\|\bfx\|^{-k}\left (\frac{A}{2^m}\right )^kA^{2+\beta}\\
&\leq& \wt C(k)\|\bfx\|^{-k}\left (\frac{1}{2^m}\right )^k2^{(2+\beta)(m+1)}\\
&\leq& \wt C(k)\|\bfx\|^{-k}2^{(2+\beta)(m+1)-mk}\to 0\,,
\eeao
as $A\to \infty$, i.e., $m\to \infty$ provided $k> 2+\beta $.

\end{proof}

\section{A small simulation study} \label{sec:simu}  \setcounter{equation}{0}
\subsection{MMA random field}
We start with an MMA random field (see Section~\ref{subsec:mma}) with 
weight \fct\ 
\begin{equation}
\label{eq:6}
w(\bfs)=\phi^{|s_1|+|s_2|}\,\1(|s_1|+|s_2|\le 5)\,,
\end{equation}
where $\phi>0$, and $Z_{\bfs}$ have a unit Fr\'echet 
\ds . The random field $(X_\bft)$ is visualized in 
Figure~\ref{fig:1} on $\Lambda_{50}^2$. 
If the noise $Z_{\bft}$ is large and $\phi<
1$ the values $w(\bfs)Z_{\bft-\bfs}$ 
decrease quickly in the neighborhood of $\bft$. For $\phi\ge 1$ we observe the opposite effect.
Thus the local extremal dependence
of $(X_{\bft})$ is very weak for $\phi<1$ and very strong for
$\phi>1$.
In Figure~\ref{fig:1} we show sample paths of the MMA field 
for $\phi=0.5,1.0,1.5$ and in Figure~\ref{fig:2} boxplots
of Whittle estimation based on 50 replications. We
  choose $m_n=20$ \st\ $1-1/m_n=0.95$ and select 
$a_m$ as the $95\%$-quantile of the sample
  $(X_{\bft})_{\bft \in \Lambda_{50}^2}$.

\begin{figure}[htbp]
\centering
\begin{minipage}[t]{0.33\linewidth}
\includegraphics[width=\linewidth]{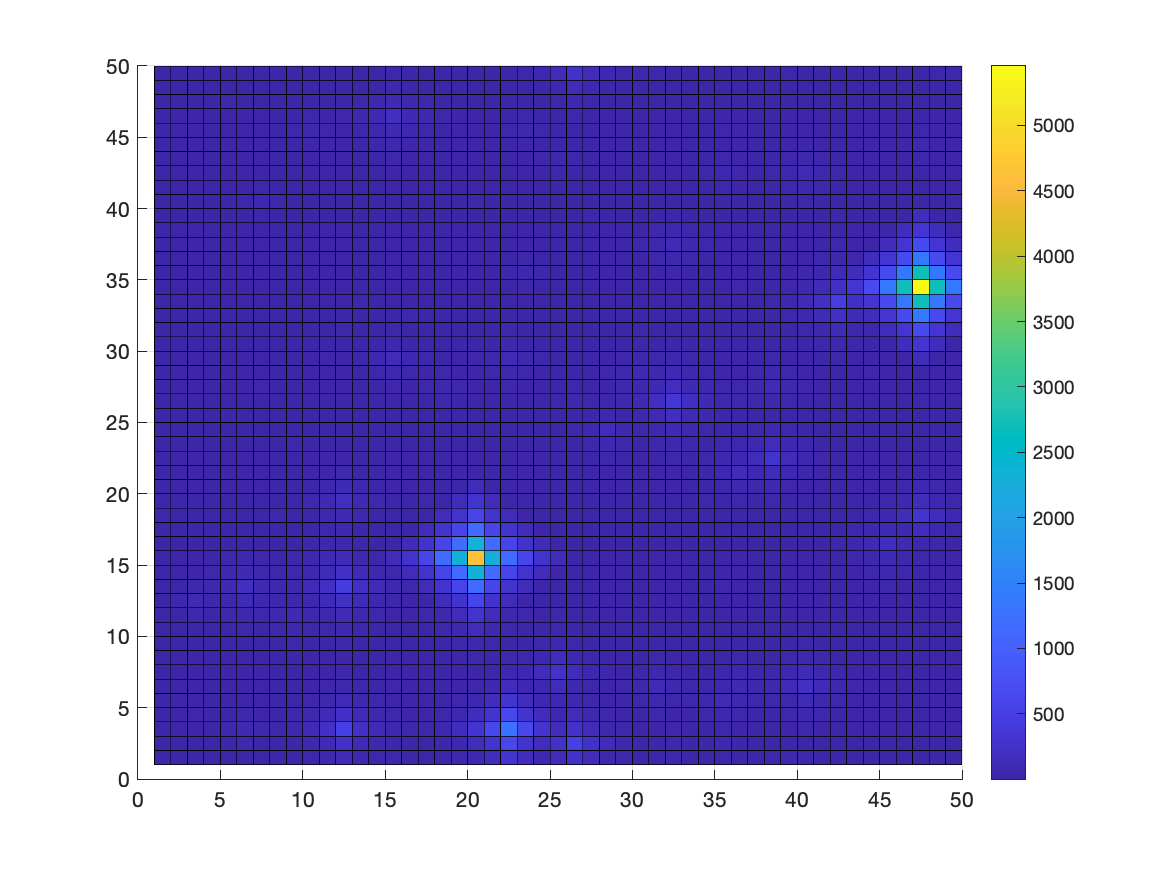}
\end{minipage}
\begin{minipage}[t]{0.33\linewidth}
\includegraphics[width=\linewidth]{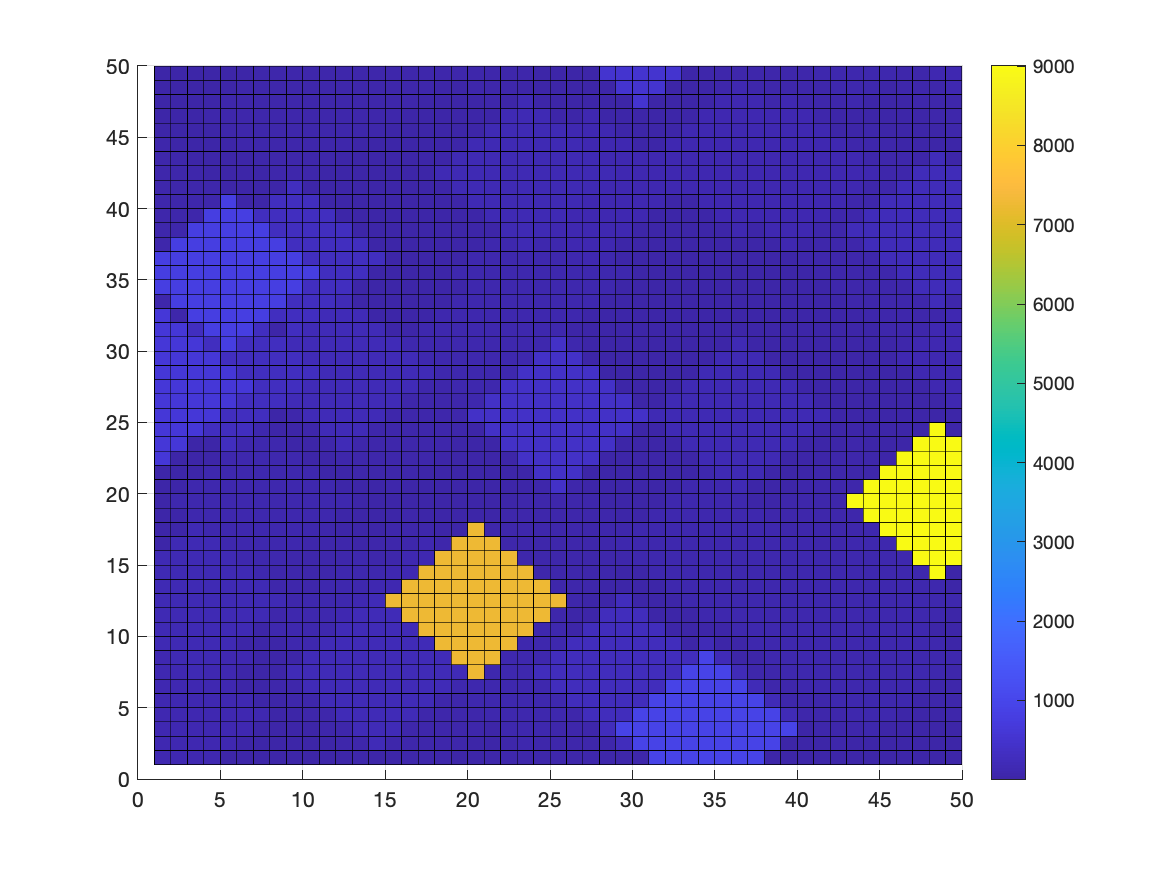}
\end{minipage}
\begin{minipage}[t]{0.33\linewidth}
\includegraphics[width=\linewidth]{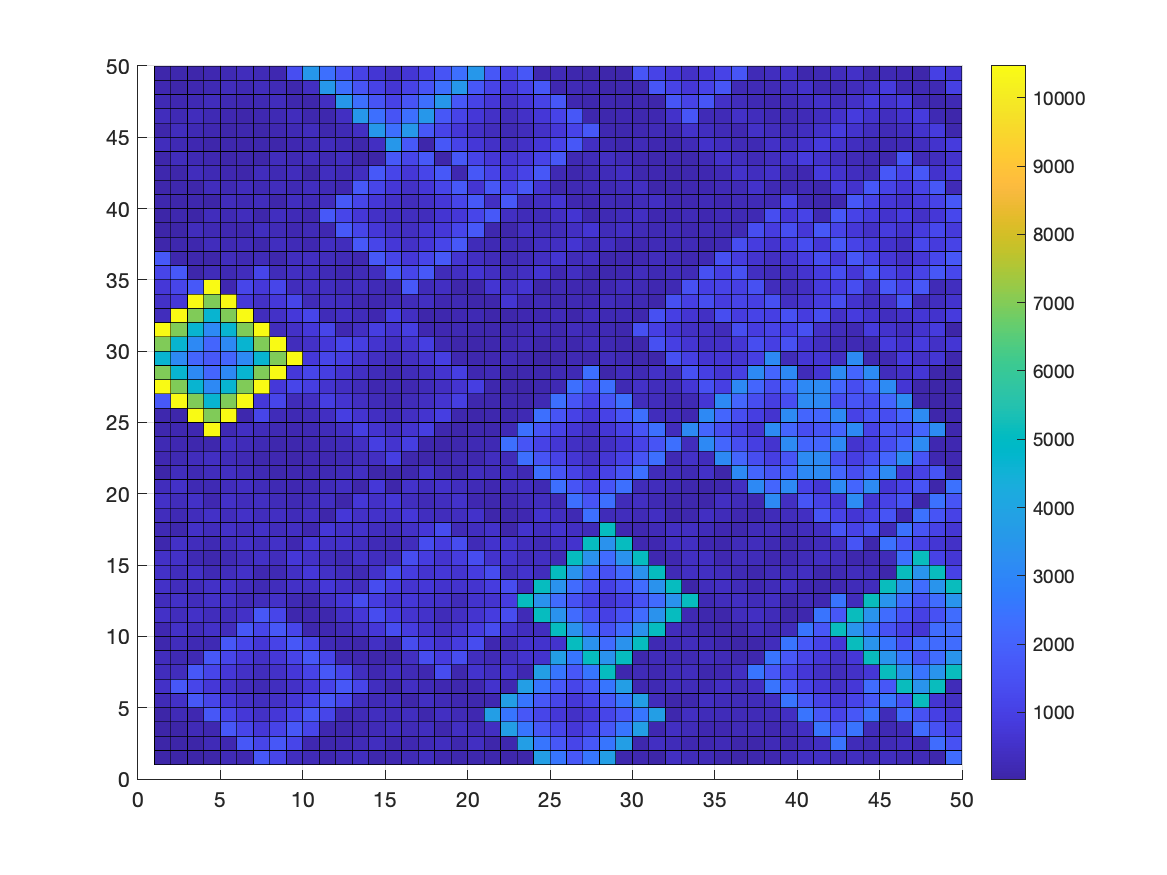}
\end{minipage}
\caption{\label{fig:1} Simulated sample paths of the MMA
  random field on $\Lambda_{50}^2$ with weight \fct\ \eqref{eq:6}, 
  $\phi =0.5$ (top left), $\phi=1.0$ (top right)  and $\phi=1.5$ (bottom).
  }
\end{figure}

\begin{figure}[htbp]
\centering
\begin{minipage}[t]{0.33\linewidth}
\includegraphics[width=\linewidth]{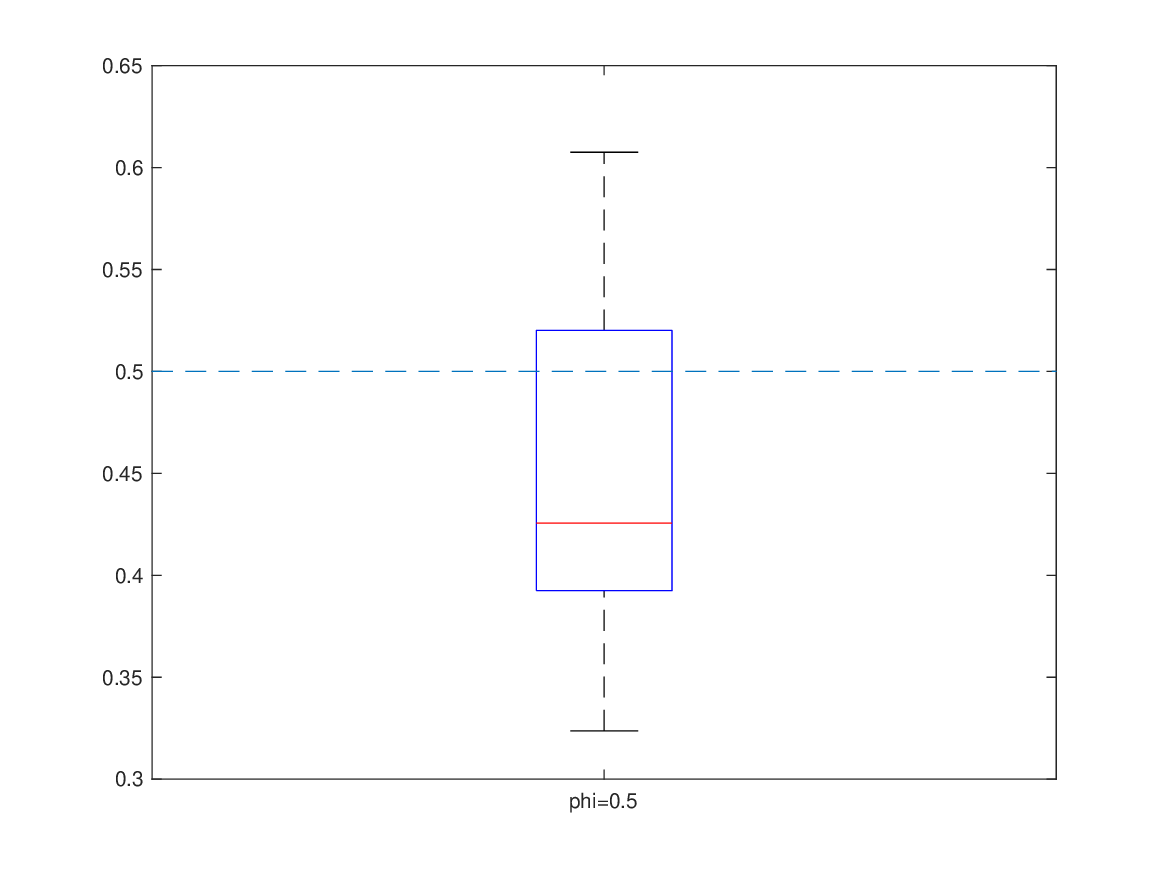}
\end{minipage}
\begin{minipage}[t]{0.33\linewidth}
\includegraphics[width=\linewidth]{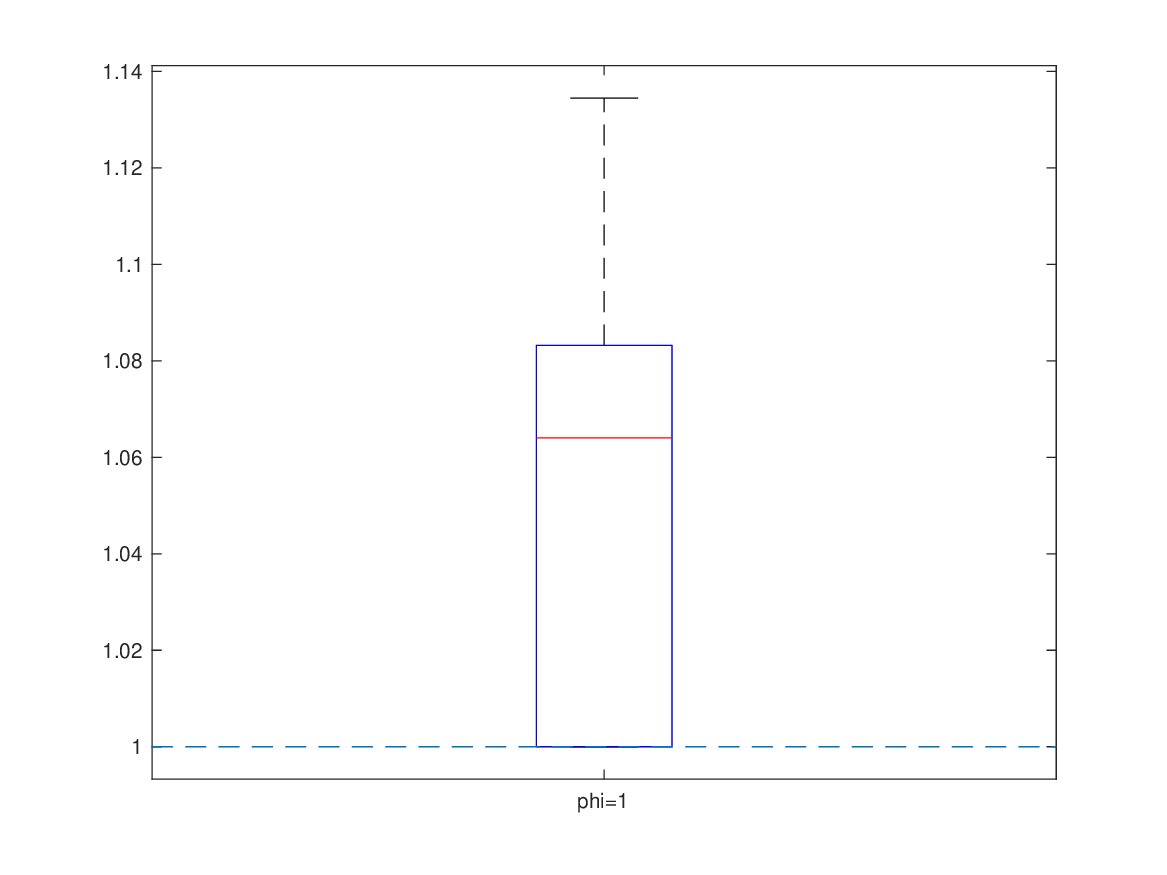}
\end{minipage}
\begin{minipage}[t]{0.33\linewidth}
\includegraphics[width=\linewidth]{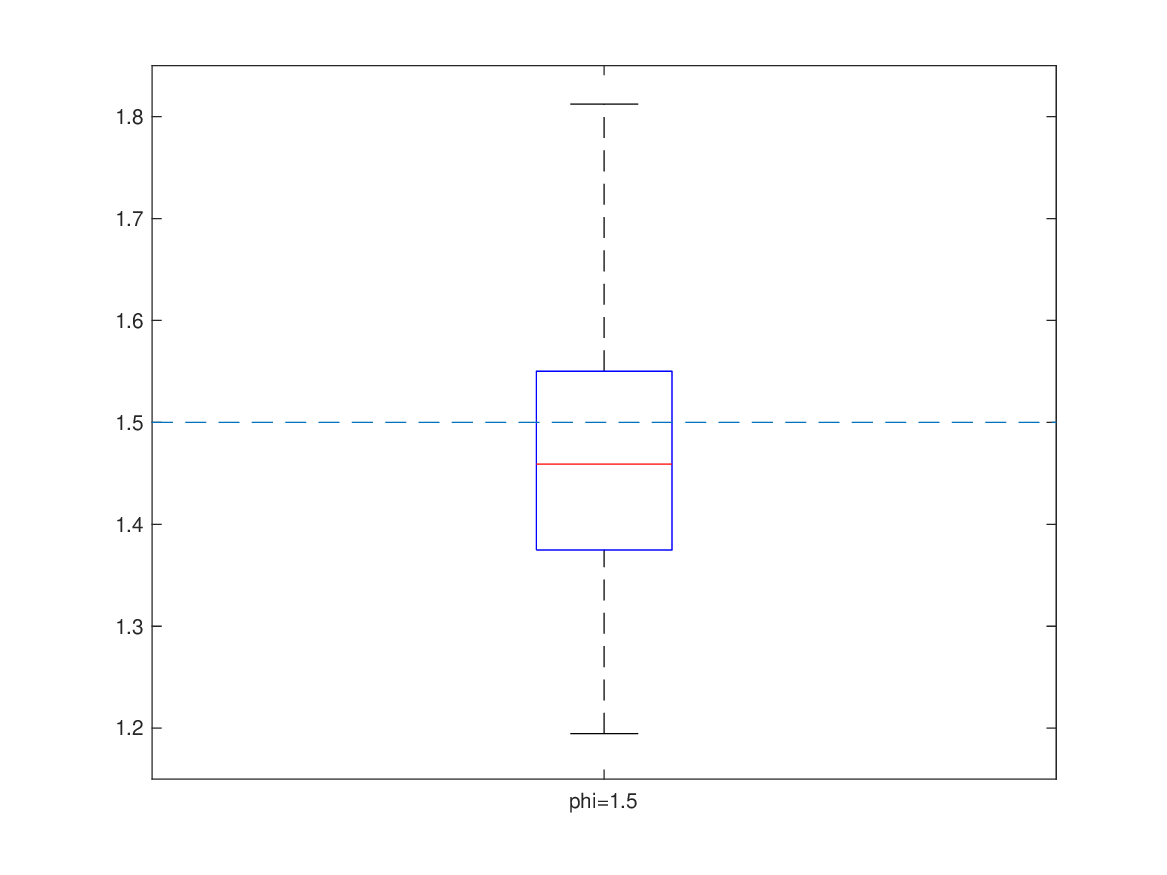}
\end{minipage}
\caption{\label{fig:2} Box-plot (50 replications) 
for MMA Whittle estimation  on $\Lambda_{50}^2$ 
of the parameter in model \eqref{eq:6}.
$\phi =0.5$ (top left), $\phi=1.0$ (top right)  and $\phi=1.5$ (bottom).
$a_m$ represents the 95\% empirical quantile.}
\end{figure}

\subsection{The Brown-Resnick random field}
In Figure~\ref{fig:brrf} we present a simulation of the Brown-Resnick random
field for $H=0.5$ on  $ \Lambda_{50}^2$. The simulation of these fields 
is complex; see the discussion in the recent overview paper 
Oesting and Strokorb \cite{oesting:strokorb:2021}. For the results in Figure~\ref{fig:brrf} 
and Table~\ref{tab:brrf} we use the algorithm
from Liu et al.  \cite{liu:blanchet:dieker:mikosch:2016} which leads to
a perfect simulation of the field. 
For the field on $\Lambda_{50}^2$ we illustrate the performance of
the Whittle estimator in boxplots based on 50 replications
for $m_{n} = 5, 10$ corresponding to the 80\%- and 90\%-quantiles $a_m$ of the 
marginal \ds\ of the field. For these 50 sample paths
we  also calculated the pairwise 
  composite likelihood estimator used in Davis et al. 
  \cite{davis:kluppelberg:steinkohl:2013}
 and showed the results in a boxplot.
Table~\ref{tab:brrf} provides the
 mean, median, standard deviation based on 50 replications for the Whittle
  and pairwise 
  composite likelihood estimators.
The pairwise
  likelihood estimator has a much smaller variance than the Whittle estimator.
On the other hand, pairwise likelihood estimation is very 
biased and costs much more time. 
\par
Figure~\ref{fig:brrf05}  shows a simulation of a truncated 
Brown-Resnick random field for $H=0.5$ based on the naive approximation
\begin{eqnarray}
  \label{eq:appbrf}
  X_{\bfs} = \sup_{1\le j \le 1000} \Gamma_j^{-1} \ex^{W_{\bfs}^{(j)} -
  \delta (\bfs)}\,, \quad \bfs \in \Lambda_{20}^2\,,
\end{eqnarray}
where $(\Gamma_j)$ is defined in \eqref{eq:brproc} which is
independent of the sequence of iid standard Brownian sheets
$(W_{\bfs}^{(j)})_{\bfs \in \Lambda_{20}^2}$. The performance of the
Whittle estimator is illustrated in boxplots based on $50$
replications for $m_n=3,5$ corresponding to the $67\%$- and
$80\%$-quantiles $a_m$ of the marginal distribution of the
field. Perhaps surprisingly, 
according to Table~\ref{tab:brrf05}, the choice
of $m_n=3$ provides better estimation results with smaller bias and variance.
For the same sample paths of the field we calculated the 
pairwise composite likelihood estimator. The boxplot in 
Figure~\ref{fig:brrf05} but also the results about  mean, median, standard
deviation in Table~\ref{tab:brrf05}
indicate that the Whittle estimator outperforms pairwise likelihood; the Whittle
procedure leads to estimators with smaller bias and variance.  
\par
Tables~\ref{tab:brrf} and \ref{tab:brrf05} show that
  the  Whittle estimator can be calculated at a much faster average speed than
the pairwise likelihood estimator. Moreover, the median  
of Whittle estimation is much closer to the true value $H=0.5$ than
for pairwise likelihood, showing that the former estimator is less biased.
\par
A delicate problem is the choice of a high quantile $a_m$.
In our simulations we calculated it as the corresponding empirical $(1-1/m)$-quantile. Then we determined the spatial \per\ based on 
the indicator \fct s $\1(|X_{\bfs}|>a_m)$, $\bfs\in \Lambda_n^2$, and
obtained the corresponding Whittle estimator by solving the optimization problem
   \eqref{eq:jan31a}. Throughout  the paper we use the \fct\ {\em
     fminbnd} in Matlab 2020a  for  
   optimization problems, also for pairwise likelihood.

\begin{figure}[htbp]
\begin{minipage}[t]{0.45\linewidth}
\centering
\includegraphics[width=\linewidth]{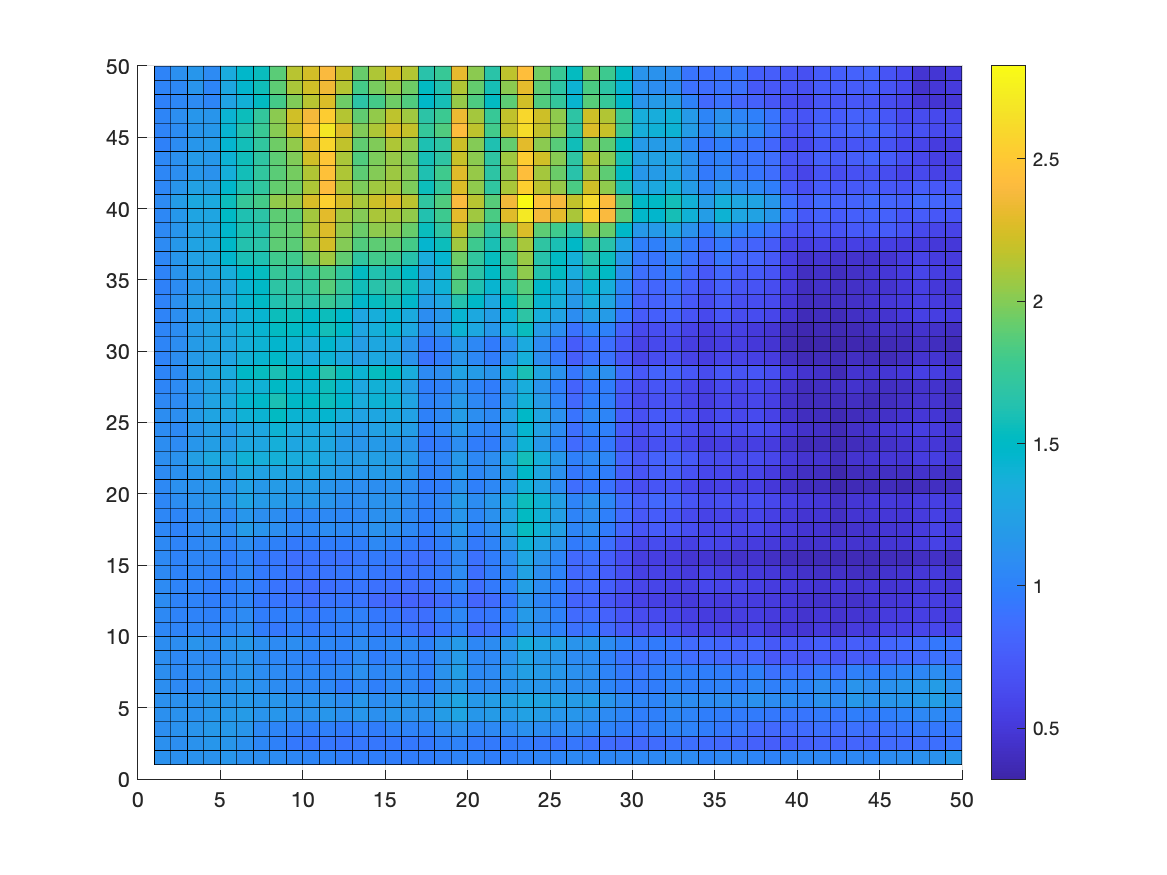}
\end{minipage}%
\begin{minipage}[t]{0.45\linewidth}
\centering
\includegraphics[width=\linewidth]{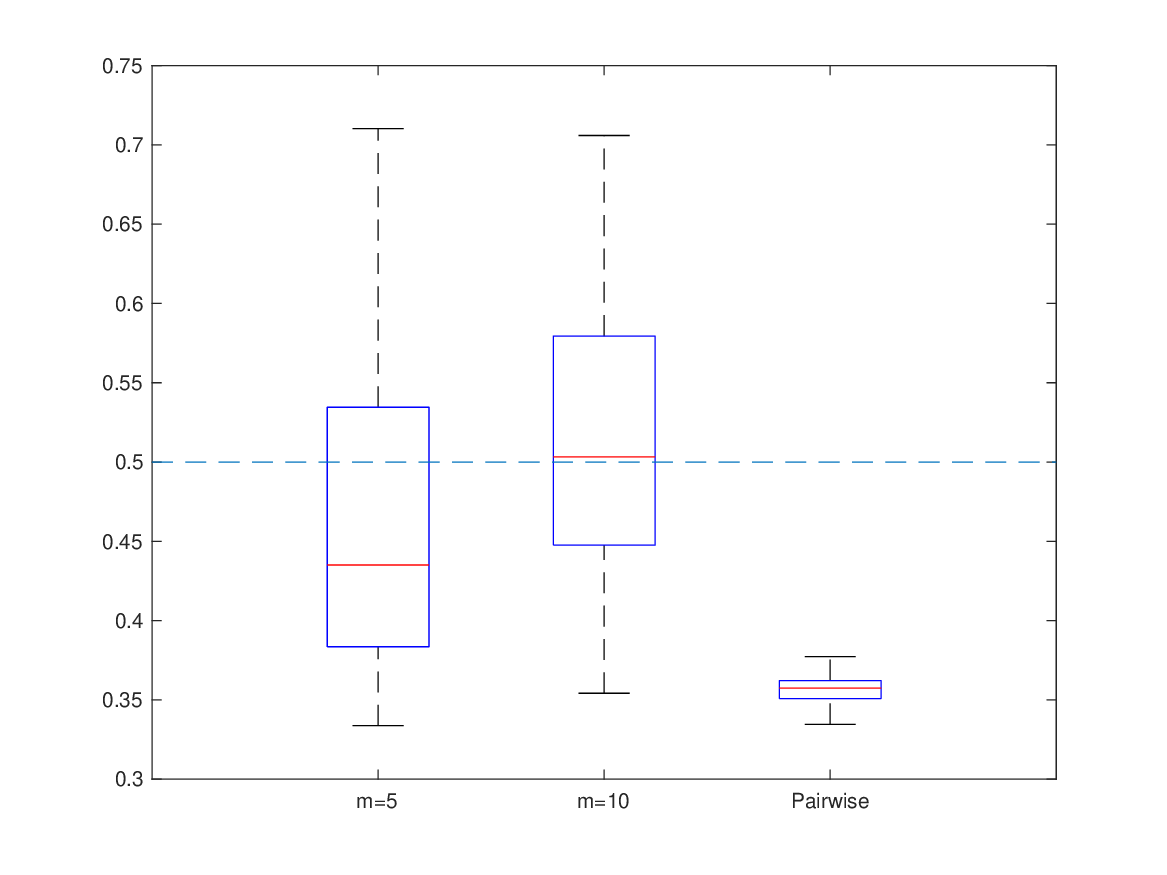}
\end{minipage}
%
\caption{Left:  A sample path of the Brown-Resnick random field
on $ \Lambda_{50}^2$ with $H=0.5$. The simulation is based on 
the algorithm from Liu et al. \cite{liu:blanchet:dieker:mikosch:2016}.
Right: Boxplots based on 50 replications of the Whittle estimator with
$m_n=5,10$ and pairwise likelihood estimator; see \cite{davis:kluppelberg:steinkohl:2013} for details.}\label{fig:brrf}
\end{figure}

\begin{table}[h!]
  \centering
  \begin{tabular}{|c|| c | c | c | c |}
    \hline
    Estimator & Mean & Median & Standard Deviation & Time cost (seconds) \\
    \hline
    \hline
    Whittle ($m=5$) & 0.49 & 0.43 & 0.15 & 319\\
    \hline
    Whittle ($m=10$) & 0.55 & 0.50 & 0.16 & 323 \\
    \hline
    Pairwise  & 0.36 & 0.36 & 0.01 & 3205\\
    \hline
  \end{tabular}
  \caption{Mean, median, standard deviation and average time cost in seconds for the Whittle
    estimator (first and second row) and the pairwise composite
    likelihood estimator (third row). The estimation (50 replications) was conducted 
  on a MacBook Pro (16-inch, 2019), CPU Intel Core i9 (2.4GHz) with
  Matlab 2020a. }
  \label{tab:brrf}
\end{table}
\begin{figure}[htbp]
\begin{minipage}[t]{0.45\linewidth}
\centering
\includegraphics[width=\linewidth]{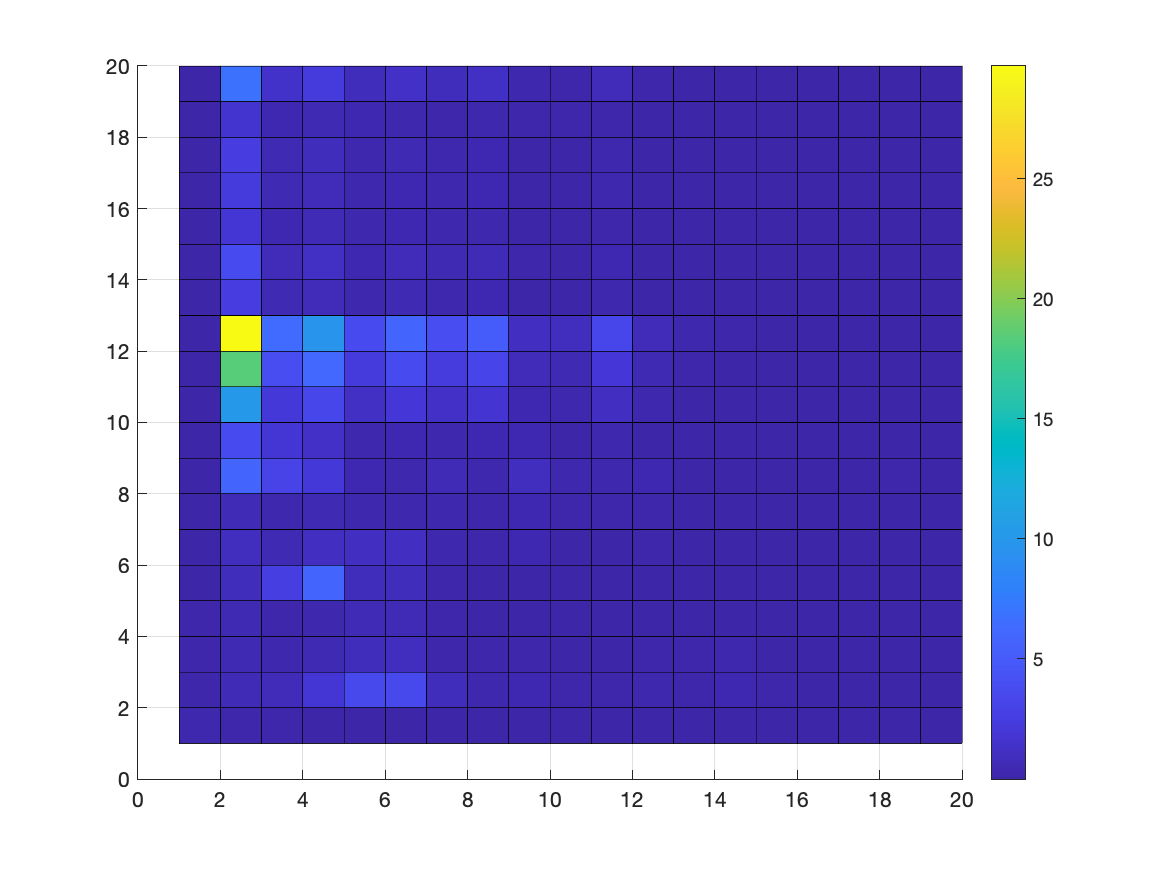}
\end{minipage}%
\begin{minipage}[t]{0.45\linewidth}
\centering
\includegraphics[width=\linewidth]{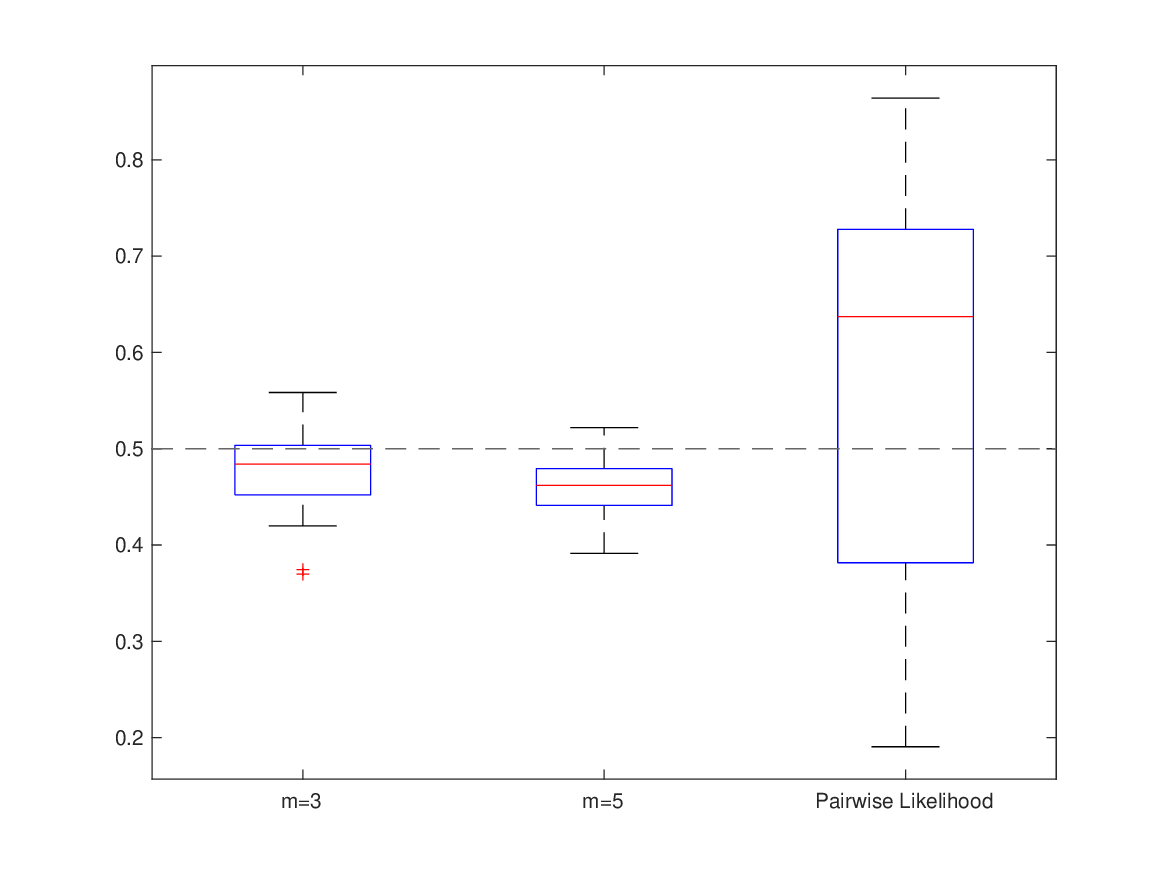}
\end{minipage}%
\caption{\label{fig:brrf05} Left: 
A sample path of the approximation \eqref{eq:appbrf}
to the Brown-Resnick random field 
on $ \Lambda_{20}^2$ with $ H=0.5$. 
Right: Boxplot based on 50 replications for Whittle estimator with $m_n=3,5$ and pairwise likelihood estimator; see 
\cite{davis:kluppelberg:steinkohl:2013} for details. } 
\end{figure}

\begin{table}[h!]
  \centering
  \begin{tabular}{|c|| c | c | c | c |}
    \hline
    Estimator & Mean & Median & Standard Deviation & Time cost (seconds) \\
    \hline
    \hline
    Whittle ($m=3$) & 0.47 & 0.48 & 0.037  & 146 \\
    \hline
    Whittle ($m=5$) & 0.46 & 0.46 & 0.027  & 160 \\
    \hline  
    Pairwise & 0.57 & 0.63 & 0.19 & 7946 \\
    \hline
  \end{tabular}
  \caption{Mean, median, standard deviation and average time cost in seconds for  Whittle
    estimator (first and second row) and pairwise composite
    likelihood estimator (third row). The estimation (50 replications) 
was conducted  on a 
  MacBook Pro (16-inch, 2019), CPU Intel Core i9 (2.4GHz) with Matlab 2020a.}
  \label{tab:brrf05}
\end{table}
\subsection{Comments}
Several problems  are not discussed in detail here.
One is the choice of the high quantile $a_m$ when dealing with real-life data.
Davis et al. \cite{davis:mikosch:2009,davis:mikosch:cribben:2012}
recommend a graphical approach by plotting the sample extremogram 
for various high empirical quantiles of the data at different lags.
If the sample extremogram collapses into zero already at small lags 
this is an indication of the fact that $a_m$ has been chosen too high.
Often 95\%--97\% quantiles of the data yield good results provided
the sample size is sufficiently large. In our simulation study
we did not choose very high quantiles and still got reasonable results. 
This may be due to the fact that the tails of the considered processes
are almost exact power laws even for small arguments.
\par
In the literature, estimation for Brown-Resnick processes/fields is often
conducted for processes with a spatio-temporal structure.
An extension of our results to these processes/fields is not straightforward
since we need positivity of the spectral density on the parameter space. 
For example, 
Theorem~\ref{thm:aux} yields the positivity of the  
spectral density $f_H$ of the Brown-Resnick process; the proof heavily depends
on the isotropy of the process.
\par
Another reason why we find it difficult to compare our results
with the literature on estimation of the Brown-Resnick process/field
is the use of distinct simulation procedures. Often the naive approximation \eqref{eq:appbrf} is employed, or the simulation technique is not explicitly mentioned. Oesting and Strokorb 
\cite{oesting:strokorb:2021} give an overview on simulation techniques
for max-stable processes and point at various problems in this context, in 
particular when using \eqref{eq:appbrf}.
\par
One of the tasks to be solved in the future is to find
confidence bands tailored for Whittle estimation. Motivated by
results in Davis et al. \cite{davis:mikosch:cribben:2012,davis:drees:segers:warchol:2018} for the estimation of the extremogram we expect that stationary and 
multiplier block bootstrap techniques may be suitable in this context.

\section{Proof of
  Theorem~\ref{thm:cltintper}}\label{sec:cltintper}\setcounter{equation}{0}
Write
\beao
\psi_{\bfh}^{(n)} = \frac{(2\pi)^2}{n^2} \sum_{\bfj\in\Lambda_n^2} \cos
  (\bfh^{\top}\bfla_{ \bfj}) \,g(\bfla_{ \bfj})\quad \mbox{ and }\quad
\psi_\bfh=  \int_{\Pi^2} \cos (\bfh^{\top } \bfomega) g(\bfomega) \dif \bfomega\,.
\eeao
Then we have 
\beao
\dfrac{(2\pi)^2}{n^2}
\sum_{\bfj\in\Lambda_n^2} (\wh f(\bfla_\bfj)-f(\bfla_\bfj))\,g(\bfla_\bfj)
&=& \sum_{\|\bfh\|< n} \psi_\bfh ^{(n)}(\wh \gamma(\bfh)
-\gamma(\bfh)) 
\eeao
\par
Since $g$ is continuous on $\Pi^2$ we have for fixed $\bfh$,
$\psi_\bfh^{(n)}\to \psi_\bfh$ and $\sup_{\bfh}|\psi_\bfh^{(n)}|<\infty$. 
Then by virtue of Theorem~\ref{thm:clt},
\beao
\frac{n}{\sqrt{m_n}}\sum_{\bfh:\| \bfh \|\le h } \psi_\bfh^{(n)} (\wh{\gamma}
(\bfh) - \gamma(\bfh)) \std \sum_{\bfh:\|\bfh \|\le h} \psi_\bfh\,Z_\bfh\,,\qquad \nto\,.
\eeao
We will show that for all $\vep>0$,
\beam\label{eq:april13a}
  \lim_{h\to \infty} \limsup_{n\to \infty} \P \Big(
  \frac{n}{\sqrt{m_n}} \Big|\sum_{h<\|\bfh\|<n} \psi_{\bfh}^{(n)}
  (\wh{\gamma}(\bfh ) - \gamma(\bfh)) \Big| >\varepsilon
  \Big) = 0\,.
  \eeam
Then, by Theorem~2 in Dehling et al. \cite{dehling:durieu:volny:2009},
it follows that 
\beao
\dfrac{n}{\sqrt{m_n}}\,\dfrac{(2\pi)^2}{n^2}\sum_{\bfj\in\Lambda_n^2} (\wh f(\bfla_\bfj)-f(\bfla_\bfj))\,g(\bfla_\bfj)\std  \sum_{\bfh\in\bbz^2} \psi_\bfh\,Z_\bfh\,,\qquad \nto\,,
\eeao
and the limit is a genuine real-valued \rv .
\par
It remains to prove the following result.
\ble\label{lem:april13b} Under the conditions of the theorem we have
\beao
  \lim_{h\to \infty} \limsup_{n\to \infty} \dfrac{n}{\sqrt{m_n}} \Big| \sum_{h <\| \bfh \| <n} \psi_{\bfh}^{(n)}(\E[\wh{\gamma}(\bfh )] - \gamma(\bfh)) \Big|
  &=& 0\,.
\eeao
\ele
\begin{proof}[Proof of Lemma~\ref{lem:april13b}]
By Lemma~\ref{lem:psih}   we have $|\psi_{\bfh}^{(n)}-\psi_\bfh|\le cn^{-1}$
for $h_1,h_2\le n/2$.
Hence
 \beao
\lefteqn{\dfrac{n}{\sqrt{m_n}}\, \Big| \sum_{h <\| \bfh \|
  \le r_n} \psi_{\bfh}^{(n)}\,(\E[\wh{\gamma}(\bfh )] - \gamma(\bfh)) \Big|}\\
  &\le &\dfrac{n}{\sqrt{m_n}}\, \Big| \sum_{h <\| \bfh \|
  \le r_n} (\psi_{\bfh}^{(n)}-\psi_\bfh)\,(\E[\wh{\gamma}(\bfh )] - \gamma(\bfh)) \Big|
+\dfrac{n}{\sqrt{m_n}}\, \Big| \sum_{h <\| \bfh \|
  \le r_n} \psi_\bfh \,(\E[\wh{\gamma}(\bfh )] - \gamma(\bfh)) \Big|\\
&\le &\dfrac{r_n^2}{\sqrt{m_n}}\, c\, 
 \sup_{h<\|\bfh\|\le r_n}|\E[\wh{\gamma}(\bfh )] - \gamma(\bfh)|
+\sum_{h <\| \bfh \|
  \le r_n} |\psi_\bfh| \,\dfrac{n}{\sqrt{m_n}}\sup_{h<\|\bfh\|\le r_n}\big|\E[\wh{\gamma}(\bfh )] -\gamma(\bfh)\big|\,.
\eeao
The \rhs\ vanishes as $\nto$ 
in view of {\bf (M2)} and the absolute summability of $(\psi_\bfh)$.
Next we consider 
\beao
\lefteqn{\dfrac{n}{\sqrt{m_n}}\, \Big| \sum_{r_n<\| \bfh \|
  <n} \psi_{\bfh}^{(n)}\,(\E[\wh{\gamma}(\bfh )] - \gamma(\bfh)) \Big|}\nonumber\\
&\le & \dfrac{n}{\sqrt{m_n}}\,c\, \sum_{r_n<\| \bfh\|
  <n}
|\psi_{\bfh}^{(n)}|\,\big(m_n\,|\P(|X_{\bf0}|>a_m\,,|X_\bfh|>a_m)-(\P(|X|>a_m))^2|
+\gamma(\bfh)\big)\\
& \le & n\, m_n^{1/2}\, c\,\sum_{r_n<\| \bfh\|
  <n}\alpha(\|\bfh\|) + 
\dfrac{n}{\sqrt{m_n}} c\,\sum_{r_n<\|\bfh\|<n} \gamma(\bfh)\\
&\le & n\,m_n^{1/2}\, c\,\alpha(r_n) + 
\dfrac{n}{\sqrt{m_n}} c\,\sum_{r_n<\|\bfh\|<n} \gamma(\bfh)\to 0\,,\qquad \nto\,.
\eeao 
In the last step we used condition \eqref{eq:newlabel}. This finishes  the proof.
\end{proof}
By virtue of Lemma~\ref{lem:april13b} we are allowed to replace 
all $\gamma(\bfh)$ in \eqref{eq:april13a} by the corresponding 
expectations $\E[\wh \gamma(\bfh)]$. Therefore we will show next that
the following quantities are \asy ally negligible for all fixed $\vep>0$:

\begin{eqnarray*}
Q_1&= &\P \Big(
     \frac{n}{\sqrt{m_n}} \Big| \sum_{\bfh:h < \|\bfh \| \le 3r_n}
     \psi_{\bfh}^{(n)}(\wh{\gamma} (\bfh) - \E [
     \wh{\gamma} (\bfh)]) \Big| >\varepsilon\Big)\,,\\
 Q_2&= &\P \Big(
     \frac{n}{\sqrt{m_n}} \Big| \sum_{\bfh:3 r_n < \|\bfh \| \le n}
     \psi_{\bfh}^{(n)}(\wh{\gamma} (\bfh) - \E [
     \wh{\gamma} (\bfh)]) \Big| >\varepsilon\Big)\,.\\
\end{eqnarray*}
By Markov's inequality and Lemma~\ref{lem:psih},3., also using the 
fact that $r_n\le n/2$, we have
\beao
Q_1&\le & \dfrac{2\,n}{\sqrt{m_n}} \sum_{\bfh:h<\|\bfh\|\le 3r_n} |\psi_\bfh^{(n)}|
\,\E[|\wh\gamma(\bfh)|]\\
&\le &c \,\sqrt{m_n} \,\sum_{\bfh:h<\|\bfh\|\le 3r_n} \big(\P(|X_{\bf0}|>a_m\,,|X_\bfh|>a_m)+(\P(|X|>a_m))^2\big)\\
&= &c \,\sqrt{m_n} \,\sum_{\bfh:h<\|\bfh\|\le 3r_n} \P(|X_{\bf0}|>a_m\,,|X_\bfh|>a_m)
+O(r_n^2m_n^{-3/2})\,.
\eeao
The \rhs\ vanishes by condition {\bf (M1)}, first letting $\nto$, then $h\to \infty$.

The negligibility of $Q_2$ is proved in the following lemma. Its proof is given in 
Appendix~\ref{sec:covv}.
\ble\label{lem:q2neg}
  Under the conditions of the theorem we have $\lim_{n\to \infty} Q_2 =0$.
\ele
This completes the proof.

\section{Proof of Theorem~\ref{thm:cltwhittle}}\setcounter{equation}{0}
We start with several auxiliary results.
\ble\label{lem:123}
Assume {\bf (W)}. Then the following statements hold:
\begin{enumerate}
\item[\rm 1.] For all $\Theta \in \bfTh$,  
$\Theta\ne \Theta_0$,
\begin{eqnarray}\label{eq:noone}
  \dfrac {1}{(2\pi)^2} \dfrac{\ov\sigma^2 (\Theta)}{\ov\sigma^2 (\Theta_0)}
\,\int_{\Pi^2} \frac{f_{\Theta_0}
  (\bfomega)}{ f_{\Theta}
  (\bfomega)} \dif \bfomega > 1\,,\mbox{ where  } \overline{\sigma}^2 (\Theta)= \exp\Big(\dfrac 1{(2\pi)^2}\int_{\Pi^2} \log f_\Theta(\bfomega)\,d\bfomega\Big)\,.
\end{eqnarray}

\item[\rm 2.] We have
\beao
 \sup_{\Theta\in\bfTh}\Big|\sigma_n^2 (\Theta) - 
\dfrac{\ov\sigma^2(\Theta)}{(2\pi)^2}\int_{\Pi^2} \dfrac{f_{\Theta_0}
(\bfomega)}{f_{\Theta}(\bfomega)}\,d\bfomega\Big|\to 0\,,\qquad \nto\,.
\eeao
\item[\rm 3.]
We have $\Theta_n\stp \Theta_0$ and 
$\sigma_n^2 (\Theta_n) \stp\overline{\sigma}^2 (\Theta_0)$.
\end{enumerate}
\ele
The proof is given at the end of this section.
\par
Part 1. says that the left-hand side of \eqref{eq:noone} is bounded away from 1.  
This fact is essentially responsible for the {\em uniqueness}
of the Whittle estimator.
Part 2. is key to showing the {\em \asy\ unbiasedness}
of the Whittle estimator.
Part 3. yields the {\em consistency} of the Whittle estimator.
\begin{proof}[Proof of Theorem~\ref{thm:cltwhittle}] 
Let 
\beam\label{eq:feb5a}
F_n(\bfomega, \Theta) = \dfrac{f_{\Theta } (\bfomega)}{\overline{\sigma}_n^2 ( \Theta)}\quad\mbox{and}\quad F(\bfomega, \Theta) =
\dfrac{f_{\Theta} (\bfomega)}{ \overline{\sigma}^2 (\Theta)}\,,\qquad  (\bfomega,
\Theta)\in \Pi^2 \times \bfTh\,.
\eeam 
By Kolmogorov's formula (see Brockwell and 
Davis \cite{brockwell:davis:1991}, Theorem~5.8.1),
\begin{eqnarray} \label{eq:oo}
  \sum_{\bfj \in\Lambda_n^2} \log F_n (\bfla_{\bfj}, \Theta) = \int_{\Pi^2}
  \log F(\bfomega, \Theta) \dif \bfomega
  =0\,, \quad \Theta \in \bfTh\,.
\end{eqnarray}
We further study the expressions of $\partial F_n(\bfomega,
\Theta)/\partial \theta_{s_1}$ and $\partial^2 F_n(\bfomega,
\Theta)/( \partial \theta_{s_1} \partial \theta_{s_2} )$ for $
s_i\in \{ 1, \ldots, s\}$. 
\begin{eqnarray*}
  \frac{\partial F_n (\bfomega, \Theta)}{\partial \theta_{s_1}}
  &= &  \dfrac1{\overline{\sigma}_n^2 (\Theta)} \frac{\partial
       f_{\Theta}(\bfomega)}{ \partial \theta_{s_1}} -
       \frac{f_{\Theta}(\bfomega )}{(\overline{\sigma}^2
       (\Theta))^2} \frac{\partial \overline{\sigma}_n^2
       (\Theta)}{\partial \theta_{s_1}}\\ 
 & = & \frac{1}{\overline{\sigma}_n^2 (\Theta)} \frac{\partial
       f_{\Theta} (\bfomega)}{ \partial \theta_{s_1}} - \frac{
       f_{\Theta}(\bfomega)}{  \overline{\sigma}_n^2 (\Theta)} \Big(
       n^{-2} \sum_{\bfj \in \Lambda_n}
       (f_{\Theta}(\lambda_{\bfj}))^{-1} \frac{\partial
       f_{\Theta}(\lambda_{\bfj})}{\partial \theta_{s_1}} \Big) \,.\\
\frac{\partial^2 F_n (\bfomega, \Theta)}{\partial \theta_{s_1}
  \partial \theta_{s_2}}
  & =&  \frac{1}{\overline{\sigma}_n^2 (\Theta)} \frac{\partial^2
       f_{\Theta} (\bfomega)}{ \partial \theta_{s_1} \partial
       \theta_{s_2}} - (\overline{\sigma}_n^2 (\Theta))^{-1}
       \frac{\partial f_{\Theta}(\bfomega)}{\partial \theta_{s_1}} \Big(
       n^{-2} \sum_{\bfj \in \Lambda_n}
       (f_{\Theta}(\lambda_{\bfj}))^{-1} \frac{\partial
       f_{\Theta}(\lambda_{\bfj})}{\partial \theta_{s_2}} \Big)\\ 
 & & + \frac{ f_{\Theta}(\bfomega)}{\overline{\sigma}_n^2 (\Theta)} \Big(
       n^{-2} \sum_{\bfj \in \Lambda_n}
       (f_{\Theta}(\lambda_{\bfj}))^{-1} \frac{\partial
       f_{\Theta}(\lambda_{\bfj})}{\partial \theta_{s_1}} \Big) \Big(
       n^{-2} \sum_{\bfj \in \Lambda_n}
       (f_{\Theta}(\lambda_{\bfj}))^{-1} \frac{\partial
       f_{\Theta}(\lambda_{\bfj})}{\partial \theta_{s_2}} \Big)\\ 
 & & +  \frac{ f_{\Theta}(\bfomega)}{\overline{\sigma}_n^2 (\Theta)}  \Big(
       n^{-2} \sum_{\bfj \in \Lambda_n}
       (f_{\Theta}(\lambda_{\bfj}))^{-1} \frac{\partial^2
       f_{\Theta}(\lambda_{\bfj})}{\partial \theta_{s_1}\partial \theta_{s_2}} \Big)\,. 
\end{eqnarray*}
In view of \eqref{eq:new41}, \eqref{eq:new42} the
first- and second-order derivatives of $F_n (\bfomega, \Theta)$ 
and $F(\bfomega, \Theta)$ with respect to $\Theta$
exist.
Moreover, \eqref{eq:march1a} implies that 
  $\overline{\sigma}^2 (\Theta)$ and $ \overline{\sigma}_n^2 (
  \Theta)$ are uniformly bounded in $\Theta$ and thus
both $F_n(\bfomega, \Theta)$, $F(\bfomega, \Theta)$ are bounded away from zero
uniformly for $(\bfomega, \Theta) \in \Pi^2 \times \bfTh$.
Consequently, the first- and second-order derivatives of
$1/F_n(\bfomega , \Theta)$ and
$1/F(\bfomega, \Theta)$ are well defined.
Taking the derivatives of both terms in \eqref{eq:oo}, we have
\begin{eqnarray}
  {\bf0}_{1\times s}
  &=&   \sum_{\bfj\in\Lambda_n^2}
      \frac{\partial F_n(\bfla_{\bfj}, \Theta)}{ \partial \Theta}
      \frac{1}{F_n(\bfla_{\bfj}, \Theta)} = \sum_{\bfj \in\Lambda_n^2}
      f_{\Theta}(\bfla_{\bfj}) \frac{\partial (1/
      F_n(\bfla_{\bfj}, \Theta))}{ \partial \Theta}\nonumber\\
 & = & \int_{\Pi^2} \dfrac 1{F (\bfomega, \Theta)} \frac{\partial
       F(\bfomega, \Theta)}{\partial \Theta} \dif \bfomega =
       \int_{\Pi^2} f_{\Theta}(\bfomega) \frac{\partial(1/
       F(\bfomega , \Theta))}{\partial \Theta} \dif \bfomega \,,\nonumber\\
&&\label{eq:zeros000}\\
  {\bf0}_{s\times s}
  & = &  \sum_{\bfj\in\Lambda_n^2} \dfrac 1{F_n(\bfla_{\bfj}, \Theta)}
        \frac{\partial^2 F_n(\bfla_{\bfj}, \Theta)}{\partial
        \Theta^2} 
- \sum_{\bfj\in\Lambda_n^2} \bigg( \dfrac 1{F_n(\bfla_{\bfj},
               \Theta)}\frac{\partial F_n(\bfla_{\bfj}, 
        \Theta)}{\partial \Theta}  \bigg)^{\top}
               \bigg(\dfrac 1{F_n(\bfla_{\bfj}, \Theta)} \frac{\partial
               F_n(\bfla_{\bfj},   \Theta)}{\partial \Theta} \bigg)\nonumber\\
  & =  & \sum_{\bfj\in\Lambda^2} \dfrac 1{F_n(\bfla_{\bfj}, \Theta)}
        \frac{\partial^2 F_n(\bfla_{\bfj}, \Theta)}{\partial
        \Theta^2} - \sum_{\bfj\in\Lambda_n^2} \bigg( \frac{\partial \log F_n(\bfla_{\bfj}, \Theta)}{\partial \Theta}  \bigg)^{\top}
               \bigg(\frac{\partial \log F_n(\bfla_{\bfj}, \Theta)}{\partial \Theta}  \bigg)\nonumber\\ 
 & = & \int_{\Pi^2}\dfrac 1{ F(\bfomega, \Theta)} \frac{\partial^2
       F(\bfomega, \Theta)}{\partial \Theta^2} \dif \bfomega -
       \int_{\Pi^2} \bigg(\frac{\partial  \log F(\bfomega, \Theta)
       }{\partial \Theta}  \bigg)^{\top} \bigg(\frac{\partial  \log F(\bfomega, \Theta)
       }{\partial \Theta}  \bigg) \dif \bfomega\,.\nonumber\\
\label{eq:zeros001}
\end{eqnarray}
By definition of the Whittle estimator $\Theta_n$ as a minimizer,
$\partial \sigma_n^2 (\Theta_n)/\partial \Theta ={\bf0}_{1\times
  s}$. Then a  Taylor expansion of $\partial \sigma^2_n (\Theta) /\partial
\Theta$ about $\Theta=\Theta_n$ yields 
\begin{eqnarray}\label{eq:firstequation}
   \frac{\partial \sigma^2_n (\Theta_0)}{ \partial
  \Theta}
  &=&  \frac{\partial \sigma^2_n (\Theta_n)}{
      \partial \Theta} -  \frac{\partial^2
      \sigma^2_n (\Theta_n^+)}{\partial \Theta^2} \big( \Theta_n
      -\Theta_0 \big)
  =   -  \frac{\partial^2
      \sigma^2_n (\Theta_n^+)}{\partial \Theta^2} \big( \Theta_n
      -\Theta_0 \big)\,,
\end{eqnarray}
for some $\Theta_n^+ \in \bfTh$ with $\Theta_n^+ \stp \Theta_0$; see 
Lemma~\ref{lem:123},3. By Lemma~\ref{prop:convderivates} and
\eqref{eq:zeros001} we have
\beam\label{eq:deriveconv2}
  \frac{\partial^2 \sigma^2_n (\Theta_n^+)}{ \partial \Theta^2}
  &\stp&
 \frac{\overline{\sigma}^2 (\Theta_0) }{(2 \pi)^2} \int_{\Pi^2 } \dfrac 1{F
                       (\bfomega, \Theta_0)} \frac{\partial^2
                       F(\bfomega, \Theta_0)}{ \partial \Theta^2} \dif
                      \bfomega\\ 
  \nonumber
 &   =& \frac{\overline{\sigma}^2 (\Theta_0) }{(2 \pi)^2} \int_{\Pi^2 }
      \bigg( \frac{\partial \log F(\bfomega, \Theta_0)}{ \partial
      \Theta} \bigg)^{\top} \bigg( \frac{\partial \log F(\bfomega,
      \Theta_0)}{ \partial \Theta} \bigg) \dif \bfomega
     =: \frac{\overline{\sigma}^2 (\Theta_0) }{(2 \pi)^2} \wh W\,.\nonumber
     \eeam
In fact, we have for $\bfU$ uniform on $\Pi^2$,
\begin{eqnarray*}
  \dfrac{\widehat{W}}{(2\pi)^2}
   &=&\dfrac{1}{(2\pi)^2} \int_{\Pi^2} \Big[\Big(\frac{\partial \log
       f_{\Theta_0}(\bfomega)}{\partial \Theta} -
       \frac{1}{(2\pi)^2}\int_{\Pi^2} \frac{\partial \log
       f_{\Theta_0}(\bfomega)}{\partial \Theta}\dif \bfomega \Big)^{\top}\\&&\hspace{0.8cm}\Big(\frac{\partial \log
       f_{\Theta_0}(\bfomega)}{\partial \Theta} -
       \frac{1}{(2\pi)^2}\int_{\Pi^2} \frac{\partial \log
       f_{\Theta_0}(\bfomega)}{\partial \Theta}\dif \bfomega \Big)
\Big]       {\rm d}\bfla\\ 
 & = & \var \Big(\frac{ \partial \log f_{\Theta_0} (\bfU)}{ \partial \Theta} \Big)\,.
\end{eqnarray*}
By condition {\bf (W)}, $\wh W$ is non-singular. Together with 
\eqref{eq:deriveconv2} this implies that 
$\partial^2 \sigma^2_n (\Theta_n^+)/ \partial \Theta^2$ is positive definite on a set $B_n$ with 
$\P(B_n)\to1$.
The definition of the Whittle likelihood $\sigma^2_n(\Theta)$ and
\eqref{eq:zeros000} yield 
\beao
\frac{\partial \sigma^2_n (\Theta_0)}{ \partial
  \Theta}  =   \sum_{\bfj\in\Lambda_n^2}
                \frac{\partial (1/F_n
                (\bfla_{\bfj}, \Theta_0))}{\partial \Theta}\, \wh{f}(\bfla_{\bfj} ) =  \sum_{\bfj\in\Lambda_n^2} \frac{\partial (1/F_n
       (\bfla_{\bfj}, \Theta_0))}{ \partial \Theta} \big(
       \wh{f} (\bfla_{\bfj}) - f_{\Theta_0}(\bfla_{\bfj})
        \big)\,. 
\eeao
Combining this relation with
\eqref{eq:firstequation}, we have 
\beao
\frac{n}{\sqrt{m_n}}(  \Theta_n -\Theta_0) 
& = &-  \bigg(\frac{\partial^2 \sigma_n^2
      (\Theta_n^+)}{ \partial \Theta^2} \bigg)^{-1}  \frac{n}{\sqrt{m_n}}\sum_{\bfj\in\Lambda_n^2}
      \dfrac 1{F^2_n(\bfla_{\bfj}, \Theta_0)} 
\frac{\partial F_n(\bfla_{\bfj},
      \Theta_0)}{\partial \Theta} \big( \wh{f} (\bfla_{\bfj})
      - f_{\Theta_0}(\bfla_{\bfj})\big) \nonumber\\
  &=&- \wh W^{-1} 
\dfrac{(2\pi)^2}{\ov\sigma^2(\Theta_0)}(1+o_\P(1))\\&&
\times \frac{n}{\sqrt{m_n}}\sum_{\bfj \in\Lambda_n^2} 
\dfrac 1 {F_n^2(\bfla_{\bfj},
      \Theta_0)} 
\frac{\partial
      F_n(\bfla_{\bfj}, \Theta_0)}{\partial \Theta} \big(
      \wh{f} (\bfla_{\bfj}) -f_{\Theta_0}(\bfla_{\bfj}) \big)\,.
\eeao
Then, uniformly for $\bfomega \in \Pi^2$,
\begin{eqnarray*}
 \frac{\partial F_n (\bfomega, \Theta_0)}{ \partial \Theta}
 & = & - \frac{f_{\Theta_0} (\bfomega ) }{\big(\overline{\sigma}_n^2
       (\Theta_0)\big)^2} \frac{\partial \overline{\sigma}_n^2 (\Theta_0)
       }{\partial \Theta} = - \frac{f_{ \Theta_0}(\bfomega ) }{
       \overline{\sigma}_n^2 (\Theta_0)} \frac{\partial \log
       \overline{\sigma}_n^2 (\Theta_0)}{ \partial \Theta}\\ 
 & =& - \frac{f_{ \Theta_0}(\bfomega)}{ \overline{\sigma}_n^2
      (\Theta_0) } \bigg( \frac{1}{n^2} \sum_{\bfj \in \Lambda_n^2}
      \dfrac 1 {f_{\Theta_0}(\bfla_{\bfj})}  \frac{\partial
      f_{\Theta_0} (\bfla_{\bfj})}{\partial \Theta} \bigg)\\
&=&- \frac{f_{ \Theta_0}(\bfomega)}{ \ov \sigma^2
      (\Theta_0)}\frac 1{ (2\pi)^2} \int_{\Pi^2} f^{-1}_{\Theta_0}(\bfla) \frac{\partial
      f_{\Theta_0} (\bfla)}{\partial \Theta}\,{\rm d}\bfla (1+o(1)) 
=:\frac{f_{\Theta_0}(\bfomega)}{ \overline{\sigma}^2
      (\Theta_0)} \,\bfC\,(1+o(1))\,.
\end{eqnarray*}
Therefore, uniformly for $\bfomega$,
\begin{eqnarray*}
\dfrac 1{F_n^2 (\bfomega, \Theta_0)} \frac{\partial F_n (\bfomega,
  \Theta_0)}{\partial \Theta} & = & \frac{\overline{\sigma}^2
                                    (\Theta_0)}{f_{
                                    \Theta_0}(\bfomega)} \,\bfC\, (1+o(1))\,,
\end{eqnarray*}
and we finally have
\beao
\frac{n}{\sqrt{m_n}}(  \Theta_n -\Theta_0) 
  &=&- \wh W^{-1} \bfC\,(2\pi)^2 (1+o_\P(1))
\frac{n}{\sqrt{m_n}}\sum_{\bfj \in\Lambda_n^2} \dfrac{
      \wh{f} (\bfla_{\bfj}) -f_{\Theta_0}(\bfla_{\bfj}) }{f_{\Theta_0}(\bfla_\bfj)}\,.
\eeao
The conditions of Theorem~\ref {thm:cltintper} are satisfied for
$g=1/f_{\Theta_0}$. The second-order partial derivatives of 
  $g(\bfx)$, $\bfx \in \Pi^2$, are given by
\begin{eqnarray*}
\frac{\partial^2 g(\bfx)}{\partial x_1 \partial x_2}=\frac{2}{ f^3_{\Theta} (\bfx)} \frac{\partial f_{\Theta}(\bfx)}{
      \partial x_1} \frac{\partial f_{\Theta}(\bfx)}{ \partial x_2} -
      \frac{1}{f^2_{\Theta}(\bfx)} \frac{\partial^2
      f_{\Theta}(\bfx)}{\partial x_1 \partial x_2}\,. 
\end{eqnarray*}
Notice that $f_{\Theta_0}$ is bounded from below by \eqref{eq:jan30a},
and due to \eqref{eq:new41} we have
\beam\label{eq:rev1}
\Big|\frac{\partial f_{\Theta}(\bfx)}{
      \partial x_1} \frac{\partial f_{\Theta}(\bfx)}{ \partial x_2}
  \Big| + \Big| \frac{\partial^2
      f_{\Theta}(\bfx)}{\partial x_1 \partial x_2} \Big| \le c\, \Big(
  \sum_{\bfh \in \mathbb{Z}^2} |h_1 h_2 \gamma_{\Theta} (\bfh)|
  \Big)^2 <\infty\,,
\eeam
Therefore, \eqref{eq:gderiv} is satisfied. 

Now  Theorem~\ref {thm:cltintper} yields
\beam
\frac{n}{\sqrt{m_n}}(  \Theta_n -\Theta_0) &\std& 
-  (2\pi)^2\,\Big(\var\Big(\dfrac{\partial \log f_{\Theta_0}(\bfU)}{\partial \Theta}\Big)\Big)^{-1} \E \Big[\dfrac 1 {f_{\Theta_0}(\bfU)} \frac{\partial
      f_{\Theta_0} (\bfU)}{\partial \Theta}\Big]\nonumber\\&&\times\sum_{\bfh\in\bbz^2}Z_h \,\int_{\Pi^2} 
\cos(\bfh^\top \bfomega)\,\dfrac 1{f_{\Theta_0}(\bfomega)}\, {\rm d}\bfomega
=:\bfG\sim N({\bf0},W)\,.\label{eq:march8a}
\eeam
\end{proof}

\subsubsection*{Proof of Lemma~\ref{lem:123}}
{\bf Part 1.} We fix $\Theta \in \bfTh$.  Condition \eqref{eq:noone} is 
satisfied \fif\ $f_{\Theta_0}/f_\Theta$ is not a constant. Indeed, writing $\bfU$ for a uniform random vector on $\Pi^2$, \eqref{eq:noone} turns into
\beao
\exp\Big(-\E\Big[\log \dfrac{f_{\Theta_0}}{f_{\Theta}}(\bfU)\Big]\Big)\E\Big[\dfrac{f_{\Theta_0}}{f_{\Theta}}(\bfU)\Big]>1\,.
\eeao 
Taking logarithms, 
we get the equivalent inequality
\beao
\log \Big(\E\Big[\dfrac{f_{\Theta_0}}{f_{\Theta}}(\bfU)\Big]\Big)
> \E\Big[\log \dfrac{f_{\Theta_0}}{f_{\Theta}}(\bfU)\Big]\,,
\eeao
which is valid by
Jensen's inequality \eqref{eq:noone} with the exception that $f_{\Theta_0}/f_{\Theta}$ is constant.

{\bf Part 2.} We have
\beao
\lefteqn{ \sup_{\Theta\in\bfTh}\Big|\sigma_n^2 (\Theta) - 
\dfrac{\ov\sigma^2(\Theta)}{(2\pi)^2}\int_{\Pi^2} \dfrac{f_{\Theta_0}
(\bfomega)}{f_{\Theta}(\bfomega)}\,d\bfomega\Big|}\\
&\le & \sup_{\Theta\in\bfTh}\Big|\sigma_n^2 (\Theta) -
\dfrac{\ov\sigma^2(\Theta)}{n^2} \sum_{\bfj\in\Lambda_n^2}\dfrac{f_{\Theta_0}
(\bfla_\bfj)}{f_{\Theta}(\bfla_\bfj)}\Big|+\sup_{\Theta\in\bfTh}\Big|\dfrac{\ov\sigma^2(\Theta)}{n^2} \sum_{\bfj\in\Lambda_n^2}\dfrac{f_{\Theta_0}
(\bfla_\bfj)}{f_{\Theta}(\bfla_\bfj)}-
\frac{\ov\sigma^2(\Theta)}{(2\pi)^2}\int_{\Pi^2} \dfrac{f_{\Theta_0}(\bfomega)}
{f_{\Theta}(\bfomega)}\,d\bfomega\Big|\\
&=:&J_1+J_2\,.
\eeao
By uniform continuity of $(\ov\sigma^2(\Theta),f_\Theta(\bfomega))$ on
$\Pi^2\times \bfTh$,  $\sup_{\Theta\in\bfTh}\ov\sigma^2(\Theta)<\infty$ and
we have uniform \con\ of the Riemann sums in $J_2$ 
to the limiting  integrals. Therefore $J_2\to 0$. For the same reasons we also have $\sup_{\Theta\in\bfTh}|\ov\sigma_n^2(\Theta)-\ov\sigma^2(\Theta)|\to0$.
Now it suffices to show that
\beao
\wt J_1&:= &\sup_{\Theta\in\bfTh} \Big|\dfrac{1}{n^2} \sum_{\bfj\in\Lambda_n^2}
\dfrac{\wh f(\bfla_\bfj)}{f_{\Theta}(\bfla_\bfj)}
-\dfrac{1}{n^2} \sum_{\bfj\in\Lambda_n^2}
\dfrac{f_{\Theta_0}(\bfla_\bfj)}{f_{\Theta}(\bfla_\bfj)}
\Big| \\
&= &\sup_{\Theta\in\bfTh} \Big|\dfrac{1}{n^2} \sum_{\bfj\in\Lambda_n^2}
\Big(\wh f(\bfla_\bfj) -
f_{\Theta_0}(\bfla_\bfj)\Big)\big(1/f_{\Theta}(\bfla_\bfj) \big)\Big|
\,,
\eeao
vanishes as $\nto$. 
Let $q_{\bfl}( \bfomega; \Theta_0)$ be the C\`{e}saro mean of the first
$\bfl=(l, l)$-Fourier approximation to $1/f_{\Theta_0}$ on $\Pi^2$, i.e., 
for $\Theta\in\bfTh$,
\begin{eqnarray*}
  q_{\bfl} (\bfomega; \Theta)
  & =& \frac{1}{l^2} \sum_{i_1=0}^{l -1} \sum_{i_2 = 0}^{l -1}
      \sum_{|j_1|<l}\sum_{|j_2| <l} b_{\bfj} \ex^{-i \bfomega^\top \bfj} =
       \sum_{|j_1|<l} \sum_{|j_2 |<l} b_{\bfj}(\Theta) \prod_{k=1}^2 \Big(
       1- \frac{|j_k|}{l} \Big) \ex^{-i \,\omega_k\, j_k}\,,
\end{eqnarray*}
where 
\beao
b_{\bfj}(\Theta) = \frac{1}{(2\pi)^2}\, \int_{\Pi^2} \ex^{-i\,\bfomega^\top \bfj}
\frac{1}{f_{\Theta}(\bfomega)} \dif \bfomega\,.
\eeao
Fix $\vep>0$. Then for all sufficiently large $l$,
\beao
\sup_{\Theta\in\Pi^2\times \bfTh} \big|q_{\bfl}(\bfomega,\Theta)-1/f_\Theta(\bfomega)\big|<\vep\,.
\eeao 
For such an $l$,
\beao
\wt J_1&\le & \sup_{\Theta\in\bfTh}\Big|\dfrac{1}{n^2} \sum_{\bfj\in\Lambda_n^2}
(\wh f(\bfla_\bfj)-f_{\Theta_0}(\bfla_\bfj))\,q_{\bfl} (\bfla_\bfj; \Theta)\Big|
+ \vep \dfrac{1}{n^2}\sum_{\bfj\in\Lambda_n^2} 
|\wh f(\bfla_\bfj)-f_{\Theta_0}(\bfla_{\bfj})|=:\wt J_{11}+\vep\,\wt J_{12}\,.
\eeao 
Since $|b_{\bfj}(\Theta)|$ 
are uniformly bounded for $\Theta\in\bfTh$, we conclude that
\beao
\wt J_{11}&=&\sup_{\Theta\in\bfTh}\Big| 
 \sum_{|j_1|<l} \sum_{|j_2 |<l} b_{\bfj}(\Theta) \prod_{k=1}^2 \Big(
       1- \frac{|j_k|}{l} \Big) \big(\wh\gamma(\bfj)-\gamma(\bfj)\big)\Big|\stp 0\,,
\eeao
according to Theorem~\ref{thm:clt} and condition {\bf (M2)}.
Since $\vep>0$ is arbitrary it suffices to show that 
$\widetilde{\Jmath}_{12}$ is stochastically bounded.
Recalling the definition of $\wh f(\bfla_{\bfj})$ from
\eqref{eq:2021a}, we have for fixed $h>0$, some constant $c>0$,
\beao
\wt J_{12}
 & \le& c \,\sum_{0\le \|\bfh \| \le h}|\widehat{\gamma}(\bfh) - \gamma
 (\bfh)| + c\,\sum_{h< \| \bfh \| \le 3r_n} | \widehat{\gamma}(\bfh)|\\&& +c\,
\Big|\frac{1}{n^2} \sum_{\bfj\in \Lambda_n^2} \sum_{3r_n <\|\bfh \| <n } \big(
   \widehat{\gamma}(\bfh) - \E[\widehat{\gamma}(\bfh)] \big)
   \cos(\bfla_{\bfj}^{T}\bfh) \Big| \\ && + c\,\sum_{r_n<\|\bfh \| <n } \big| \E [\widehat{\gamma}(\bfh)]
   \big| + c\,\sum_{\|\bfh \| >h} \gamma(\bfh)\\
&=:&  Q_1+\cdots + Q_5  \,.
\eeao
Theorem~\ref{thm:clt} and condition {\bf (M2)} imply $Q_1\stp 0$ as $\nto$.

As discussed in the proof of Theorem~\ref{thm:cltintper}, condition {\bf (M1)} ensures that 
\begin{align*}
\E[Q_2]
\le c\,m_n \,\sum_{\bfh:h<\|\bfh\|\le 3r_n} \P(|X_{\bf0}|>a_m\,,|X_\bfh|>a_m)
  +O(r_n^2m_n^{-1})\,,
\end{align*}
and the \rhs\ vanishes by first letting $\nto$ and then $h\to\infty$.
Hence $Q_2$ is stochastically bounded. Since $\sum_{\bfj \in
  \Lambda_n^2} \cos (\bfla_{\bfj}^{\top} \bfh) =0$ we have $Q_3=0$. 
 An application of 
{\bf (M1)} shows that $Q_4\to 0$ as $\nto$. 
By assumption, $(\gamma(\bfh))$ is absolutely summable and therefore $Q_5$ vanishes as $h\to\infty$.

This completes the proof of Part 2.\\
{\bf Part 3.} 
We will prove $\Theta_n\stp\Theta_0$. Then $\sigma_n^2(\Theta_n)\stp \ov \sigma^2(\Theta_0)$ follows from Part 2.
\par
Suppose that $\Theta_n\stp \Theta_0$ does not hold.
Then by definition of $\Theta_n$ as minimizer of $\sigma_n^2(\Theta)$ for
 $\Theta\in\bfTh$, we obtain from Part 2. For all $x\ge 0$,
\beam\label{eq:feb4a}
\P\big(\sigma_n^2(\Theta_n)/\ov\sigma^2(\Theta_0)\le x\big)\ge \P\big(
\sigma_n^2(\Theta_0)/\ov\sigma^2(\Theta_0)\le x\big)
\to \1_{[0,x]}(1)\,.
\eeam
By the Helly-Bray Theorem and compactness of $\bfTh$, there exists a 
non-random integer \seq\ $(n_k)$ \st\ $\Theta_{n_k}\std \Theta^\ast$ and the 
\rv\ $\Theta^\ast$ is different from $\Theta_0$ on a set of positive \pro y.
The \fct al $F:C(\bfTh)\times \bfTh\to\bbr$ \st\ $F(g,\Theta)=g(\Theta)/\ov\sigma^2(\Theta_0)$
is continuous, where $C(\bfTh)$ is the space of continuous \fct s 
on $\bfTh$ equipped with the sup-norm. In view of Part 2., 
\beao
\sigma_n^2(\cdot)\stp \dfrac{\ov\sigma^2(\cdot)}{(2\pi)^2}\int_{\Pi^2} \dfrac{f_{\Theta_0}(\bfomega)}{f_\bullet(\omega)}\,d\bfomega\,.
\eeao
Hence $(\sigma_n^2)$ is tight in $C(\bfTh)$ and since $\Theta_{n_k}\std \Theta^\ast$,
$(\Theta_{n_k})$ is tight as well. Thus $(\sigma_{n_k}^2,\Theta_{n_k})$ is
tight in $C(\bfTh)\times \bfTh$ and there is a sub\seq\ $(n_k')$ of $(n_k)$
\st\  $(\sigma_{n_k'}^2,\Theta_{n_k'})$ converges in \ds . By the \cmt ,
\beao
F\big(\sigma_{n_k'}^2,\Theta_{n_k'}\big)=\dfrac{\sigma_{n_k'}^2(\Theta_{n_k'})}{\ov\sigma^2(\Theta_0)}
\std  \dfrac{1}{(2\pi)^2}\dfrac{\ov \sigma^2(\Theta^\ast)}{\ov\sigma^2(\Theta_0)}\int_{\Pi^2} \dfrac{f_{\Theta_0}(\bfomega)}{f_{\Theta^\ast}(\omega)}\,d\bfomega=:T(\Theta^\ast)\,.
\eeao
For a continuity point $1+\delta>0$ of the \ds\ of $T(\Theta^\ast)$ we have
\beao
\lim_{\kto} \P\big(\sigma_{n_k'}^2(\Theta_{n_k'})/\ov \sigma^2(\Theta_0)\le 1+\delta\big)
= \P(\Theta^\ast=\Theta_0)+ \P\big(T(\Theta^\ast)\le 1+\delta\,,\Theta\ne \Theta_0\big)\,.
\eeao
Note that by Part 1., $\{\Theta\in \bfTh: T(\Theta)>1\}=
\{\Theta\in\bfTh: \Theta\ne \Theta_0\}$. By \eqref{eq:feb4a}
we conclude that for sufficiently small $\delta$ and $x=1+\delta$ of the 
above type, 
\beao
1&\le &\lim_{\nto} \P\big(\sigma_{n_k'}^2(\Theta_{n_k'})/\ov\sigma^2(\Theta_0)\le 1+\delta\big)\\
&=& \P(\Theta^\ast=\Theta_0)+\P\big(1<T(\Theta^\ast)\le 1+\delta\,,\Theta\ne \Theta_0\big)\,,
\eeao 
and the \rhs\ can be made arbitrarily close $\P(\Theta^\ast=\Theta_0)<1$
which yields a contradiction and proves Part 3. \qed 

Recall the definition of the \fct s $F$ and $F_n$ from \eqref{eq:feb5a}.
\begin{lemma} \label{prop:convderivates}
 Assume {\bf (W)}. Let $(\Theta_n)$ be a sequence
 in $\bfTh$ \st\ $\Theta_n\stp \Theta_0$. Then
\begin{equation}
\label{eq:qq4}
\frac{\partial^2 \sigma_n^2(\Theta_n)}{ \partial \Theta^2}
\stp\frac{\overline{\sigma}^2 (\Theta_0)}{(2\pi)^2}
\int_{\Pi^2} \dfrac 1 {F (\bfomega, \Theta_0)}\frac{\partial^2 F(\bfomega,
  \Theta_0)}{\partial \Theta^2} \dif \bfomega\,.
\end{equation}
\end{lemma}
\begin{proof}[Sketch of the proof]
The proof uses similar arguments as in Part 2. above.
Notice that the left-hand side is an
$s\times s$-matrix. It is enough to show that for each pair
$(s_1,s_2)\in \{ 1, \ldots ,s\}^2$,
\begin{equation}
\label{eq:qq5}
\frac{\partial^2 \sigma_n^2 (\Theta_n)}{\partial \theta_{s_1}
  \partial\theta_{s_2} } \stp \frac{\overline{\sigma}^2 (\Theta_0)}{(2\pi)^2}
\int_{\Pi^2} \dfrac 1{F (\bfomega, \Theta_0)}\frac{\partial^2 F(\bfomega,
  \Theta_0)}{\partial \theta_{s_1} \partial \theta_{s_2}} \dif \bfomega\,.
\end{equation}
Recall that \[
  \sigma_n^2 (\Theta) = \frac{1}{n^2} \sum_{\bfj\in\Lambda*n^2}
  \frac{\wh{f}(\bfla_{\bfj})}{f_{\Theta} (\bfla_{\bfj})/
    \overline{\sigma}_n^2 (\Theta)} = \frac{1}{n^2} \sum_{\bfj \in\Lambda_n^2}
\dfrac{  \wh{f}(\bfla_{\bfj})}{ F_n(\bfla_{\bfj}, \Theta)}\,,
  \quad \Theta\in \bfTh\,,
\]
and thus
\[
\frac{\partial^2 \sigma_n^2(\Theta)}{\partial \theta_{s_1} \partial
  \theta_{s_2}} = \frac{1}{n^2} \sum_{\bfj\in\Lambda_n^2}^n \wh{f}
(\bfla_{\bfj} ) \frac{\partial^2 F_n^{-1}(\bfla_{\bfj}, \Theta) }{
\partial \theta_{s_1} \partial \theta_{s_2}}\,.
\]
Now we can proceed as in the proof of Part 2. We omit further details.
\end{proof}
\appendix 
\section{Auxiliary results}\setcounter{equation}{0}
\begin{lemma}\label{lem:psih} Assume that the \fct\ $g$ satisfies the
  conditions of Theorem ~\ref{thm:cltintper}. The following statements
hold for some constant $c>0$: 
  \begin{enumerate}[\rm 1.]
  \item $|\psi_{\bfh}|\le c\, |h_1h_2|^{-1}$, $h_1h_2\ne 0$, 
\item[\rm 2.] $|\psi_{\bfh}^{(n)}| \le c \,\max(|h_1h_2|^{-1}, n^{-1})$, 
$1\le |h_1|,|h_2|\le n/2$,
\item[\rm 3.]  $|\psi_{\bfh} - \psi_{\bfh}^{(n)}|\le  cn^{-1}$,  $1\le
  |h_1|,|h_2| \le n/2$. 
  \end{enumerate}
\end{lemma}
\begin{proof}
If 1.  and 3. hold then 2. follows. 
Integration by parts and the fact that derivatives of periodic \fct s are periodic yield
\begin{eqnarray*}
\psi_{\bfh} & =&  \int_{\Pi^2} \cos (\bfh^{\top} \bfx) g(\bfx) \dif
                 \bfx 
  = \int_{\Pi^2} \cos( h_1 x_1) \cos (h_2 x_2)
                 g(\bfx) \dif \bfx - \int_{\Pi^2} \sin (h_1 x_1) \sin
                 (h_2 x_2) g(\bfx )\dif \bfx\\ 
 & = & \int_0^{2\pi} \cos (h_2 x_2)\Big[ g(\bfx) \dfrac{\sin (h_1 x_1)}{h_1}\mid_{x_1=0}^{2\pi} - \int_0^{2\pi} \dfrac{\sin (h_1 x_1)}{h_1}
       \frac{\partial g(\bfx)}{ \partial x_1}\dif x_1  \Big]\dif
       x_2\\ 
 & & -  \int_0^{2\pi} \sin (h_2 x_2)\Big[ -g(\bfx) \dfrac{\cos (h_1 x_1)}{h_1} \mid_{x_1
       =0}^{2\pi} + \int_0^{2\pi} \dfrac{\cos (h_1 x_1)}{h_1}
       \frac{\partial g(\bfx)}{ \partial x_1}\dif x_1  \Big]\dif
       x_2\\ 
 & = & \int_{\Pi^2} \dfrac{\sin (h_1x_1) \sin (h_2 x_2)}{h_1h_2}
       \frac{\partial^2 g(\bfx)}{\partial x_1 \partial x_2} \dif
       \bfx - \int_{\Pi^2} \dfrac{\cos (h_1x_1) \cos (h_2 x_2)}{h_1h_2}
       \frac{\partial^2 g(\bfx)}{\partial x_1 \partial x_2} \dif
       \bfx\\ 
 & =&  - \int_{\Pi^2} (h_1 h_2 )^{-1} \cos (\bfh^{\top} \bfx)
      \frac{\partial^2 g(\bfx)}{\partial x_1 \partial x_{2}}
      \dif \bfx\,.
\end{eqnarray*}
By assumption \eqref{eq:gderiv},  
$\sup_{\bfx\in \Pi^2}|\partial^2 g(\bfx) /(\partial
x_1 \partial x_2)|<\infty$. Therefore $|\psi_\bfh|\le c\,|h_1h_2|^{-1}$,
proving 1.
\par
In the sequel we deal with part 3.
Write $\bfla_\bfj=(\lambda_{j_1},\lambda_{j_2})$ and $\int'= 
\int_{{\bf0}\le\bfx\le {\bfla}_{(1,1)}}$. Then we have
\begin{eqnarray*}
  \big| \psi_{\bfh} - \psi_{\bfh}^{(n)} \big| 
  & = & \Big|\sum_{\bfj \in \Lambda_n^2} \int'
        \big(\cos(\bfh^\top (\bfla_\bfj+\bfx))\, 
g(\bfla_{\bfj } +\bfx)  - \cos(\bfh^\top \bfla_\bfj)\, g(\bfla_\bfj)
        \dif \bfx \Big|\\ 
 &\le  & \Big| \sum_{\bfj \in \Lambda_n^2} \int'
       \Big(\prod_{i=1}^2 \cos (h_i (\lambda_{j_i} +x_i))\,
       g(\bfla_\bfj +\bfx)  - \prod_{i=1}^2 \cos
       (h_i \lambda_{j_i})\, g(\bfla_\bfj)\Big)\dif
       \bfx \Big|\\ 
 & & + \Big| \sum_{\bfj \in \Lambda_n^2} \int'
       \Big(\prod_{i=1}^2 \sin (h_i (\lambda_{j_i} +x_i))\,
       g(\bfla_\bfj +\bfx)  - \prod_{i=1}^2 \sin
       (h_i \lambda_{j_i}) g(\bfla_\bfj)\Big)\dif
       \bfx \Big| \\  
 & =& Q_1(\bfh) + Q_2(\bfh)\,. 
\end{eqnarray*}
We only deal with  $Q_1(\bfh )$; $Q_2(\bfh)$ can be treated similarly.
We have
\beao
|Q_1(\bfh)|&\le & \Big| \sum_{\bfj \in \Lambda_n^2} \int'
       \prod_{i=1}^2 \cos (h_i (\lambda_{j_i} +x_i))\,
       \big(g(\bfla_\bfj +\bfx) -g(\bfla_\bfj)\big)\dif \bfx\Big|\\
       && + \Big| \sum_{\bfj \in \Lambda_n^2} \int'
\Big(\prod_{i=1}^2 \cos(h_i \lambda_{j_i})\, -\prod_{i=1}^2 \cos (h_i (\lambda_{j_i} +x_i))\Big)\, g(\bfla_\bfj)\dif
       \bfx \Big|\\ 
&=&|Q_{11}(\bfh)|+|Q_{12}(\bfh)|\,.
 \eeao
In view of \eqref{eq:gderiv} the \fct\ $g$ has a uniformly bounded 
derivative on $\Pi^2$. Therefore a  Taylor expansion yields
$|g(\bfla_\bfj +\bfx) -g(\bfla_\bfj)|\le c\,\|\bfx\|$ and
\beao
|Q_{11}(\bfh)|&\le &c\,\sum_{\bfj \in \Lambda_n^2} \int'\|\bfx\|\,\dif \bfx
\le c\,n^{-1}\,. 
\eeao 
Now we turn to $|Q_{12}(\bfh)|$:
\beao
|Q_{12}(\bfh)|&\le &\Big| \sum_{\bfj \in \Lambda_n^2} g(\bfla_\bfj)\int'
\big(\cos (h_1 (\lambda_{j_1} +x_1))-\cos(h_1\lambda_{j_1})\big)\cos(h_2\lambda_{j_2}+x_2)
\, \dif
       \bfx \Big|\\
&&+\Big| \sum_{\bfj \in \Lambda_n^2}g(\bfla_\bfj) \int'
\big(\cos (h_2 (\lambda_{j_2} +x_2))-\cos(h_2\lambda_{j_2})\big)\cos(h_1\lambda_{j_1})\, \dif
\bfx \Big|\\
& =&|Q_{121}(\bfh)|+ |Q_{122}(\bfh)|\,.
\eeao
We bound $|Q_{121}(\bfh)|$ only, the bound for $|Q_{122}(\bfh)|$ is
similar. We have 
%
\begin{eqnarray*}
|Q_{121}(\bfh)|&\le &c\,\Big|\sum_{j_1 =1}^ng(\bfla_{\bfj})\int_0^{\lambda_1}\,
\big(\cos (h_1 (\lambda_{j_1} +x_1))-\cos(h_1\lambda_{j_1})\big)
\, \dif x_1 \Big|\\
&=&c\,\Big|\sum_{j_1 =1}^ng(\bfla_{\bfj}) \int_0^{\lambda_1}  \int_0^{x_1} h_1
      \sin(h_1 (\lambda_{j_1} +y))
      \dif y \dif x_1 \Big|\\
&=&c\,\Big|\sum_{j_1 =1}^ng(\bfla_{\bfj}) \int_0^{\lambda_1} \int_0^{x_1}
        \frac{h_1}{ 2 \sin (\lambda_{h_1}/2)}\big( 2 \sin
        (\lambda_{h_1}/2)  \sin(h_1 (\lambda_{j_1} + y))\big)
              \dif y \dif x_1 \Big|\\
&=&c\,\Big|\frac{h_1}{ 2
      \sin (\lambda_{h_1}/2)} \sum_{j_1 =1}^n \int_0^{\lambda_1} g(\bfla_{\bfj})\\
&&\int_0^{x_1} 
      \big(\cos(h_1 \lambda_1 (2 j_1 -1) /2 +h_1 y) -\cos(h_1
      \lambda_1\,(2j_1+1)/2 +h_1 y)\big) \dif y\dif x_1 \Big|\\
&=&c\,\dfrac{n}{2\pi}\Big|\frac{\lambda_{h_1}/2}{
      \sin (\lambda_{h_1}/2)}\Big|\,\wt Q(\bfh)\,.
\end{eqnarray*}
In the second last step we used the identity $2 \sin x \sin y =\cos (x-y)- 
\cos (x+y)$. 
 The \fct\ $x/\sin x$ is continuous on $[0, \pi/2]$ (with the
convention that $0/\sin (0)=1$), hence $\sup_{x\in [0,\pi/2]}|x/\sin x|< \infty$.
Observing that $\lambda_{h_1}/2 \in [0, \pi/2]$ is equivalent to
$h_1\le n/2$, we obtain
\beao
\sup_{h_1: \lambda_{h_1}/2 \in [0, \pi/2]}(\lambda_{h_1}/2)/\sin ( \lambda_{h_1}/2) <\infty\,.
\eeao

By virtue of the component-wise
periodicity of $g$, i.e., $ g(2\pi  +x_1,x_2)=g(\bfx)$ for $\bfx\in \Pi$, we have 
\begin{eqnarray*}
\wt Q(\bfh)
 & =&   \Big| \int_0^{\lambda_1} \int_0^{x_1} \Big( \sum_{j_1 =2}^{n} \cos(h_1
       \lambda_{j_1} -\lambda_{h_1}/2 +h_1 y) \big( g(\bfla_{\bfj}) -
       g(\lambda_{j_1 -1}, \lambda_{j_2}) \big)\\ 
 & &   \qquad+ \cos
      (\lambda_{h_1}/2 +h_1 y) g(\lambda_1, \lambda_{j_2}) - \cos
      (\lambda_{h_1}(2n+1)/2 +h_1 y) g(\lambda_n, \lambda_{j_2})
     \Big)\dif y\dif x_1   \Big|\\ 
 &= &   \Big| \int_0^{\lambda_1}\int_0^{x_1} \Big( \sum_{j_1 =2}^{n} \cos(h_1
       \lambda_{j_1} -\lambda_{h_1}/2 +h_1 y) \big( g(\bfla_{\bfj}) -
       g(\lambda_{j_1 -1}, \lambda_{j_2}) \big)\\ 
 & &   \qquad+ \cos
      (\lambda_{h_1}/2 +h_1 y) g(\lambda_1, \lambda_{j_2}) - \cos
      (\lambda_{h_1}/2 +h_1 y) g(0, \lambda_{j_2}) \Big)\dif y\dif x_1   \Big|\\ 
 & = &  \Big| \int_0^{\lambda_1} \int_0^{x_1} \sum_{j_1 =1}^n \cos(h_1
       \lambda_{j_1} -\lambda_{h_1}/2 +h_1 y) \big( g(\bfla_{\bfj}) -
       g(\lambda_{j_1 -1}, \lambda_{j_2}) \big) \dif y\dif x_1 \Big|\\ 
 & \le & 
         \dfrac{\lambda_1^2}{2}\,\sum_{j_1 =1}^n \big| g(\bfla_{\bfj}) -
       g(\lambda_{j_1 -1}, \lambda_{j_2}) \big|  \le c\, n^{-2}
        \,.
\end{eqnarray*}
In the last step we used the fact that $| g(\bfla_{\bfj})
 -g(\lambda_{j_1 -1}, \lambda_{j_2}) \big| \le c/n$. Combining the previous
bounds, we proved that $|Q_{12} (\bfh)| =O(n^{-1})$ as $n\to \infty$ uniformly 
for $h_1,h_2\le n/2$.  

This completes the proof. 
\end{proof}

\section{Proof of Lemma~\ref{lem:q2neg}} \label{sec:covv}
We start with some bounds for the  covariances $\cov(\widehat{\gamma}(\bfs),
\widehat{\gamma}(\bfs+ \bfh ))$, $\|\bfs\|, \|\bfs+\bfh\|<n$. 
For simplicity we restrict ourselves to $\bfs,\bfs+\bfh\in \Lambda_n^2$;
the remaining cases can be dealt with in a similar way.
We have
\begin{align}
\cov(\widehat{\gamma}(\bfs), \widehat{\gamma}(\bfs + \bfh)) & =
  \frac{m_n^2}{n^4} \Big[ \sum_{i_1 =1}^{n-s_1} \sum_{i_2 =1}^{n-s_2}
  \sum_{j_1 =1}^{n-s_1 -h_1} \sum_{j_2 =1}^{n-s_2 -h_2}
 \big( \widehat{\Imath}_{\bfii} \widehat{\Imath}_{\bfii+\bfs}
  \widehat{\Imath}_{\bfj } \widehat{\Imath}_{\bfj + \bfs + \bfh}-
   \E \big[   \widehat{\Imath}_{\bfii} \widehat{\Imath}_{\bfii+\bfs}
 \big]\, \E\big[\widehat{\Imath}_{\bfj } \widehat{\Imath}_{\bfj + \bfs + \bfh}\big]\big)\,.\nonumber\\
&=: J_1(\bfs,\bfs+\bfh)-J_2(\bfs,\bfs+\bfh)\,.\label{eq:covvv}
\end{align}
We will make use of a variant of Theorem 17.2.1 in
Ibragimov and Linnik~\cite{ibragimov:linnik:1971}:
\begin{lemma}\label{lem:ibragimov}
  Let $(X_{\bft})_{\bft \in \mathbb{Z}^2}$ be a strictly stationary
  strongly mixing random field with mixing rate $\alpha$. 
Consider two sets $L_1, L_2\subset
  \mathbb{Z}^2$, whose distance (respective to the max-norm $\|\cdot\|$)
is larger than $l>0$, and measurable \fct s $Y_i$ of 
  $(X_{\bft})_{\bft \in L_i}$, $i=1,2$. If 
 $|Y_i| \le C_i$, $i=1,2$, for some constant $C_i$, $i=1,2$, then
\begin{align}
\label{eq:ibragimov}
|\cov(Y_1,Y_2)| \le 4\,C_1 \,C_2\,
  \alpha(l)\,.
\end{align}
\end{lemma}
Then we get for $\bfii,\bfj\in\Lambda_n^2$,
$
|\cov(\widehat{\Imath}_{\bfii}, \widehat{\Imath}_{\bfj})| 
\le c \,\alpha(\|\bfii - \bfj\|)\,,
$ and for $\|\bfs\|, \|\bfs+\bfh\|<n$,
\beam
|J_2(\bfs,\bfs+\bfh)| &\le  & 
  \dfrac{m_n^2}{n^4}  \sum_{i_1 =1}^{n-s_1} \sum_{i_2 =1}^{n-s_2}
  \sum_{j_1 =1}^{n-s_1 -h_1} \sum_{j_2 =1}^{n-s_2 -h_2} \big|
\E \big[   \widehat{\Imath}_{\bfii} \widehat{\Imath}_{\bfii+\bfs}
    \big] \big| \, \big|\E\big[\widehat{\Imath}_{\bfj } \widehat{\Imath}_{\bfj + \bfs + \bfh}
  \big] \big|\nonumber\\ 
 & \le& c m_n^2 \alpha(\|\bfs \|) \alpha(\|\bfs + \bfh \|)\,.\label{eq:xxx} 
\eeam
We observe that for $\bfii\not\in\{\bfj,\bfk,\bfl\}$,
\beam \label{eq:mixingii}
\big|\E\big[\widehat{\Imath}_{\bfii }
\widehat{\Imath}_{\bfj } \widehat{\Imath}_{\bfk }
\widehat{\Imath}_{\bfl }\big]\big|=\big|\cov\big(\widehat{\Imath}_{\bfii },
\widehat{\Imath}_{\bfj } \widehat{\Imath}_{\bfk }
\widehat{\Imath}_{\bfl }\big)\big|\le  c \,\alpha(d_{\bfii; \bfj \bfk \bfl})\,,
\eeam
where $d_{\bfii; \bfj \bfk \bfl}$ is the distance between $\{\bfii\}$ and
$\{\bfj,\bfk,\bfl\}$ with respect to $\|\cdot\|$.
Similarly, for  disjoint
$\{\bfii,\bfj\},\{\bfl,\bfk\}$,
\begin{eqnarray}\label{eq:mixingij}
\big|\E [\widehat{\Imath}_{\bfii} \widehat{\Imath}_{\bfj}
  \widehat{\Imath}_{\bfk} \widehat{\Imath}_{\bfl}] 
  -\E[\widehat{\Imath}_{\bfii} \widehat{\Imath}_{\bfj}]
  \E[\widehat{\Imath}_{\bfk} \widehat{\Imath}_{\bfl}] \big| & \le &c\, 
\alpha(d_{\bfii \bfj; \bfk\bfl})\,,
\end{eqnarray}
where $d_{\bfii \bfj; \bfk\bfl}$ is the distance between $\{\bfii,\bfj\}$ and 
$\{\bfk,\bfl\}$.
\begin{proof}[Proof of Lemma~\ref{lem:q2neg}]
Due to
  \eqref{eq:new41}  
there exists a constant $c$ such that
  $|\psi_{\bfh}^{(n)}|\le c$ for all $\bfh$ satisfying $\|\bfh\|<n$.  By Markov's inequality and Lemma~\ref{lem:psih}  
we have for $\vep>0$,
  \beam
    Q_2 &=&\P \Big(
     \frac{n}{\sqrt{m_n}} \Big| \sum_{\bfh:3 r_n < \|\bfh \| \le n}
     \psi_{\bfh}^{(n)}(\wh{\gamma} (\bfh) - \E [
     \wh{\gamma} (\bfh)]) \Big| >\varepsilon\Big)\nonumber\\
&\le &\dfrac{n^2}{m_n\,\vep^2} \E\Big[\Big(\sum_{\bfh:3 r_n < \|\bfh \| \le n}
     \psi_{\bfh}^{(n)}(\wh{\gamma} (\bfh) - \E [
     \wh{\gamma} (\bfh)])
\Big)^2\Big]\label{eq:renwe}\\
& \le& \frac{c \,n^2}{m_n} \sum_{3r_n<\|\bfs\|, \|\bfs+\bfh\|\le n}
|\psi_\bfs^{(n)}\psi_{\bfs+\bfh}^{(n)}|
\big| \cov \big(\widehat{\gamma}(\bfs), \widehat{\gamma} (\bfs+\bfh) \big)
  \big|\nonumber\\
& \le& \underbrace{\frac{c \,n^2}{m_n} \sum_{3r_n<\|\bfs\|, \|\bfs+\bfh\|\le n}\big(|\psi_\bfs^{(n)}\psi_{\bfs+\bfh}^{(n)}|\,|J_1(\bfs,\bfs+\bfh)|}_{=:J_3}
+|J_2(\bfs,\bfs+\bfh)|\big)\nonumber\\
&\le &J_3+c\,\Big(n m_n^{1/2} \sum_{\bfh: \|\bfh \|>3r_n} \alpha\big( \|\bfh \|\Big)^2
= J_3 + O\big((n\,m_n^{1/2}\alpha(r_n))^2\big)=J_3+o(1)\,.\nonumber
  \eeam
In the last step we used \eqref{eq:xxx} and condition \eqref{eq:newlabel}.
The second term converges to zero in view of \eqref{eq:newlabel}.
We have 
\beao
J_3&\le &c\,\dfrac {m_n}{n^2}\sum_{3r_n<\|\bfs\|, \|\bfs+\bfh \|<n}
|\psi_\bfs^{(n)}\psi_{\bfs+\bfh}^{(n)}|
\sum_{\bfii, \bfj
  \in \Lambda_n^2} \big|\E
  \big[ \widehat{\Imath}_{\bfii} \widehat{\Imath}_{\bfii+\bfs}
  \widehat{\Imath}_{\bfj} \widehat{\Imath}_{\bfj+\bfs+\bfh} \big] \big|
\eeao
We consider 3 cases:
{\bf 1.} $d_{\bfii (\bfii+\bfs); \bfj (\bfj+\bfs+\bfh)}>r_n$, 
{\bf 2.} $d_{\bfii \bfj; (\bfii+\bfs) (\bfj+\bfs+\bfh)}>r_n$ and 
{\bf 3.} $d_{\bfii (\bfj+\bfs+\bfh);(\bfii+\bfs)\bfj}>r_n$. In case 1. we can apply \eqref{eq:mixingij} and obtain
\beam\label{eq:nov2021a}
\big|\E
  \big[ \widehat{\Imath}_{\bfii} \widehat{\Imath}_{\bfii+\bfs}
  \widehat{\Imath}_{\bfj} \widehat{\Imath}_{\bfj+\bfs+\bfh} -\E[\wh I_{\bfii}\,\wh I_{\bfii+\bfs}]\,
\E[\wh I_{\bfj}\,\wh I_{\bfj+\bfs+\bfh}]\big] \big|\le c\,\a(r_n)\,.
\eeam
Together with \eqref{eq:xxx} and  \eqref{eq:newlabel} we conclude that
\beao&&
\dfrac {m_n}{n^2}\sum_{3r_n<\|\bfs\|, \|\bfs+\bfh \|<n}
|\psi_\bfs^{(n)}\psi_{\bfs+\bfh}^{(n)}|
\sum_{\bfii, \bfj
  \in \Lambda_n^2\,,d_{\bfii (\bfii+\bfs); \bfj (\bfj+\bfs+\bfh)}>r_n} \big|\E
  \big[ \widehat{\Imath}_{\bfii} \widehat{\Imath}_{\bfii+\bfs}
  \widehat{\Imath}_{\bfj} \widehat{\Imath}_{\bfj+\bfs+\bfh} \big] \big|\\
&\le &c\,m_n\,n^6\,\alpha(r_n) +o(1)=o(1)\,.
\eeao

Now we consider the case that 1. does not hold, i.e., one of the following cases appears:\\
{\bf 1a.} $\|\bfii-\bfj\|\le r_n$. But then we have $\|\bfii-(\bfii+\bfs)\|=\|\bfs\|>3\,r_n$,
$\|\bfii-(\bfj+\bfs+\bfh)\|\ge \|\bfs+\bfh\|-\|\bfii-\bfj\|\ge 2\,r_n$ and 
$\|\bfj-(\bfj+\bfs+\bfh)\|=\|\bfs+\bfh\|>3r_n$, $\|\bfj-(\bfii+\bfs)\|\ge \|\bfs\|-\|\bfii-\bfj\|\ge 2r_n$.
Hence we can apply case~2.\\
{\bf 1b.} $\|\bfii-(\bfj+\bfs+\bfh)\|\le r_n$. By the triangle inequality,
$2r_n\le \|\bfs+\bfh\|-r_n\le \|\bfii-\bfj\|$, $\|\bfii-(\bfii+\bfs)\|=\|\bfs\|>3 \,r_n$,
$\|(\bfj+\bfs+\bfh)-\bfj\|=\|\bfs+\bfh\|>3r_n$ 
and  $\|(\bfii+\bfs)-(\bfj+\bfs+\bfh)\|=\|\bfii-(\bfj+\bfs+\bfh)+\bfs \|\ge 
\|\bfs\|-\|\bfii-(\bfj+\bfs+\bfh)\|\ge 2 r_n$.
Hence we can apply case 3.\\
{\bf 1c.} $\|(\bfii+\bfs)-\bfj|\le r_n$. By the triangle inequality,  
$2r_n\le \|\bfs\|-r_n\le \|\bfii-\bfj\|$.
We also have $\|\bfii-(\bfii+\bfs)\|=\|\bfs\|>3r_n$, $\|\bfj-(\bfj+\bfs+\bfh)\|=\|\bfs+\bfh\|>3r_n$
and $\|(\bfii+\bfs)- (\bfj+\bfs+\bfh)\|\ge \|\bfs+\bfh\| -\|(\bfii+\bfs)-\bfj\|\ge 2r_n$.
Hence we can apply case 3.\\
{\bf 1d.} $\|(\bfii+\bfs)-(\bfj+\bfs+\bfh)|\le r_n$. Then $\|\bfii-(\bfii+\bfs))\|=\|\bfs\|>3r_n$, 
$\|\bfj-(\bfj+\bfs+\bfh)\|=\|\bfs+\bfh\|>3r_n$ and 
$\|\bfii-(\bfj+\bfs+\bfh)\|=\|(\bfii+\bfs)-(\bfj+\bfs+\bfh)-\bfs\|\ge \|\bfs\|-
\|(\bfii+\bfs)-(\bfj+\bfs+\bfh)\|\ge 2r_n$, 
$\|(\bfii+\bfs)-\bfj\|=\|(\bfii+\bfs)-(\bfj+\bfs+\bfh)+(\bfs+\bfh)\|\ge \|\bfs+\bfh\|-\|(\bfii+\bfs)-(\bfj+\bfs+\bfh)\|\ge 2r_n$. Hence we can apply case 2.
\par
Now assume that 2. is satisfied. Then
\beao
&&\big|\E
  \big[ \widehat{\Imath}_{\bfii} \widehat{\Imath}_{\bfj}
  \widehat{\Imath}_{\bfii+\bfs} \widehat{\Imath}_{\bfj+\bfs+\bfh} -\E[\wh I_{\bfii}\,\wh I_{\bfj}]\,
\E[\wh I_{\bfii+\bfs}\,\wh I_{\bfj+\bfs+\bfh}]\big] \big| +\big|\E[\wh I_{\bfii}\,\wh I_{\bfj}]\,
\E[\wh I_{\bfii+\bfs}\,\wh I_{\bfj+\bfs+\bfh}]\big]\big|\\
&\le &\alpha(r_n)+\alpha(\|\bfii-\bfj\|)\,\alpha(\|\bfii-(\bfj+\bfh)\|)\,.
\eeao
If $\|\bfii-\bfj\|>r_n$ or $\|\bfii-\bfj\|\le r_n$ and $\|\bfh\|>2r_n$ then the \rhs\ is $O(\a(r_n))$ and we can proceed as in case 1.
Now assume that $\|\bfii-\bfj\|\le r_n$ and $\|\bfh\|\le 2r_n$.
Then we have by stationarity,
\beao
\big|\E
  \big[ \widehat{\Imath}_{\bfii} \widehat{\Imath}_{\bfj}
  \widehat{\Imath}_{\bfii+\bfs} \widehat{\Imath}_{\bfj+\bfs+\bfh}\big|
&=&\big|\E
  \big[ \widehat{\Imath}_{\bf0} \widehat{\Imath}_{\bfj-\bfii}
  \widehat{\Imath}_{\bfs} \widehat{\Imath}_{(\bfj-\bfii)+\bfs+\bfh}\big|\\
&\le &\alpha(r_n) +|\E[\widehat{\Imath}_{\bf0} \widehat{\Imath}_{\bfj-\bfii}]|\,
|\E[\widehat{\Imath}_{\bf0} \widehat{\Imath}_{(\bfj-\bfii)+\bfh}]|\\
&\le &\a(r_n)+ c\,m_n^{-2} \,.
\eeao
Thus it remains to bound the following expression:
\beam\label{eq:999}&&
\dfrac {1}{n^2m_n}\sum_{3r_n<\|\bfs\|, \|\bfs+\bfh \|<n\,,\|\bfh\|\le 2r_n}
|\psi_\bfs^{(n)}\psi_{\bfs+\bfh}^{(n)}|
\#\{
\bfii, \bfj
  \in \Lambda_n^2:d_{\bfii \bfj; (\bfii+\bfs) (\bfj+\bfs+\bfh)}>r_n\,,\|\bfii-\bfj\|\le r_n\}\nonumber\\
&\le & c\,\dfrac {r_n^2}{m_n}\sum_{3r_n<\|\bfs\|, \|\bfs+\bfh \|<n\,,\|\bfh\|\le 2r_n}
|\psi_\bfs^{(n)}\psi_{\bfs+\bfh}^{(n)}|\,.\nonumber\\
\eeam
In view of the definition of $\psi_{\bfh}^{(n)}$ we have
\begin{eqnarray*}
  \psi_{\bfh}^{(n)} &= & 
 \frac{(2\pi)^2}{n^2} \sum_{\bfj \in \Lambda_n^2} (\cos(h_1
      \lambda_1) \cos (h_2 \lambda_{j_2}) - \sin(h_1 \lambda_{j_1})
      \sin (h_2 \lambda_{j_2}) ) g(\bfla_{\bfj})\\ 
 & =& \psi_{h_1+n,h_2}^{(n)} = \psi_{n-h_1,-h_2}^{(n)} = \psi_{n-h_1, n-h_2}^{(n)}\,.  
\end{eqnarray*}
Here we used the fact that $n\lambda_j = 2 \pi j$ and the
periodicity of the cosine and sine functions. Therefore, for
$n/2\le\|\bfh \|<n$, we can find $\overline{\bfh} \in \mathbb{N}^2$
such that $\| \overline{\bfh}\|<n/2$, $(|h_1|, |h_2|)\in \Lambda_n^2$
and $\psi_{\bfh}^{(n)} = \psi_{\overline{\bfh}}^{(n)}$. 
By virtue of this argument and in view of our previous calculations we
may assume that summation in \eqref{eq:999} is taken over the indices
 $\bfs,\bfs+\bfh$ whose norm is less than $n/2$.
\par
Applications of Lemma~\ref{lem:psih} yield
\beao
&&
\dfrac {r_n^2}{m_n}\sum_{3r_n<\|\bfs\|, \|\bfs+\bfh \|<\sqrt{n}\,,\|\bfh\|\le 2r_n}
|\psi_\bfs^{(n)}\psi_{\bfs+\bfh}^{(n)}|=O( (\log n)^4 r_n^2/m_n)=o(1)\,,\\
&&\dfrac {r_n^2}{m_n}\sum_{\sqrt{n}<\|\bfs\|, \|\bfs+\bfh \|<n/2\,,\|\bfh\|\le 2r_n}
|\psi_\bfs^{(n)}\psi_{\bfs+\bfh}^{(n)}|\le c \dfrac{r_n^2}{m_n} \dfrac 1{n^2} (nr_n)^2
=O(r_n^4/m_n)=o(1)\,.
\eeao
In the last step we used condition \eqref{eq:newlabel}. This proves the lemma
under 2.
\par
Now assume 3. Then
\beao
&&\big|\E
  \big[ \widehat{\Imath}_{\bfii} \widehat{\Imath}_{\bfj}
  \widehat{\Imath}_{\bfii+\bfs} \widehat{\Imath}_{\bfj+\bfs+\bfh} -\E[\wh I_{\bfii}\,\wh I_{\bfj+\bfs+\bfh}]\,
\E[\wh I_{\bfii+\bfs}\,\wh I_{\bfj}]\big] \big| +\big|\E[\wh I_{\bfii}\,\wh I_{\bfj+\bfs+\bfh}]\,
\E[\wh I_{\bfii+\bfs}\,\wh I_{\bfj}]\big]\big|\\
&\le &\alpha(r_n)+\alpha(\|\bfii-(\bfj+\bfs+\bfh)\|)\,\alpha(\|(\bfii+\bfs)-\bfj\|)\,.
\eeao
If  $\|(\bfii+\bfs)-\bfj\|>r_n$ or $\|\bfii-(\bfj+\bfs+\bfh)\|>r_n$ 
then the \rhs\ is $O(\a(r_n))$ and we can proceed as 
in cases 1. and 2. Now we assume that $\|(\bfii-\bfj)+\bfs\|\le r_n$ and 
$\|(\bfii-\bfj)-(\bfs+\bfh)\|\le r_n$. This means that $\bfii-\bfj $ is contained in the balls (\wrt\ $\|\cdot\|$) with radius $r_n$ and centers $-\bfs$ and 
$\bfs+\bfh$. However, this is impossible if $\bfs,\bfs+\bfh$ belong to the 
same quadrant. A glance at \eqref{eq:renwe} convinces one that
the right-hand expectation can be bounded by four expected values
containing sums of indices $\bfh$ from distinct quadrant. Since the previous
arguments work for each quadrant separately this finishes the proof.

\end{proof}

\end{document}